\documentclass[11pt, A4]{amsart}
\let\oldtocsection=\tocsection
\renewcommand{\tocsection}[2]{\hspace{0em}\oldtocsection{#1}{#2}}

\usepackage[a4paper]{geometry}

\usepackage{amsbsy}
\let\oldmarginpar\marginpar
\renewcommand\marginpar[1]{\-\oldmarginpar[\raggedleft\footnotesize #1]%
{\raggedright\tiny #1}}
\usepackage{amsmath,amssymb,amsthm,mathrsfs,color}
\usepackage[bbgreekl]{mathbbol}
\usepackage{bbm}

\usepackage{tikz-cd}

\usepackage{graphicx}
\usepackage{ifpdf}

\usepackage[small,nohug,heads=littlevee]{diagrams}
\usepackage[all]{xy}
\diagramstyle[labelstyle=\scriptstyle]
\usepackage[mathscr]{euscript}
\usepackage{marginnote}

\usepackage{stmaryrd}

\usepackage{hyperref}
\hypersetup{colorlinks,%
    citecolor=black,%
    filecolor=black,%
   linkcolor=black,%
   urlcolor=black}
   
\begin{document}
\setlength{\unitlength}{2.5cm}

%%%%%%%%%%% theorem styles
\newtheorem{thm}{Theorem}[section]
\newtheorem{lm}[thm]{Lemma}
\newtheorem{prop}[thm]{Proposition}
\newtheorem{cor}[thm]{Corollary}

\theoremstyle{definition}
\newtheorem{dfn}[thm]{Definition}
\newtheorem{eg}[thm]{Example}
\newtheorem{rmk}[thm]{Remark}

\newcommand{\TA}{\text{A}}
\newcommand{\TB}{\text{B}}
\newcommand{\TC}{\text{C}}
\newcommand{\TD}{\text{D}}
\newcommand{\TE}{\text{E}}
\newcommand{\TF}{\text{F}}
\newcommand{\TG}{\text{G}}
\newcommand{\ii}{\mathbf{i}}

\newcommand{\F}{\mathbf{F}}
\newcommand{\N}{\mathbf{N}}
\newcommand{\R}{\mathbf{R}}
\newcommand{\C}{\mathbf{C}}
\newcommand{\Z}{\mathbf{Z}}
\newcommand{\Q}{\mathbf{Q}}
\newcommand{\T}{\mathbf{T}}
\newcommand{\K}{\mathbf{K}}
\newcommand{\LL}{\mathcal{L}}

\newcommand{\G}{\text{G}}
\newcommand{\re}{\text{Re}}
\newcommand{\im}{\text{Im}}
\newcommand{\gal}{\text{Gal}}
\newcommand{\ke}{\text{Ker}}
\newcommand{\ma}{\text{Max}}
\newcommand{\Spec}{\text{Spec}}
\newcommand{\Pic}{\text{Pic}}
\newcommand{\ord}{\text{ord}}
\newcommand{\app}{\thickapprox}
\newcommand{\deh}{H_\mca{D}}
\newcommand{\moh}{H_\mca{M}}
\newcommand{\ab}{\text{ab}}
\newcommand{\Mp}{\text{Mp}}
\newcommand{\Sp}{\text{Sp}}
\newcommand{\GL}{\text{GL}}
\newcommand{\PGL}{\text{PGL}}
\newcommand{\SL}{\text{SL}}
\newcommand{\Spin}{\text{Spin}}
\newcommand{\Ind}{\text{Ind}}
\newcommand{\Res}{\text{Res}}
\newcommand{\vs}{\vec{s}}
\newcommand{\Hom}{\text{Hom}}
\newcommand{\msc}[1]{\mathscr{#1}}
\newcommand{\mfr}[1]{\mathfrak{#1}}
\newcommand{\mca}[1]{\mathcal{#1}}
\newcommand{\mbf}[1]{{\bf #1}}
\newcommand{\wchi}{\wt{\chi}}
\newcommand{\GGMA}{\pmb{\Gamma}}

\newcommand{\dual}[1]{{}^L\wt{#1}}
\newcommand{\cd}[1]{{\wt{#1}}^\vee}
\newcommand{\noa}{\mathbf{n}_o^\alpha}
\newcommand{\koa}{\mathbf{k}_o^\alpha}
\newcommand{\Hs}{\textsf{Hs}}
\newcommand{\Inc}{\textsf{Inc}}
\newcommand{\CExt}{\textsf{CExt}}
\newcommand{\Bis}{\textsf{Bis}}
\newcommand{\Rec}{\text{Rec}}
\newcommand{\s}{\mathbf{s}}
\newcommand{\cc}{\mathbf{c}}
\newcommand{\g}{\mathbf{g}_{\psi^{-1}}}
\newcommand{\w}{\mathbbm{w}}
\newcommand{\Ftn}{{\sf Ftn}}
\newcommand{\A}{\mbf{A}_F}
\newcommand{\p}{\mathbf{p}}
\newcommand{\q}{\mathbf{q}}
\newcommand{\WD}{\text{WD}}
\newcommand{\W}{\text{W}}
\newcommand{\Wh}{{\sf Wh}_\psi}
\newcommand{\ggma}{\pmb{\gamma}}
\newcommand{\sct}{\text{sc}}
\newcommand{\OF}{\mca{O}^\digamma}
\newcommand{\vep}{\pmb{\varepsilon}}

\newcommand{\cu}[1]{\textsc{\underline{#1}}}
\newcommand{\set}[1]{\left\{#1\right\}}
\newcommand{\ul}[1]{\underline{#1}}
\newcommand{\wt}[1]{\overline{#1}}
\newcommand{\angb}[2]{\left\langle #1, #2 \right\rangle}
\newcommand{\seq}[3]{\xymatrix{
#1 \ar@{^(->}[r] & #2 \ar@{>>}[r] & #3
}}
\newcommand{\wm}[1]{\wt{\mbf{#1}}}
\newcommand{\elt}[1]{\pmb{\big[} #1\pmb{\big]} }
\newcommand{\ceil}[1]{\left\lceil #1 \right\rceil}
\newcommand{\val}[1]{\left| #1 \right|}

\title[Distinguished theta representations for certain covering groups]{\textsc{Distinguished theta representations for certain covering groups}}
\author{Fan Gao}
\address{Department of Mathematics, Purdue University, 150 N. University Street, West Lafayette, IN 47907}
\email{gaofan.math@gmail.com}
%\date{}
\subjclass[2010]{Primary 11F70; Secondary 22E50}
\keywords{Brylinski-Deligne covering groups, theta representations, Whittaker functionals, distinguished characters, dual groups}
\maketitle
\centerline{\emph{To Professor Freydoon Shahidi on his 70th birthday}}

\begin{abstract}
For Brylinski-Deligne covering groups of an arbitrary split reductive group, we consider theta representations attached to certain exceptional genuine characters. The goal of the paper is to study the dimension of the space of Whittaker functionals of a theta representation. In particular, we investigate when the dimension is exactly one, in which case the theta representation is called distinguished. For this purpose, we first give effective lower and upper bounds for the dimension of Whittaker functionals for general theta representations. As a consequence, the dimension in many cases can be reduced to simple combinatorial computations, e.g., the Kazhdan-Patterson covering groups of the general linear groups, or covering groups whose complex dual groups (\`a la Finkelberg-Lysenko-McNamara-Reich) are of adjoint type. In the second part of the paper, we consider coverings of certain semisimple simply-connected groups and give necessary and sufficient condition for the theta representation to be distinguished. There are subtleties arising from the relation between the rank and the degree of the covering group. However, in each case we will determine the exceptional character such that its associated theta representation is distinguished.
\end{abstract}

%\tableofcontents

\section{Introduction and main results}
\subsection{Introduction}
Let $F$ be a non-archimedean local field of characteristic 0 and residue characteristic $p$. Let $\mbf{G}$ be a connected split reductive group over $F$, and let $G:=\mbf{G}(F)$ be its rational points. One of the central ingredients in the study of irreducible admissible representation of $G$ is the uniqueness of Whittaker functionals  (cf. \cite{Rod, Shal}). For instance, this uniqueness property is crucial in the Langlands-Shahidi theory of $L$-functions (\cite{Sha}) for the so-called generic representations of $G$, i.e., those with nontrivial Whittaker functionals.

For a natural number $n\ge 1$, we assume that $F^\times$ contains the full subgroup of the $n$-th roots of unity, which is then denoted by $\mu_n$. In this paper, we work with the Brylinski-Deligne $n$-fold covering groups $\wt{G}^{(n)}$ of $G$, see \S \ref{Sec:SF} for description on such covering groups. We may write $\wt{G}^{(n)}$ and $\wt{G}$ interchangeably if no confusion arises. For simplicity, the phrase \emph{covering groups} in this paper is used to refer to the Brylinski-Deligne covering groups.  For this purpose, it is noteworthy to mention that the Brylinski-Deligne framework is quite encompassing and contains almost all classically interesting covering groups (\cite{S}, \cite{Mo} and \cite{Ma}), in particular the Matsumoto covering groups of semisimple simply-connected groups in \cite{Mo} and the Kazhdan-Patterson covering groups $\wt{\GL}_r^{(n)}$ of $\GL_r$ in \cite{KP}. 

For covering groups, the uniqueness of Whittaker functionals for genuine representations of $\wt{G}^{(n)}$ holds rarely and one nontrivial example is the classical double cover $\wt{\Sp}_{2r}^{(2)}$ of the symplectic group $\Sp_{2r}$, see \cite{Szp2}. This uniqueness plays pivotal role in the work of Szpruch (cf. \cite{Szp1}-\cite{Szp4}) generalizing the method of Langlands-Shahidi to $\wt{\Sp}_{2r}^{(2)}$. Besides this special family of examples, the uniqueness of Whittaker functionals fails widely, and one almost never expects such uniform property  for all genuine representations of a general covering group. For example, it is well-known that certain theta representations for the Kazhdan-Patterson coverings $\wt{\GL}_r^{(n)}$ of $\GL_r$ could have high dimensional space of Whittaker functionals (cf. \cite{KP}). In fact, such theta representations show that the analogous standard module conjecture (which is a theorem for linear algebraic groups by \cite{CS}) does not hold for covering groups. 

The failure of the uniquness of Whittaker functionals for general genuine representations  of covering groups, however, has been the source of both obstacles and inspirations to some advancement of the representation theory of such groups. On the one hand, for instance, it is not a priori clear how to generalize the Langlands-Shahidi theory of $L$-functions to covering groups because of the non-uniqueness of Whittaker functionals for unramified principal series representations. Equivalently, it is essentially due to the fact that the analogous Casselman-Shalika formula for covering groups as in \cite{CO} and \cite{Mc2} is vector-valued, whereas for linear algebraic groups it is scalar-valued (cf. \cite{CS1}).

On the other hand, there are various streams of rich theories stemming from the non-existence or multi-dimensionality of Whittaker functionals. For instance, for genuine representations of covering groups without Whittaker functionals, one may consider semi-Whittaker functionals as in \cite{Tak} or degenerate Whittaker-functionals \cite{MW}, which interact fruitfully with the arithmetic and character theory of the representations.  Meanwhile, the theory of unipotent orbit as discussed in \cite{Gin} and \cite{FG1}-\cite{FG3} for instance  also rectify the situation in the absence of Whittaker functionals. In the latter case where multi-dimensionality holds, the theory of multiple Weyl Dirichlet series makes deep and fascinating connections between representation theory of covering groups, quantum physics and statistical mechanics etc, see \cite{BBF}, \cite{BFH} and \cite{BFG} for some of the ideas involved. In particular, the book \cite{BFG} contains several excellent expository articles on multiple Dirichlet series.

Nevertheless, in this paper we consider only the so-called theta representations $\Theta(\wt{G}^{(n)}, \wchi)$ which appear as the local representations for the residue of the Borel Eisenstein series (see Definition \ref{D:excep}). Moreover, we are mostly interested in determining when the space of Whittaker functionals  for $\Theta(\wt{G}^{(n)}, \wchi)$ has dimension one, in which case $\Theta(\wt{G}^{(n)}, \wchi)$ is called distinguished following Suzuki in \cite{Suz1}. Here $\wchi$ is an exceptional genuine character (see Definition \ref{D:excep}) of the center $Z(\wt{T})$ of the covering torus $\wt{T}\subseteq \wt{G}$.  The reason for considering this problem is two-fold.

First, $\Theta(\wt{G}^{(n)}, \wchi)$ is in certain sense the simplest family of genuine representations of a general covering group $\wt{G}^{(n)}$. Indeed, if $n=1$, then it follows from definition that $\Theta(\wt{G}^{(n)}, \wchi)$ could be the trivial representation of the linear group $\wt{G}=\wt{G}^{(1)}$, depending on a proper choice of the exceptional character $\wchi$. Therefore, for the genericity question regarding Whittaker functionals of genuine representations, it is reasonable to consider this family first. Moreover, theta representations for the Kazhdan-Patterson covering groups of $\GL_r$, to which we have just alluded, are already studied in depth in the seminal paper \cite{KP}. Despite the fact that the idea in \cite{KP} could be applicable for general covering groups, to the best of our knowledge, it seems that there is no systematic treatment on theta representations for general covering groups in literature. Perhaps this gap is caused by the tedious cocycle computation to be carried out by any potential author. However, the Brylinski-Deligne framework enables us to compute by invoking some neat structural fact of the covering groups of interest, and to handle only a minimized usage of cocycle on the torus. In brief, we wish to fill in the gap by generalizing the relevant work of Kazhdan-Patterson to Brylinski-Deligne covering groups.

Second, distinguished theta representations have important and emergingly wider applications. Theta representations are the representation-theoretic analogues of theta functions, one of the early applications of which was given by Riemann in his seminal paper to prove the functional equation of the Riemann zeta function.  In the language of modern theory of representations, theta representations for $\wt{\Sp}_{2r}^{(2)}$ gain deep applications in the Shimura correspondence (cf. \cite{Shi} \cite{Gel}). On the other hand, following the work of Kazhdan-Patterson, theta representations for $\wt{\GL}_r^{(n)}$ are also studied extensively in the work Bump-Hoffstein (cf. \cite{BH}) and Suzuki (\cite{Suz1}, \cite{Suz2}), to mention a few. In particular, these authors made some deep conjectures and also provided evidence for a generalized Shimura correspondence regarding $\wt{\GL}_r^{(n)}$, and the distinguishedness property is exploited to achieve the goals in their work.
Another significant direction of applications is the Rankin-Selberg integral representation for the symmetric square and cube $L$-functions (cf. \cite{BG}, \cite{BGH}, \cite{Tak} and \cite{Ka2}). Evidently, it should be mentioned that for distinguished theta representations, the theory of $L$-function could be developed as in the linear algebraic case, since the Casselman-Shalika formula is then scalar-valued. More recently, the work of E. Kaplan \cite{Ka1}-\cite{Ka3}, S. Friedberg and D. Ginzburg \cite{FG1}-\cite{FG3} also relies heavily on the local and global theta representations in their consideration of Fourier coefficient, Rankin-Selberg $L$-function and descent integral etc. Notably in their work, distinguishedness is responsible for proving that a global integral admits an Euler factoriazation into local factors. Besides these, the problem on global cuspidal theta representations is important and many problems are open (cf. \cite{FG1}, \cite{Suz1}). In any case, we believe that distinguished theta representations are objects of great interest and significance, and we hope that our paper could shed some light on the relevant questions.

\subsection{Main results}
We consider a Brylinski-Deligne $n$-fold covering group $\wt{G}^{(n)}$. Let $\wchi$ be an exceptional character for $\wt{G}^{(n)}$. Fix an unramified additive character $\psi$ of $F$ and consider the space $\Wh(\Theta(\wt{G}^{(n)}, {\wchi}))$ of $\psi$-Whittaker functionals of the theta representation $\Theta(\wt{G}^{(n)}, {\wchi})$. The pair $(\wt{G}^{(n)}, \wchi)$ such that $\dim \Wh(\Theta(\wt{G}^{(n)}, {\wchi}))=1$ is quite unique, and the goal is to investigate when $\Theta(\wt{G}^{(n)}, \wchi)$ is distinguished. We remark that for fixed $\wt{G}^{(n)}$, the set of  unramified exceptional characters $\wt{\chi}$ is a torsor over $Z(\wt{G}^\vee)$, the center of the complex dual group $\wt{G}^\vee$ of $\wt{G}$. For details on $\wt{G}^\vee$, we refer to \cite{FL}, \cite{Mc1}, \cite{Re} and \cite{We2}.
\vskip 5pt

We outline the structure of the paper and state the main results.

In \S 2, we recall the basic structural  facts on a Byrlinski-Deligne covering group $\wt{G}^{(n)}$ which will be crucial for our computations. In this paper, we consider exclusively unramified covering group $\wt{G}^{(n)}$ and unramified exceptional character $\wchi$. In \S 3, the space $\Wh(\Theta(\wt{G}^{(n)}), \wchi)$ is analyzed following the strategy in \cite{KP} closely. In particular, it relies crucially on the Shahidi local coefficient matrix $\left[\tau(\wchi, \w_\alpha, \gamma, \gamma')\right]_{\gamma, \gamma'}$ for covering groups. Note that $\left[\tau(\wchi, \w_\alpha, \gamma, \gamma')\right]_{\gamma, \gamma'}$ is also referred to as the scattering matrix in \cite{BBB} and transition matrix in \cite{CO}. Since the matrix is an analogue (and in fact the reciprocal) of Shahidi's local coefficient in the linear algebraic case (cf. \cite[Chapter 5]{Sha}), we call it the Shahidi local coefficient matrix in this paper. See also \cite{Bud} and \cite{Szp5}. In the unramified setting, the matrix is computed in \cite{Mc2}; it is also computed for ramified places  in \cite{GS}. 

The first main result is Theorem \ref{T:LUB} from \S 3:

\begin{thm} \label{T:MT1}
Let $\wt{G}^{(n)}$ be an arbitrary unramified Brylinski-Deligne covering group. Let  $\wchi$ be an unramified exceptional genuine character of $\wt{G}^{(n)}$ with associated theta representation $\Theta(\wt{G}^{(n)}, \wchi)$. Then,
$$\val{\wp_{Q,n}(\OF_{Q,n})} \le \dim \Wh(\Theta(\wt{G}^{(n)}, \wchi))  \le \val{\wp_{Q,n}(\OF_{Q,n, \sct})}.$$
\end{thm}

These two bounds are combinatorial quantities involving certain Weyl-action on lattices. The readers are referred to \S 2 for details. We highlight here some consequences from the above Theorem. 

Firstly, Theorem \ref{T:MT1} above recovers the results of Kazhdan-Patterson.
More precisely, for covering groups $\wt{\GL}_r^{(n)}$ studied in \cite{KP}, the authors  determine that $\dim \Wh(\Theta(\wt{\GL}_r^{(n)}, \wchi))=1$ if and only if
\begin{enumerate}
\item[1)] $n=r$ and  $\wt{\GL}_r^{(n)}$ is any Kazhdan-Patterson covering group, or 
\item[2)] $n=r+1$ and $\wt{\GL}_r^{(n)}$ belongs to a special type of degree $n$ Kazhdan-Patterson covering groups. 
\end{enumerate}
In fact, for any covering group $\wt{\GL}_r^{(n)}$ studied in \cite{KP}, one has $\OF_{Q,n}=\OF_{Q,n,\sct}$. Therefore $\dim \Wh(\Theta(\wt{\GL}_r^{(n)}, \wchi)) = \val{\wp_{Q,n}(\OF_{Q,n, \sct})}$. In particular, the dimension does not depend on the choice of the exceptional character $\wchi$ and can be computed effectively. For details, see Example \ref{E:KP}. 

In general, for cases where the two bounds in Theorem \ref{T:MT1} actually agree, the computation of the dimension is reduced to purely combinatorial problem, and thus amenable to a straightforward calculation. This includes the case where $Y_{Q,n}=Y_{Q,n}^{\sct}$, or equivalently $Z(\wt{G}^\vee)=1$. For example, odd degree coverings of simply-connected groups of type $\TB_r, \TC_r$ have this property. See \S 5, \S 6.

Secondly in contrast, when the two bounds in Theorem \ref{T:MT1} do not agree, $\dim \Wh(\Theta(\wt{G}, \wchi))$ becomes sensitive to the choice of the exceptional character $\wchi$. The second half of this paper is devoted to investigate this.  This phenomenon already occurs for the degree two metaplectic covering $\wt{\SL}_2^{(2)}$, see Example \ref{E:SL2}. In this case $\Theta(\wt{\SL}_2^{(2)}, \wchi)$ is the even Weil representation. Consider $\Theta(\wt{\SL}_2^{(2)}, \wchi_{\psi_a})$, where $\wchi_{\psi_a}$ is an exceptional character defined by using the twisted additive character $\psi_a, a\in F^\times$. It is well-known that $\dim \Wh(\Theta(\wt{\SL}_2^{(2)}, \wchi_{\psi_a}))\le 1$ and the equality holds if and only if $a\in (F^\times)^2$. Our analysis shows that similar phenomenon occurs for higher rank groups, see \S \ref{S:A-delic}, in particular Corollary \ref{C:AC01}.
\vskip 5pt

In any case, we summarize our results for certain coverings of simply-connected groups as follows. We write for instance $\wt{\TA}_r^{(n)}$ for the degree $n$ covering of the simply-connected group of type $\TA_r$ of rank $r$. Here the covering group arises from a quadratic form $Q$ on the coroot lattice $Y=Y^{\sct}$ such that $Q(\alpha^\vee)=1$ for any short coroot $\alpha^\vee$. The following theorem is an amalgam of Theorem \ref{T:A}, Theorem \ref{T:C}, Theorem \ref{T:B} and  Theorem \ref{T:G2}. Only for $\wt{\TA}_r^{(n)}$, we impose the condition $n\le r+2$ for technical reasons.

\begin{thm} \label{T:MT2}
Let $\wt{G}^{(n)}$ be an unramified Brylinski-Deligne degree $n$ covering  of a simply-connected semisimple group of type $\TA_r, \TB_r, \TC_r$ or $\TG_2$. If $\wt{G}^{(n)}=\wt{\TA}_r^{(n)}$, we further assume $n\le r+2$. Let $\wt{\chi}$ be an unramified exceptional character for $\wt{G}^{(n)}$. In each case for $\wt{G}^{(n)}$ below, if $\dim \Wh(\Theta(\wt{G}^{(n)}, \wchi)) =1 $, then necessarily the following relation between $r$ and $n$ holds:
$$\begin{cases}
\wt{\TA}_r^{(n)}, r\ge 1, n\le r+2: & n=r+2 \text{ or } r+1; \\
\wt{\TC}_r^{(n)}, r\ge 2: & n=4r-2, 4r, 4r+2 \text{ or } 2r+1; \\
\wt{\TB}_r^{(n)}, r\ge 3: & n=2r+1 \text{ or } 2r+2; \\
%\wt{\TD}_r^{(n)}, r\ge 3: & n=2r-1 \text{ or } 2r-2;\\
\wt{\TG}_2^{(n)}:            & n=7 \text{ or } 12.
\end{cases}
$$
Conversely, suppose that $r$ and $n$ satisfy the above relations; then for every case above except $\wt{\TC}_r^{(4r)}$, there exists a unique exceptional character $\wchi$ such that $\dim \Wh(\Theta(\wt{G}^{(n)}, \wchi)) =1$ for above $\wt{G}^{(n)}$.
\end{thm}

We actually determine the unique exceptional character specified in Theorem \ref{T:MT2}, see Theorems \ref{T:A}, \ref{T:C}, \ref{T:B} and \ref{T:G2}. In the $\wt{\TA}_r^{(r+1)}$ case, our result generalizes that for the even Weil representation of $\wt{\SL}_2^{(2)}$ mentioned above. As noted, the collection of unramified exceptional characters is a torsor over $Z(\wt{G}^\vee)$. Moreover, for covering groups of simply connected groups, the choice of $\psi$ actually gives a base point for this torsor. Thus, any exceptional character $\wchi$ gives rise to an element in $Z(\wt{G}^\vee)$, depending on the choice of $\psi$. That is, the explicit requirement given in those theorems could be viewed as determining the corresponding element in $Z(\wt{G}^\vee)$.

We note that for classical groups and similitude groups, an extensive study is included in \cite{FGS}. Our result from Thereom \ref{T:MT2} also agrees with the pertinent discussion in \cite{FG2} for symplectic groups. For example, the local statement for the second part of Conjecture 1 in \cite{FG2} follows from Proposition \ref{P:Sp-odd} in our paper. Moreover, the factoriazability property of the Whittaker function in \cite{FG2} for $\wt{\Sp}_{2n}^{(4n-2)}$ also agrees with our result for the $\wt{\TC}_r^{(n)}$ case in Thereom \ref{T:MT2}.

Finally, we remark that groups of type $\TD_r, \TE_6, \TE_7, \TE_8, \TF_4$ could be analyzed by the same procedure. In principle, Theorem \ref{T:MT1} coupled with analogous argument for Theorem \ref{T:MT2} enable one to determine completely $\dim \Wh(\Theta(\wt{G}^{(n)}, \wchi))$ for arbitrary $(\wt{G}^{(n)}, \wchi)$.

\vskip 10pt

\textbf{Acknowledgment.} I would like to thank Solomon Friedberg, David Goldberg and Freydoon Shahidi for interesting discussions on various topics. Thanks are also due to Wee Teck Gan for helpful comments on the paper. Meanwhile, I would like to thank Boston College for its support and hospitality during a visit in the Fall semester. Finally, I would like to thank the referee for very helpful and detailed comments which greatly improve the exposition of the paper.

%%%%%%%%%%%%%%%%%%%%%%%%%%%%%%
\vskip 15pt

\section{Basic set-ups}
\subsection{Structural facts on $\wt{G}$} \label{Sec:SF}
For ease of reading, we first recall some structural facts on $\wt{G}$. The main references are \cite{BD}, \cite{FL}, \cite{Re}, \cite{Mc1}-\cite{Mc2}, \cite{We2} and \cite{GG}. In this paper, we concentrate exclusively on unramified Brylinski-Deligne covering group $\wt{G}$ (to be explained below). We follow notations in \cite{GG}.

%We will use the language of incarnator $(D, \eta)$ in \cite{GG}.
%We recall briefly some structure facts of Brylinski-Deligne covering group from cite[GG].
Let $F$ be a nonarchimedean field of characteristic 0, with residual characteristic $p$. Fix a uniformiser $\varpi$ of $F$. Let $\mbf{G}$ be a split linear algebraic group over $F$ with maximal split torus $\mbf{T}$. Write $(X, \Phi, \Delta, Y, \Phi^\vee, \Delta^\vee)$ for the root data of $\mbf{G}$. Here $X$ (resp. $Y$) is the character lattice (resp. cocharacter lattice) for $(\mbf{G}, \mbf{T})$. Choose a set $\Delta\subseteq \Phi$ of simple roots from the set of roots $\Phi$, and $\Delta^\vee$ the corresponding simple coroots from $\Phi^\vee$. Let $\mbf{B}$ be  the Borel subgroup associated with $\Delta$. Write $Y^{\sct}\subseteq Y$ for the lattice generated by $\Phi^\vee$. 

Fix a Chevalley system of pinnings for $(\mbf{G}, \mbf{T}, \mbf{B})$. That is, fix an isomorphism $e_\alpha: \mbf{G}_\text{a} \to \mbf{U}_\alpha$ for each $\alpha \in \Phi$, where $\mbf{U}_\alpha \subseteq \mbf{G}$ is the root subgroup  associated with $\alpha$. Moreover, for each $\alpha\in \Phi$, there is a unique morphism $\varphi_\alpha: \SL_2 \to \mbf{G}$ which restricts to $e_{\pm \alpha}$ on the upper and lower triangular subgroup of unipotent matrices of $\SL_2$.

Consider the algebro-geometric covering $\wm{G}$ of $\mbf{G}$ by $\mbf{K}_2$, which is categorically equivalent to the pairs $\set{(D, \eta)}$ (cf. \cite{GG}). Here $\eta: Y^{\sct} \to F^\times$ is a homomorphism. On the other hand, $D$ is a bisector associated to a Weyl-invariant quadratic form $Q: Y\to \Z$. That is, let $B_Q$ be the Weyl-invariant bilinear form associated to $Q$ such that $B_Q(y_1, y_2)=Q(y_1+y_2)-Q(y_1) -Q(y_2)$, then $D$ is a bilinear form on $Y$ satisfying
$$D(y_1, y_2) + D(y_2, y_1)=B_Q(y_1, y_2).$$
The bisector $D$ is not necessarily symmetric. 
%A bisector $D$ is said to be \emph{fair} if 
%\begin{enumerate}
%\item[$\bullet$] for any $\alpha^\vee \in \Delta^\vee$ such $2| Q(\alpha^\vee)$, one has $2|D(y, \alpha^\vee)$ and $2|D(\alpha^\vee, y)$ for all $y\in Y$.
%\end{enumerate}
Any $\wm{G}$ is, up to isomorphism, incarnated by (i.e. categorically associated to) $(D,\eta)$ for a bisector $D$ and some $\eta$. 
%Thus, there is no loss of generalities to impose the fairness condition, and it will be explicitly mentioned whenever we make such assumption on $D$.
 
Let $n\ge 1$ be a natural number. Assume that $F^\times$ contains the full group $\mu_n$ of $n$-th roots of unity and $p\nmid n$. Let $\wm{G}$ be incarnated by $(D, \eta)$. One obtains naturally degree $n$ topological covering groups $\wt{G}, \wt{T}, \wt{B}$ of the rational points $G:=\mbf{G}(F), T:=\mbf{T}(F), B:=\mbf{B}(F)$, e.g.,
%$$\begin{tikzcd}
%\mu_n \ar[r, hook] & \wt{G} \ar[r, two heads] & G.
%\end{tikzcd}$$
$$\xymatrix{
\mu_n \ar@{^(->}[r] & \wt{G} \ar@{>>}[r]  & G.
}$$
We may write $\wt{G}^{(n)}$ for $\wt{G}$ to emphasize the degree of covering. For any set $H\subseteq G$, we write $\wt{H}\subseteq \wt{G}$ for the preimage of $H$ with respect to the quotient map $\wt{G} \to G$. The Bruhat-Tits theory gives a maximal compact subgroup $K\subseteq G$, which depends on the fixed pinnings. We assume that $\wt{G}$ splits over $K$ and fix such a splitting; call $\wt{G}$ an unramified Brylinski-Deligne covering group in this case. We remark that if the derived group of $\mbf{G}$ is simply-connected, then $\wt{G}$ splits over $K$ (cf. \cite[Theorem 4.2]{GG}). On the other hand, we refer the reader to \cite[\S 4.6]{GG} for a counterexample from a certain double cover of $\text{PGL}_2$ where the splitting does not exist.

The data $(D, \eta)$ play the following role for the structural fact on $\wt{G}$.

\begin{itemize}
\item[$\bullet$] The group $\wt{G}$ splits canonically over any unipotent element of $G$. In particular, we write $\wt{e}_\alpha(u) \in \wt{G}, \alpha\in \Phi, u\in F$ for the canonical lifting of $e_\alpha(u) \in G$. For any $\alpha\in \Phi$, there is a natural representative $w_\alpha:= e_\alpha(1) e_{-\alpha}(-1) e_\alpha(1) \in K$ (and therefore $\wt{w}_\alpha\in \wt{G}$ by the splitting of $K$) of the Weyl element $\w_\alpha\in W$. Moreover, for $h_\alpha(a):=\alpha^\vee(a)\in G, \alpha\in \Phi, a\in F^\times$, there is a natural lifting $\wt{h}_\alpha(a) \in \wt{G}$ of $h_\alpha(a)$, which depends only on the pinning and the canonical unipotent splitting. For details, see \cite{GG}.
\item[$\bullet$] There is a section $\s$ of $\wt{T}$ over $T$ such that the group law on $\wt{T}$ is given by
\begin{equation} \label{F:s}
\s(y_1(a)) \cdot \s(y_2(b)) = (a, b)_n^{D(y_1, y_2)} \cdot \s(y_1(a)\cdot y_2(b)).
\end{equation}
Moreover, for the natural lifting $\wt{h}_\alpha(a)$ for $\alpha \in \Delta$, one has
\begin{equation} \label{F:h-s}
\wt{h}_\alpha(a)=(\eta(\alpha^\vee), a)_n \cdot \s(h_\alpha(a)) \in \wt{T}.
\end{equation}
\item[$\bullet$] Let $w_\alpha \in G$ be the natural representative of $\w_\alpha\in W$. For any $\wt{y(a)} \in \wt{T}$, one has
\begin{equation} \label{F:W-act}
w_\alpha \cdot \wt{y(a)} \cdot w_\alpha^{-1} = \wt{y(a)} \cdot \wt{h}_\alpha(a^{-\angb{y}{\alpha}}),
\end{equation}
where $\angb{-}{-}$ is the paring between $Y$ and $X$.
\end{itemize}

Consider the sublattice $Y_{Q,n}:=\set{y\in Y: B_Q(y, y')\in n\Z}$ of $Y$. For every $\alpha^\vee\in \Phi^\vee$, define $n_\alpha:= n/\text{gcd}(n, Q(\alpha^\vee))$. Write $\alpha_{Q,n}^\vee:=n_\alpha \alpha^\vee$ and $\alpha_{Q,n}:=n_\alpha^{-1} \alpha$. Let $Y_{Q,n}^{\sct} \subseteq Y$ be the sublattice generated by $\set{\alpha_{Q,n}^\vee}_{\alpha\in \Phi}$. The complex dual group $\wt{G}^\vee$ for $\wt{G}$ as given in \cite{FL},  \cite{Mc1} and \cite{Re} has root data $(Y_{Q,n}, \set{\alpha_{Q,n}^\vee}, \text{Hom}(Y_{Q,n}, \Z), \set{\alpha_{Q,n}})$. In particular, $Y_{Q,n}^{\sct}$ is the root lattice for $\wt{G}^\vee$. What is most pertinent to our paper is that the center $Z(\wt{G}^\vee)$ could be identified as
$$Z(\wt{G}^\vee):=\Hom(Y_{Q,n}/Y_{Q,n}^{\sct}, \C^\times).$$

%%%%%%%%%%%
\subsection{Theta representations $\Theta(\wt{G}, \wchi)$} \label{Sec:Theta}
Fix an embedding $\iota: \mu_n \hookrightarrow \C^\times$. A representation of $\wt{G}$ is called $\iota$-genuine if $\mu_n$ acts via $\iota$. We consider throughout the paper $\iota$-genuine (or simply genuine) representations of $\wt{G}$.

Let $U$ be the unipotent subgroup of $B=TU$. As $U$ splits canonically in $\wt{G}$, we have $\wt{B}=\wt{T}U$. The covering torus $\wt{T}$ is a Heisenberg group with center $Z(\wt{T})$. The image of $Z(\wt{T})$ in $T$ is equal to the image of the isogeny $Y_{Q,n}\otimes F^\times \to T$ induced from $Y_{Q,n} \to Y$.

Let $\wchi \in \Hom_\iota(Z(\wt{T}), \C^\times)$ be a genuine character of $Z(\wt{T})$, write $i(\wchi):=\text{Ind}_{A}^{\wt{T}} \wchi'$ for the induced representation on $\wt{T}$, where $A$ is any maximal abelian subgroup of $\wt{T}$, and $\wchi'$ is any extension of $\wchi$. By the Stone von-Neumann theorem (cf. \cite[Theorem 3.1]{We1} and \cite[Theorem 3]{Mc1}), the construction $\wchi \mapsto i(\wchi)$ gives a bijection between isomorphism classes of genuine representations of $Z(\wt{T})$ and $\wt{T}$. Since we consider an unramified covering group $\wt{G}$ in this paper, we take $\wt{A}$ to be $Z(\wt{T})\cdot (K\cap T)$ from now.

View $i(\wchi)$ as a genuine representation of $\wt{B}$ by inflation from the quotient map $\wt{B} \to \wt{T}$. Write $I(i(\wchi)):=\text{Ind}_{\wt{B}}^{\wt{G}}\ i(\wchi)$ for the normalized induced principal series representation of $\wt{G}$. For simplicity, we may also write $I(\wchi)$ for $I(i(\wchi))$. One knows that $I(\wchi)$ is unramified (i.e. $I(\wchi)^K\ne 0$) if and only if $\wchi$ is unramified, i.e., $\wchi$ is trivial on $Z(\wt{T})\cap K$. We consider in this paper only unramified genuine representations (and characters). In fact, one has the naturally arising abelian extension
%\begin{equation} \label{Ext1}
%\begin{tikzcd}
%\mu_n \ar[r, hook] & \wt{Y}_{Q,n} \ar[r, two heads] & Y_{Q,n}
%\end{tikzcd}
%\end{equation}
\begin{equation} \label{Ext1}
\xymatrix{
\mu_n \ar@{^(->}[r] & \wt{Y}_{Q,n} \ar@{>>}[r] & Y_{Q,n}
}
\end{equation}
such that unramified genuine characters of $\wchi$ of $Z(\wt{T})$ correspond to genuine characters of $\wt{Y}_{Q,n}$. Here $\wt{Y}_{Q,n}:=Z(\wt{T})/Z(\wt{T})\cap K$. Since $\wt{A}/(T\cap K)\simeq \wt{Y}_{Q,n}$ as well, there is a canonical extension (also denoted by $\wchi$) of an unramified character $\wchi$ of $Z(\wt{T})$ to $\wt{A}$, by composing $\wchi$ with $\wt{A} \twoheadrightarrow \wt{Y}_{Q,n}$. Therefore, we will identify $i(\wchi)$ as $\text{Ind}_{\wt{A}}^{\wt{T}}\ \wchi$ for this  $\wchi$.

For any $\w\in W$, the intertwining operator $T_{\w, \chi}: I(\wchi) \to I({}^{\w}\wchi)$ is defined by
$$(T_{\w, \wchi} f)(\wt{g})=\int_{U_{w}} f(w^{-1} u \wt{g}) du$$
whenever it is absolutely convergent. Moreover, it can be meromorphically continued for all $\wt{\chi}$ (cf. \cite[\S 7]{Mc1}). For $I(\wchi)$ unramified and $\w=\w_\alpha$ with $\alpha\in \Delta$, $T_{\w_\alpha, \chi}$ is determined by
$$T_{\w_\alpha, \chi} (f_0) = c(\w_\alpha, \wchi) \cdot f_0' \text{ with } c(\w_\alpha, \wchi)=\frac{1-q^{-1}\wchi(\wt{h}_\alpha(\varpi^{n_\alpha}))}{1-\wchi(\wt{h}_\alpha(\varpi^{n_\alpha})}, $$
where $f_0\in I(\wchi)$ and $f_0' \in I({}^{\w_\alpha} \wchi)$ are the unramified vectors. Moreover, $T_{\w, \wchi}$ satisfies the cocycle condition as in the linear case.  The coefficient $c(\w_\alpha, \wchi)$ was determined in \cite[Theorem 12.1]{Mc2} and  later reformulated in \cite{Gao}. We use the latter formalism which is more suitable for our needs in this paper.

\vskip 5pt

The following definition mimics that in \cite[\S I.2]{KP}.

\begin{dfn} \label{D:excep}
An unramified genuine character $\wchi$ of $Z(\wt{T})$ is called exceptional if $\wchi(\wt{h}_\alpha(\varpi^{n_\alpha}))=q^{-1}$ for all $\alpha\in \Delta$. The theta representation $\Theta(\wt{G}, \wchi)$ associated to an exceptional character $\wchi$ is the unique Langlands quotient (cf. \cite{BJ}) of $I(\wchi)$, which is also equal to the image of the intertwining operator $T_{\w_0, \wchi}: I(\wchi) \to I({}^{\w_0}\wchi)$, where $\w_0\in W$ is the longest Weyl element.
\end{dfn}

The extension $\wt{Y}_{Q,n}$ gives rise to an extension $\wt{Y}_{Q,n}^{\sct}$ of $Y_{Q,n}^{\sct}$ by restriction. All exceptional characters agree on $\wt{Y}_{Q,n}^{\sct}$, and therefore the set of exceptional characters is a torsor over $Z(\wt{G}^\vee)$.

%%%%%%%%%%%
\subsection{Unitary distinguished characters} \label{S:UDC}
Depending on a choice of  a nontrivial additive character $\psi'$ of $F$, a special class of the so-called distinguished genuine characters of $Z(\wt{T})$ is singled out in \cite{GG} for the consideration of the $L$-group extension for $\wt{G}$. Distinguished characters, in the sense of \cite{GG}, may not exist for general Brylinski-Deligne covering groups. However, if $\mbf{G}$ has simply-connected derived group or if the composition $\eta: Y^{\sct} \to F^\times \to F/(F^\times)^n$ is trivial, such characters exist. One special property of a distinguished character is its Weyl-invariance, and thus it could serve as a \emph{distinguished} base point in the set of genuine characters of $Z(\wt{T})$.

For the purpose of \S4-\S7, we recall the explicit construction in \cite{GG} when a distinguished character exists. In particular, we make the above assumption on $\wt{G}$, which is clearly satisfied in the simply-connected case in \S4-\S7. 

First, let $\set{y_i}$ be a basis of $Y_{Q,n}$ such that $\set{k_i y_i}$ is a basis for the lattice $J=nY+Y_{Q,n}^{\sct}$ for some $k_i\in \Z$. Let $\psi'$ be a nontrivial additive character of $F$. Let $\ggma_{\psi'}$ be the Weil index valued in $\mu_4$ satisfying
$$\ggma_{\psi'}(b^2)=1, \ \ggma_{\psi'}(b)^2=(b, b)_2, \ \ggma_{\psi'}(bc)=\ggma_{\psi'}(b) \ggma_{\psi'}(c) \cdot (b, c)_2.$$
For any $a\in F^\times$, let $\psi_a': x\mapsto \psi'(ax)$ be the twisted additive character. Then
$$\ggma_{\psi_a'}(b)=\ggma_{\psi'}(b) \cdot (a, b)_2.$$
By definition, a unitary distinguished character $\wchi^0_{\psi'}$ of $Z(\wt{T})$ is given by
$$\wchi_{\psi'}^0(y_i(a))=\ggma_{\psi'}(a)^{\frac{2(k_i-1)Q(y_i)}{n}},$$
and for $y=\sum_i n_i y_i$ and $a\in F^\times$,
\begin{equation} \label{UDC}
\wchi_{\psi'}^0(y(a))=(a, a)_n^{\sum_{i<j} n_i n_j D(y_i, y_j)} \cdot \prod_i \wchi_{\psi'}^0(y_i(a^{n_i}))^{\frac{2(k_i-1)Q(y_i)}{n}}.
\end{equation}
Note that in \cite{GG}, the exponent of $\ggma_{\psi'}(a)$ in the formula of $\wchi_{\psi'}^0(y_i(a))$ is the negative of what we use here. However, both give rise to distinguished characters.

\subsection{Conventions and notations}

Let $2\rho:=\sum_{\alpha^\vee >0} \alpha^\vee$ be the sum of all positive coroots of $\mbf{G}$. Consider the affine translation $\ell_\rho: Y\otimes \Q \to Y\otimes \Q$ given by $y\mapsto y-\rho$. Write $\w(y)$ for the natural Weyl group action on $Y$ and $Y\otimes \Q$. Endow the codomain of $\ell_\rho$ with this action. By transport of structure, one has an induced action of $W$ on the domain of $\ell_\rho$ (i.e. the first $Y\otimes \Q$), which we denote by $\w[y]$. That is,
$$\w[y]:=\w(y-\rho)+ \rho.$$
Clearly $Y$ is stable under this action. Write $y_\rho:=y-\rho$ for any $y\in Y$, then $\w[y]-y=\w(y_\rho) - y_\rho$. From now, by Weyl orbits in $Y$ or $Y\otimes \Q$  we always refer to the ones with respect to the action $\w[y]$. Write $\mca{O}$ (respectively $\mca{O}^\digamma$) for the set of $W$-orbits (resp. free $W$-orbits) in $Y$.

We remark that for $\mbf{GL}_r$, the Weyl-action considered by Kazhdan-Patterson (see \cite[page 78]{KP}) is actually $\w(y+\rho) -\rho$. However, as the indexing of Whittaker functionals also differs from ours  by taking an ``inverse", thus our terminology is different but equivalent to that of \cite{KP}.

\begin{dfn}
For any subgroup $\Lambda \subseteq Y$, a free orbit $\mca{O}_y \in \mca{O}^\digamma$ is called $\Lambda$-free if the quotient map $Y \to Y/\Lambda$ is injective on $\mca{O}_y$. We write $\OF_\Lambda \subseteq \OF$ for the set of $\Lambda$-free orbits of $Y$.
\end{dfn}

Note that $\Lambda$-free orbits are assumed to be free by definition. For simplicity, we will write $\mca{O}_{Q, n, \sct}^\digamma$ and $\mca{O}_{Q,n}^\digamma$ for the set of $Y_{Q,n}^\sct$ and $Y_{Q,n}$-free orbits of $Y$ respectively. Clearly, the inclusions $\mca{O} \supseteq \OF \supseteq \OF_{Q, n, \sct} \supseteq \OF_{Q,n}$ hold.
\vskip 5pt

Generally, notations will be either self-explanatory or explained the first time they occur. For convenience, we list some notations which appear frequently in the text:
\vskip 5pt

$\vep:$ the element $\iota\big( (-1, \varpi)_n\big) \in \C^\times$. In particular, for $n$ odd, $\vep=1$. We use freely in the paper the following identity:
$$\vep^{D(y, y')}=\vep^{D(y', y)} \text{ for any } y\in Y_{Q,n}, y'\in Y.$$

$\wp_{Q, n}$:  the projection $Y \to Y/Y_{Q,n}$.

$\wp^{\sct}_{Q,n}$:  the projection $Y \to Y/Y_{Q,n}^{\sct}$.

$\psi$: a fixed additive character of $F$ into $\C^\times$ with conductor $O_F$. For any $a\in F^\times$, the twisted character $\psi_a$ is given by $\psi_a: x\mapsto \psi(ax)$.

$\s_y$: for any $y\in Y$, we write $\s_y:=\s(\varpi^y)\in \wt{T}$.

$\ceil{x}$: the minimum integer such that $\ceil{x}\ge x$ for a real number $x$.
%%%%%%%%%%%%%%%%%%%%%%%%%%%%%%%%%%%%%%%%%%%%%%%%%%%%%%%%%%%%%%%
\vskip 15pt

\section{Bounds for $\dim \Wh(\Theta(\wt{G}, \wchi))$}

\subsection{Whittaker functionals}
We follow the notations in \S 2.2. In particular, consider the principal series $I(\wt{\chi}):=I(i(\wchi))$ for an unramified character $\wchi \in \text{Hom}_\iota (Z(\wt{T}), \C^\times)$. 

Let $\Ftn(i(\wchi))$ be the vector space of  functions $\cc$ on $\wt{T}$  satisfying
$$\cc(\wt{t} \cdot \wt{z}) =  \cc(\wt{t}) \cdot \wchi(\wt{z}), \quad \wt{t} \in \wt{T} \text{ and } \wt{z} \in \wt{A}.$$
The support of any $\cc \in \Ftn(i(\wchi))$ is a disjoint union of cosets in $\wt{T}/\wt{A}$. Moreover, $\dim \Ftn(i(\wchi))=\val{Y/Y_{Q,n}}$ since $\wt{T}/\wt{A}$ has the same size as $Y/Y_{Q,n}$. 

There is a natural isomorphism of vector spaces $\Ftn(i(\wchi)) \simeq i(\wchi)^\vee$, where $i(\wchi)^\vee$ is the complex dual space of functionals of $i(\wchi)$. More explicitly, let $\set{\gamma_i}\subseteq \wt{T}$ be a chosen set of representatives of $\wt{T}/\wt{A}$, consider $\cc_{\gamma_i} \in \Ftn(i(\wchi))$ which has support $\gamma_i \cdot \wt{A}$ and $\cc_{\gamma_i}(\gamma_i)=1$. It gives rise to a linear functional $\lambda_{\gamma_i}^{\wchi} \in i(\wchi)^\vee$ such that $\lambda_{\gamma_i}^{\wchi}(f_{\gamma_j})=\delta_{ij}$, where $f_{\gamma_j}\in i(\wchi)$ is the unique element such that $\text{supp}(f_{\gamma_j})=\wt{A}\cdot \gamma_j^{-1}$ and $f_{\gamma_j}(\gamma_j^{-1})=1$. That is, $f_{\gamma_j}=i(\wt{\chi})(\gamma_j)\phi_0$, where $\phi_0\in i(\wchi)$ is the normalized unramified vector of $i(\wchi)$ such that $\phi_0(1_{\wt{T}})=1$.  Thus, the isomorphism $\Ftn(i(\wchi)) \simeq i(\wchi)^\vee$ is given explicitly by
$$ \cc   \mapsto   \lambda_\cc^{\wchi}:= \sum_{\gamma_i \in \wt{T}/\wt{A}} \cc(\gamma_i) \lambda_{\gamma_i}^{\wchi}.
$$
It can be checked easily that the isomorphism does not depend on the choice of representatives for $\wt{T}/\wt{A}$.

%For $\w_\alpha\in W$, one has $\Ftn(i(\wchi))=\Ftn({}^{w_\alpha} i(\wchi))$. On the other hand, ${}^{w_\alpha} i(\wchi) \simeq i({}^{\w_\alpha} \wchi)$ canonically by the Stone von-Neumann theorem. This gives a canonical vector space isomorphism $\Ftn({}^{w_\alpha}i(\wchi))\simeq \Ftn(i({}^{\w_\alpha}\wchi))$ and also
%\begin{equation} \label{vec iso}
%\Ftn(i(\wchi)) \simeq \Ftn(i({}^{\w_\alpha}\wchi)).
%\end{equation}

Let $\psi_U: U\to \C^\times$ be the character on $U$ such that its restriction to every $U_\alpha, \alpha\in \Delta$ is given by $\psi \circ e_\alpha^{-1}$. We may write $\psi$ for $\psi_U$ if no confusion arises.

\begin{dfn}
For any genuine representation $(\wt{\sigma}, V_{\wt{\sigma}})$ of $\wt{G}$, a linear functional $\ell: V_{\wt{\sigma}} \to \C$ is called a $\psi$-Whittaker functional if $\ell(\wt{\sigma}(u) v)= \psi(u) \cdot v$ for all $u\in U$ and $v\in V_{\wt{\sigma}}$.
Write $\Wh(\wt{\sigma})$ for the space of $\psi$-Whittaker functionals for $\wt{\sigma}$.
\end{dfn}

There is an isomorphism between $i(\wchi)^\vee$ and the space $\Wh(I(\wchi))$ of $\psi$-Whittaker functionals  on $I(\wt{\chi})$ (cf. \cite[\S 6]{Mc2}), given by $\lambda \mapsto W_\lambda$ with
$$W_\lambda:  I(\wt{\chi}) \to \C, \quad f \mapsto \lambda \left( \int_{U} f(w_0^{-1}u) \psi(u)^{-1} \mu(u) \right),$$
where $f\in I(\wt{\chi})$ is an $i(\wt{\chi})$-valued function on $\wt{G}$. Here $w_0=w_{\alpha_1} w_{\alpha_2} ... w_{\alpha_k}\in K$ is a representative of $\w_0$, where $\w_0=\w_{\alpha_1} \w_{\alpha_2} ... \w_{\alpha_k}$ is a minimum decomposition of $\w_0$. For any $\cc\in \Ftn(i(\wchi))$, by abuse of notation, we will write $\lambda_\cc^{\wchi} \in \Wh(I(\wchi))$ for the resulting $\psi$-Whittaker functional of $I(\wt{\chi})$ from the isomorphism $\Ftn(i(\wchi))\simeq i(\wchi)^\vee \simeq \Wh(I(\wchi))$. An easy consequence is
$$\dim \Wh(I(\wchi)) = \val{Y/Y_{Q,n}}.$$

\vskip 5pt
Let $J(\w, \wchi)$ be the image of $T_{\w, \wchi}$. The operator $T_{\w, \wchi}$ induces a homomorphism $T_{\w, \wchi}^*$ of vectors spaces with image $\Wh(J(\w, \wt{\chi}))$:
%$$\begin{tikzcd}
%T_{\w, \wchi}^*: \Wh(I({}^{\w}\wchi))\ar[r]  \ar[dr, two heads] & \Wh(I(\wchi)) \\
% &  \Wh(J(\w, \wt{\chi})) \ar[u, hook],
%\end{tikzcd}$$
$$\xymatrix{
T_{\w, \wchi}^*: \Wh(I({}^{\w}\wchi)) \ar[r]  \ar@{>>}[rd] & \Wh(I(\wchi)) \\
 &  \Wh(J(\w, \wt{\chi})), \ar@{^(->}[u]
}$$
which is given by 
$\angb{\lambda_\cc^{{}^{\w}\wchi} }{-} \mapsto \angb{\lambda_\cc^{{}^{\w}\wchi} }{T_{\w,\wchi}(-)}$ for any $\cc \in \Ftn(i({}^{\w}\wchi))$. Let $\set{\lambda_{\gamma}^{^{\w}\wchi}}_{\gamma \in \wt{T}/\wt{A}}$ be a basis for  $\Wh(I({}^{\w}\wchi))$, and $\set{ \lambda_{\gamma'}^{\wchi} }$ a basis for $\Wh(I(\wchi))$. The map $T_{\w,\wchi}^*$ is
then determined by the square matrix
$[\tau(\wchi, \w, \gamma, \gamma')]_{\gamma, \gamma'\in \wt{T}/\wt{A}}$ of size $\val{Y/Y_{Q,n}}$
such that
$$T_{\w, \wchi}^*(\lambda_{\gamma}^{^{\w}\wchi}) = \sum_{\gamma'\in \wt{T}/\wt{A}} \tau(\wchi, \w, \gamma, \gamma') \cdot \lambda_{\gamma'}^{\wchi}.$$

Some immediate properties are as follows.
\begin{lm} \label{L:CBF}
For $\w\in W$ and $\wt{z}, \wt{z}'\in \wt{A}$, the identity holds:
$$\tau(\wchi, \w, \gamma \cdot \wt{z}, \gamma' \cdot \wt{z}')=({}^{\w}\wchi)^{-1}(\wt{z}) \cdot \tau(\wchi, \w, \gamma, \gamma') \cdot \wchi(\wt{z}').$$
Moreover, for $\w_1, \w_2 \in W$ such that $l(\w_2\w_1)=l(\w_2) + l(\w_1)$, one has
$$\tau(\wchi, \w_2\w_1, \gamma, \gamma')=\sum_{\gamma''\in \wt{T}/\wt{A}} \tau({}^{\w_1}\wchi, \w_2, \gamma, \gamma'') \cdot \tau(\wchi, \w_1, \gamma'', \gamma'),$$
which is referred to as the cocycle relation.
\end{lm}
\begin{proof} The first equality follows from a change of basis formula from a different choice of representations for $\wt{T}/ \wt{A}$. The second equality follows from the cocycle relation of intertwining operators.
\end{proof}
%%%%%%
\vskip 5pt

\subsection{Reduction of $\Wh(\Theta(\wt{G}, \wchi))$}

Let $\w_0$ be the longest Weyl element of $\mbf{G}$. Consider the theta representation $\Theta(\wt{G}, \wchi)=T_{\w_0, \wchi}(I(\wchi))$ attached to an unramified exceptional character $\wchi$ (see Definition \ref{D:excep}). 

\begin{dfn} \label{D:DTR}
A theta representation $\Theta(\wt{G}, \wchi)$ attached to an unramified exceptional  genuine character $\wchi$ is called distinguished if $\dim \Wh(\Theta(\wt{G}, \wchi))=1$.
\end{dfn}

The distinguishedness of a theta representation here is not to be confused with that in distinguished genuine character as given in \S \ref{S:UDC}.

%\begin{rmk} Our definition of distinguishedness is less-refined than the usual notion in literature (see for example \cite{Li}), since we fix our Whittaker datum $(B=TU, \psi_U)$. For the consideration of Fourier expansion of a cusp form, the Rankin-Selberg integral etc, one would like to impose the stronger condition that $\dim \Wh(\Theta(\wt{G}, \wchi))=1$ for a unique Whittaker datum. However, as we prove broader results, the desired distinguished $\Theta(\wt{G}, \wchi)$ can be inferred from our Theorems \ref{T:MT1} and \ref{T:MT2}.
%\end{rmk}

\begin{prop} \label{P:inter}
Let $\wchi$ be an unramified exceptional character of $\wt{G}$, and $\Delta$ the set of simple roots. Then
$$\Wh(\Theta(\wt{G}, \wchi)) =\bigcap_{\alpha\in \Delta} \text{Ker}\big( T^*_{\w_\alpha, {}^{\w_\alpha}\wchi}: \Wh(I(\wchi)) \to \Wh(I(^{\w_\alpha}\wchi)) \big),$$
where $T_{\w_\alpha, {}^{\w_\alpha}\wchi}$ is the intertwining operator from  $I(^{\w_\alpha}\wchi)$ to $I(\wchi)$.
\end{prop}
\begin{proof}
The same proof for \cite[Theorem I.2.9]{KP} applies here mutatis mutantis.
\end{proof}

Let $\lambda_{\gamma}^{\wchi} \in \Wh(I(\wchi))$ and $\alpha\in \Delta$, then
$$T^*_{\w_\alpha, {}^{\w_\alpha}\wchi}(\lambda_{\gamma}^{\wchi}) = \sum_{\gamma'} \tau({}^{\w_\alpha}\wchi, \w_\alpha, \gamma, \gamma') \cdot \lambda_{\gamma'}^{{}^{\w_\alpha}\wchi}.$$
In general, let $\cc\in \Ftn(i(\wchi))$, and write
$$\lambda_\cc^{\wchi} = \sum_{\gamma \in \wt{T}/\wt{A}} \cc(\gamma) \lambda_{\gamma}^{\wchi} .$$
Then,
\begin{align*}
T^*_{\w_\alpha, {}^{\w_\alpha}\wchi}(\lambda_\cc^{\wchi}) & = \sum_{\gamma} \cc(\gamma) \left( \sum_{\gamma'} \tau({}^{\w_\alpha}\wchi, \w_\alpha, \gamma, \gamma') \cdot \lambda_{\gamma'}^{{}^{\w_\alpha}\wchi} \right) \\
& = \sum_{\gamma'} \left( \sum_{\gamma} \cc(\gamma) \tau({}^{\w_\alpha}\wchi, \w_\alpha, \gamma, \gamma') \right) \lambda_{\gamma'}^{{}^{\w_\alpha}\wchi} .
\end{align*}

As an immediate consequence of Proposition \ref{P:inter}, one has (see also \cite[page 76]{KP}):
\begin{cor} \label{C:iff-1}
A function $\cc \in \Ftn(i(\wchi))$ gives rise to a functional in $\Wh(\Theta(\wt{G}, \wchi))$ (i.e. $\lambda_\cc^{\wchi} \in \Wh(\Theta(\wt{G}, \wchi))$) if and only if
for all $\alpha\in \Delta$,
$$\sum_{\gamma \in \wt{T}/\wt{A}} \cc(\gamma) \tau({}^{\w_\alpha}\wchi, \w_\alpha, \gamma, \gamma')=0 \text{ for all } \gamma'.$$
The left hand side is independent of the choice of representatives for $\wt{T}/\wt{A}$ by Lemma \ref{L:CBF}.
\end{cor}

\vskip 5pt
\subsection{The Shahidi local coefficient matrix}
We would like to compute the matrix $\left[\tau(\wchi, \w_\alpha, \gamma, \gamma')\right]_{\gamma, \gamma'}$ for any unramified character $\wchi$ (not necessarily exceptional) and simple reflection $\w_\alpha, \alpha\in \Delta$. 

For Kazhdan-Patterson coverings $\wt{\GL}_r^{(n)}$, the matrix $\left[\tau(\wchi, \w_\alpha, \gamma, \gamma')\right]_{\gamma, \gamma'}$ is firstly studied in \cite{KP}. It also appears in the work of Suzuki \cite{Suz1}, Chinta and Offen \cite{CO} among others. For a subclass of Brylinski-Deligne covering groups,
the study of matrix $\left[\tau(\wchi, \w_\alpha, \gamma, \gamma')\right]_{\gamma, \gamma'}$ is conducted in \cite{Mc2} for unramified characters $\chi$, generalizing that of Kazhdan and Patterson. Meanwhile, for ramified characters, it is included in the work of \cite{GS}. However, in order to work with the full class of Brylinski-Deligne covering groups and also remove the assumption $\mu_{2n}\subseteq F^\times$ in \cite{Mc2}, we refine the computation in \cite{Mc2} slightly. This is achieved by invoking the structural facts of Brylinski-Deligne covering groups, in particular those from \S \ref{Sec:SF}. We also note that interesting phenomena dissipate when the assumption $\mu_{2n}\subseteq F^\times$ is imposed, for example for the type $\TA_r$ case in \S4. There are subtleties arising from the fact that $-1$ is not a square root. For this purpose, it is important to rigidify the formula for the matrix and express its entries in terms of naturally defined elements of the group.

Consider the Haar measure $\mu$ of $F$ such that $\mu(O_F)=1$. Thus, $\mu(O_F^\times)=1-1/q$.  The Gauss sum is given by
$$G_\psi(a, b)=\int_{O^\times_F} (u, \varpi)_n^a \cdot \psi(\varpi^b u) \mu(u), \quad a, b\in \Z.$$
It is known that
\begin{equation*}
G_\psi(a, b)=
\begin{cases}
0 & \text{ if } b<-1, \\
1-1/q & \text{ if } n| a, b\ge 0,\\
0 &\text{ if } n\nmid a, b\ge 0, \\
-1/q &\text{ if } n|a, b=-1,\\
G_\psi(a, -1) \text{ with } |G_\psi(a,-1)|=q^{-1/2} &\text{ if } n\nmid a, b=-1.
\end{cases}
\end{equation*}
Recall $\vep:=\iota((-1,\varpi)_n) \in \C^\times$. One has $\overline{G_\psi(a, b)}=\vep^a \cdot G_\psi(-a, b)$. For any $k\in \Z$, we write
$$\mathbf{g}_{\psi}(k):=G_{\psi}(k, -1).$$

As in \cite[\S 9]{Mc2}, let $f_{\gamma'} \in I(\wt{\chi})$ be the function with $\text{supp}(f_{\gamma'})=\wt{B}w_0 K_1, \ f_{\gamma'}(w_0^{-1})=i(\wchi)(\gamma') \phi_0$ for a certain compact open subgroup $K_1$. Here $\phi_0 \in i(\wchi)^{T\cap K}$ is the unramified vector in $i(\wchi)$. 
From \cite[Corollary 9.2]{Mc2}, one has $\tau(\wchi, \w_\alpha, \gamma, \gamma') =\angb{\lambda_\gamma^{{}^{\w_\alpha}\wchi} }{T_{\w_\alpha,\wchi}(f_{\gamma'})}/\val{U^-\cap K_1}$. More precisely, from Equality (9.3) of \cite{Mc2} one could evaluate $\tau(\wchi, \w_\alpha, \gamma, \gamma')$ by applying $\lambda^{{}^{\w_\alpha} \wchi}_{\gamma} \in i({}^{\w_\alpha}\wchi)^\vee$ to the integral
\begin{align} \label{Int}
& \int_F f_{\gamma'}\big( \wt{h}_\alpha(x^{-1}) \cdot \wt{e}_\alpha(-x) \cdot w_0^{-1}\big) \cdot  \psi^{-1}\big(\wt{e}_\alpha(x^{-1})\big) \mu(x) \in i({}^{\w_\alpha} \wchi).
\end{align}
Note that the integrand of (\ref{Int}) takes values in $i(\wchi)$. However, on the one hand, as vector spaces of functions on $\wt{T}$, the underlying space $i(\wchi)$ is identical to that of ${}^{w_\alpha}i(\wchi)$ (cf. \cite{Gao}); on the other hand, it follows from the Stone von-Neumann theorem that ${}^{w_\alpha} i(\wchi) \simeq i({}^{\w_\alpha}\wchi)$ as representations of $\wt{T}$. Therefore, there is a canonical vector space isomorphism $i(\wchi)\simeq i({}^{\w_\alpha}\wchi)$. 
%This isomorphism is given by $f\mapsto \ell_{\w_\alpha}(f)$ where $\ell_{\w_\alpha}(f)(\wt{t})=f(w_\alpha^{-1} \cdot \wt{t}\cdot w_\alpha)$.
For the computation below, we will follow \cite{Mc2} closely and adopt this viewpoint implicitly.
\vskip 5pt

To ease notations, write $\pi=i(\wchi)$. Use the partition $F=\bigcup_{m\in \Z} \varpi^{-m} O_F^\times$ and write $x=\varpi^{-m} u^{-1}, u\in O_F^\times$. Then $\mu(x)=|\varpi|^{-m} \mu(u)$ and the integral in (\ref{Int}) is equal to
\begin{align*}
& \sum_{m\in \Z} |\varpi|^{-m} \int_{O_F^\times} f_{\gamma'} \big( \wt{h}_\alpha(\varpi^m \cdot u) \cdot \wt{e}_\alpha(-\varpi^{-m} u^{-1}) \cdot w_0^{-1}\big) \cdot  \psi^{-1}\big(\wt{e}_\alpha(\varpi^m \cdot u)\big) \mu(u) \\
=&  \sum_{m\in \Z} \int_{O_F^\times} (u, \varpi)_n^{m Q(\alpha^\vee)} \cdot \pi(\wt{h}_\alpha(\varpi^m)) \cdot \pi(\wt{h}_\alpha(u)) \cdot \pi(\gamma') \phi_0 \cdot \psi^{-1} (\varpi^m \cdot u) \mu(u) .
\end{align*}
Suppose $\gamma' = \s_y \in \wt{T}$ for some $y\in Y$. (We write $\s_y:=\s(\varpi^y) \in \wt{T}$ for $y\in Y$, see \S 2 for notations.) Then the above is equal to
\begin{equation} \label{Sum}
 \sum_{m\in \Z} \int_{O_F^\times} (u, \varpi)_n^{mQ(\alpha^\vee) + B(\alpha^\vee, y)} \cdot \pi(\wt{h}_\alpha(\varpi^m)) \cdot  \pi(\s_y) \phi_0 \cdot \psi^{-1} (\varpi^m \cdot u) \mu(u) .
\end{equation}
From now, we write $\Gamma(m, y, \alpha^\vee):=\vep^{(m+ \angb{y}{\alpha})D(y, \alpha^\vee)}$ and $\GGMA(y, \alpha^\vee):=\Gamma(-1, y, \alpha^\vee)$, which lie in $\set{\pm 1}$. It follows from the equality (\ref{F:W-act}) that $\wt{h}_\alpha(\varpi^m) \cdot \s_y=w_\alpha \cdot \left(\Gamma(m, y, \alpha^\vee) \cdot \s_{y+m\alpha^\vee} \right) \cdot w_\alpha^{-1}$.
Therefore (\ref{Sum}) is equal to
$$\sum_{m\in \Z} \Gamma(m, y, \alpha^\vee) \cdot {}^{w_\alpha}\pi \left(\s_{\w_\alpha(y+m\alpha^\vee)}\right) \phi_0 \cdot \int_{O_F^\times} (u, \varpi)_n^{mQ(\alpha^\vee) + B(\alpha^\vee, y)}  \psi^{-1} (\varpi^m \cdot u) \mu(u).
$$
There are three cases for each term in the sum:

$\bullet$ For $m\le -2$, the integral over $O_F^\times$ vanishes, and thus the contribution to $\tau(\wchi, \w_\alpha, \gamma, \gamma')$ is 0.

$\bullet$ For $m=-1$, the contribution $\tau(\wchi, \w_\alpha, \gamma, \gamma')$ is nonzero only when $\w_\alpha(y_1) \equiv y-\alpha^\vee \mod Y_{Q,n}$ where $\gamma=\s_{y_1}, \gamma'=\s_{y}$. When $\w_\alpha(y_1)= y-\alpha^\vee$, the contribution to $\tau(\wchi, \w_\alpha, \gamma, \gamma')$ is
$$\GGMA(y, \alpha^\vee) \cdot \g(B(\alpha^\vee, y) - Q(\alpha^\vee))=\GGMA(y, \alpha^\vee) \cdot \g(\angb{y_\rho}{\alpha}Q(\alpha^\vee)).$$

$\bullet$ For any $x\in \R$, recall that we denote by $\ceil{x}$ the minimum integer such that $\ceil{x}\ge x$. The sum for $m\ge 0$ is equal to
\begin{align*}
& \sum_{m\ge 0} \Gamma(m, y, \alpha^\vee) \cdot {}^{w_\alpha}\pi \left(\s_{\w_\alpha(y+m\alpha^\vee)}\right) \phi_0 \cdot \int_{O_F^\times} (u, \varpi)_n^{mQ(\alpha^\vee) + B(\alpha^\vee, y)}  \mu(u)\\
= & \sum_{\substack{m=k\cdot n_\alpha -B(\alpha^\vee, y)/Q(\alpha^\vee) \\ k\ge \ceil{B(\alpha^\vee, y)/n_\alpha Q(\alpha^\vee)} }} \Gamma(m, y, \alpha^\vee) \cdot \vep^{(m+\angb{y, \alpha}) D(\alpha^\vee, y)} \\
& \qquad \qquad \qquad \cdot {}^{w_\alpha}\pi \left(\s_{(-m-\angb{y}{\alpha})\alpha^\vee}\right) {}^{w_\alpha}\pi \left(\s_y\right) \phi_0 \cdot (1-q^{-1}) \\
=& (1-q^{-1}) \sum_{k\ge \ceil{\angb{y}{\alpha^\vee}/n_\alpha} } \vep^{kn_\alpha B(\alpha^\vee, y)} \cdot {}^{w_\alpha}\pi \left(\wt{h}_\alpha(\varpi^{-kn_\alpha})\right) \cdot  {}^{w_\alpha}\pi(\s_y) \phi_0 \\
= & (1-q^{-1}) \sum_{k\ge \ceil{\angb{y}{\alpha^\vee}/n_\alpha} }  \wchi (\wt{h}_\alpha(\varpi^{n_\alpha}))^k \cdot  {}^{w_\alpha}\pi (\s_y) \phi_0\\
= & (1-q^{-1}) \frac{\wchi (\wt{h}_\alpha(\varpi^{n_\alpha}))^{k_{y,\alpha}}}{1-\wchi (\wt{h}_\alpha(\varpi^{n_\alpha}))} \cdot  {}^{w_\alpha}\pi (\s_y) \phi_0, \text{ where } k_{y,\alpha}=\ceil{\angb{y}{\alpha}/n_\alpha}.
\end{align*}
The contribution is nonzero only for $\gamma=\s_{y_1}$ with $y_1 \equiv y\mod Y_{Q,n}$. In particular, if $y_1= y$, then the contribution to $\tau(\wchi, \w_\alpha, \gamma, \gamma')$ (for $\gamma=\gamma'=\s_y$) is
$$(1-q^{-1}) \frac{\wchi (\wt{h}_\alpha(\varpi^{n_\alpha}))^{k_{y,\alpha}}}{1-\wchi (\wt{h}_\alpha(\varpi^{n_\alpha}))}, \text{ where } k_{y,\alpha}=\ceil{\frac{\angb{y}{\alpha}}{n_\alpha}}.$$

\vskip 5pt

To summarize, we state the following theorem by McNamara which generalizes \cite[Lemma I.3.3]{KP}:
\begin{thm}[{\cite[Theorem 13.1]{Mc2}}]
Suppose that $\gamma=\s_{y_1}$ is represented by $y_1$ and $\gamma'=\s_y$ by $y$. Then we can write $\tau(\wchi, \w_\alpha, \gamma, \gamma')=\tau^1(\wchi, \w_\alpha, \gamma, \gamma') + \tau^2(\wchi, \w_\alpha, \gamma, \gamma')$ with the following properties:
\begin{enumerate}
\item[$\bullet$] $\tau^i(\wchi, \w_\alpha,\gamma \cdot \wt{z}, \gamma' \cdot \wt{z}')=({}^{\w_\alpha} \wchi)^{-1}(\wt{z}) \cdot \tau^i(\wchi, \w_\alpha, \gamma, \gamma') \cdot \wchi(\wt{z}'), \quad \wt{z}, \wt{z}'\in \wt{A}$;
\item[$\bullet$] $\tau^1(\wchi, \w_\alpha, \gamma, \gamma')=0$  unless  $y_1 \equiv y \mod Y_{Q,n}$;
\item[$\bullet$] $\tau^2(\wchi, \w_\alpha, \gamma, \gamma')=0$   unless $y_1 \equiv \w_\alpha[y] \mod Y_{Q,n}$.
\end{enumerate}
Moreover, 
\begin{enumerate}
\item[$\bullet$] If $y_1= y$, then 
$$\tau^1(\wchi, \w_\alpha, \gamma, \gamma')=(1-q^{-1}) \frac{\wchi (\wt{h}_\alpha(\varpi^{n_\alpha}))^{k_{y,\alpha}}}{1-\wchi (\wt{h}_\alpha(\varpi^{n_\alpha}))}, \text{ where } k_{y,\alpha}=\ceil{\frac{\angb{y}{\alpha}}{n_\alpha}}.$$
\item[$\bullet$] If $y_1=\w_\alpha[y]$, then
$$\tau^2(\wchi, \w_\alpha, \gamma, \gamma') = \GGMA(y, \alpha^\vee) \cdot \g(\angb{y_\rho}{\alpha}Q(\alpha^\vee)).$$
\end{enumerate}
\end{thm}

As an analogue of \cite[Corollary I.3.4]{KP}, we have the following result.
\begin{cor} \label{C:iff-2}
Let $\wchi$ be an unramified exceptional character. Let $\lambda_\cc^{\wchi} \in \Wh(I(\wchi))$ be the $\psi$-Whittaker functional of $I(\wchi)$ associated to some $\cc \in \Ftn(i(\wchi))$. Then, $\lambda_\cc^{\wchi}$ lies in $\Wh(\Theta(\wt{G}, \wchi))$ if and only if for any simple root $\alpha\in \Delta$ one has
\begin{equation} \label{C:crucial}
\cc\big( \s_{\w_\alpha[y]}) \big) =q^{k_{y,\alpha} -1} \cdot \GGMA(y, \alpha^\vee) \cdot \g(\angb{y_\rho}{\alpha}Q(\alpha^\vee))^{-1} \cdot \cc(\s_y) \text{ for all } y.
\end{equation}
\end{cor}
\begin{proof}
By Corollary \ref{C:iff-1}, for all $\alpha\in \Delta$, we have
$$\cc (\s_y) \cdot \tau({}^{\w_\alpha}\wchi, \w_\alpha, \s_y, \s_y) +\cc( \s_{\w_\alpha[y]}) \cdot \tau({}^{\w_\alpha}\wchi, \w_\alpha, \s_{\w_\alpha[y]}, \s_y)=0,$$
where $y\in Y$ is any element. The preceding theorem gives
\begin{align*}
&\cc(\s_{\w_\alpha[y]})\\
=& -(1-q^{-1})\frac{(\wchi(\wt{h}_\alpha(\varpi^{n_\alpha})))^{-k_{y,\alpha}}}{1-\wchi(\wt{h}_\alpha(\varpi^{n_\alpha}))^{-1}} \cdot \GGMA(y, \alpha^\vee) \cdot \g(\angb{y_\rho}{\alpha}Q(\alpha^\vee))^{-1} \cdot \cc(\s_y) \\
=& q^{k_{y,\alpha} -1} \cdot \GGMA(y, \alpha^\vee) \cdot \g(\angb{y_\rho}{\alpha}Q(\alpha^\vee))^{-1} \cdot \cc(\s_y).
\end{align*}
The proof is completed.
\end{proof}

From now on, for $y\in Y$ and $\alpha\in \Delta$, we write
\begin{equation} \label{abb:t-y}
\mathbf{t}(\w_\alpha, y):= q^{k_{y,\alpha} -1} \cdot \GGMA(y, \alpha^\vee) \cdot \g(\angb{y_\rho}{\alpha}Q(\alpha^\vee))^{-1}
\end{equation}
where
$$
k_{y,\alpha}=\ceil{\frac{\angb{y_\rho}{\alpha}+1}{n_\alpha}} \text{ and } \GGMA(y, \alpha^\vee)=\vep^{\angb{y_\rho}{\alpha}\cdot D(y, \alpha^\vee)}.
$$
It is clear $\mathbf{t}(\w_\alpha, y) \ne 0$.

\begin{dfn}
For $\cc \in \Ftn(i(\wchi))$, we say that $\cc$ vanishes on $y\in Y$ if and only if $\cc(\s_y)=0$. It is said to vanish on the orbit $\mca{O}_{y_0} \subset Y$ if and only if it vanishes on all $y \in \mca{O}_{y_0}$, in which case we write $\cc(\mca{O}_{y_0})=0$.
\end{dfn}

Assume that $\cc$ gives rise to $\lambda_{\cc}^{\wchi} \in \Wh(\Theta(\wt{G}, \wchi))$. Since $\mathbf{t}(\w_\alpha, y)\ne 0$ for all $y$ and $\alpha\in \Delta$, it follows from Corollary \ref{C:iff-2} that $\cc$ vanishes on $\mca{O}_{y_0}$ if and only if it vanishes on any $y\in \mca{O}_{y_0}$. It is therefore easy to see that
\begin{equation} \label{dWh}
\dim \Wh(\Theta(\wt{G}, \wchi))=
\left|\left\{
\begin{array}{cc}
\wp_{Q,n}(\mca{O}_{y_0}): \\
\bullet \ \ \mca{O}_{y_0} \in \mca{O} \text{ is a $W$-orbit in $Y$,}\\
\bullet \ \ \text{ there exists } \cc\in \Ftn(i(\wchi)) \\
\quad \text{ satisfying (\ref{C:crucial}) for all } \alpha \in \Delta, y\in \mca{O}_{y_0}\\
\quad \text{ and also } \cc(\mca{O}_{y_0})\ne 0
\end{array}\right\} \right|.
\end{equation}
In the remaining part of this section we will prove an effective lower and upper bound for $\dim \Wh(\Theta(\wt{G}, \wchi))$.

%%%%%%%%
\subsection{A lower bound for $\dim \Wh(\Theta(\wt{G}, \wchi))$}
The Weyl group $W$ of $\mbf{G}$ has the presentation
$$W=\left\langle \w_\alpha: (\w_\alpha \w_\beta)^{m_{\alpha \beta}} =1 \text{ for } \alpha, \beta\in \Delta\right\rangle.$$

\begin{lm} \label{L:order2}
Let $\mca{O}_y \in \OF_{Q,n,\sct}$ be a $Y_{Q,n}^{\sct}$-free orbit in $Y$. Then the following holds:
$$\mathbf{t}(\w_\alpha, \w_\alpha[y]) \cdot \mathbf{t}(\w_\alpha, y)=1\text{ for all }\alpha\in \Delta.$$
\end{lm}
\begin{proof}
Note that $\w_\alpha[y]=\w_\alpha(y)+\alpha^\vee=y+(1-\angb{y}{\alpha})\alpha^\vee$. It follows that $\angb{\w_\alpha[y]}{\alpha}=2-\angb{y}{\alpha}$. Therefore
\begin{align*}
& \mathbf{t}(\w_\alpha, \w_\alpha[y]) \\
=& q^{\ceil{\angb{\w_\alpha[y]}{\alpha}/n_\alpha} -1} \cdot \GGMA( \w_\alpha[y], \alpha^\vee) \cdot \g(Q(\alpha^\vee)(\angb{\w_\alpha[y]}{\alpha}-1))^{-1} \\
=&  q^{\ceil{(2-\angb{y}{\alpha})/n_\alpha} -1} \cdot \vep^{\angb{y_\rho}{\alpha} \cdot \big(D(y, \alpha^\vee) - \angb{y_\rho}{\alpha^\vee}Q(\alpha^\vee) \big)} \cdot \g(-Q(\alpha^\vee)\angb{y_\rho}{\alpha})^{-1}
\end{align*}
and
\begin{align*}
& \mathbf{t}(\w_\alpha, \w_\alpha[y])  \cdot \mathbf{t}(\w_\alpha, y) \\
=& q^{\ceil{(2-\angb{y}{\alpha})/n_\alpha} + \ceil{\angb{y}{\alpha}/n_\alpha} -2} \cdot \vep^{\angb{y_\rho}{\alpha}^2\cdot Q(\alpha^\vee)} \\
& \qquad \qquad \cdot  \g(Q(\alpha^\vee)\cdot \angb{y_\rho}{\alpha})^{-1} \cdot \g(-Q(\alpha^\vee)\cdot \angb{y_\rho}{\alpha})^{-1}.
\end{align*}
However, it follows from $\g(k)=\vep^k \cdot \overline{\g(-k)}$ that $|\g(k)|=q^{-1/2}$. Moreover, since $\mca{O}_y$ is a $Y_{Q,n}^{\sct}$-free orbit, $\w_\alpha[y] - y \notin Y_{Q,n}^{\sct}$. Therefore, $n_\alpha\nmid (1-\angb{y}{\alpha})$ and thus
$$\ceil{\frac{2-\angb{y}{\alpha}}{n_\alpha}} + \ceil{\frac{\angb{y}{\alpha}}{n_\alpha}}=1.$$
Now it can be checked easily that $\mathbf{t}(\w_\alpha, \w_\alpha[y])  \cdot \mathbf{t}(\w_\alpha, y)=1$.
\end{proof}

Consider adjacent $\alpha, \beta\in \Delta$ from the Dynkin diagram. We would like to show that for $Y_{Q,n}^{\sct}$-free orbit $\mca{O}_y$ the equality $\prod_{i=1}^{ m_{\alpha \beta} } \mathbf{t}(\w_\alpha \w_\beta, (\w_\alpha \w_\beta)^i[y])=1$ holds, where $\mathbf{t}(\w_\alpha \w_\beta, y):=\mathbf{t}(\w_\alpha,  \w_\beta[y]) \cdot\mathbf{t}(\w_\beta, y)$. This will follow from a case by case discussion. We will give the details for $m_{\alpha \beta}=3, 4$ below and leave the case for $m_{\alpha \beta}=6$ to the reader.
\vskip 5pt

%%%%%%%%%%%
\cu{Case $m_{\alpha \beta}=3$}. The relation $(\w_\alpha \w_\beta)^{m_{\alpha \beta}}=1$ is equivalent to $\w_\alpha \w_\beta \w_\alpha=\w_\beta \w_\alpha \w_\beta$. By Lemma \ref{L:order2}, it suffices to show
\begin{equation} \label{M3}
\mathbf{t}(\w_\alpha, \w_\beta \w_\alpha[y])\cdot \mathbf{t}(\w_\beta, \w_\alpha[y])\cdot \mathbf{t}(\w_\alpha, y)=
\mathbf{t}(\w_\beta, \w_\alpha \w_\beta[y])\cdot \mathbf{t}(\w_\alpha, \w_\beta[y])\cdot \mathbf{t}(\w_\beta, y).
\end{equation}
We first note that
$$\mathbf{t}(\w_\alpha, y)= q^{\ceil{\frac{\angb{y_\rho}{\alpha}+1}{n_\alpha}} -1} \cdot \vep^{\angb{y_\rho}{\alpha}\cdot D(y, \alpha^\vee)} \cdot \g(B_Q(y_\rho, \alpha^\vee))^{-1}.$$
We also have $\angb{\w_\beta \w_\alpha(y_\rho)}{\alpha}=\angb{y_\rho}{\beta}$ since the pairing $\angb{-}{-}$ is $W$-equivariant and $\w_\alpha \w_\beta(\alpha)=\beta$. Similarly, $\angb{\w_\alpha \w_\beta(y_\rho)}{\beta}=\angb{y_\rho}{\alpha}$.
A simple computation gives
$$\begin{cases}
\mathbf{t}(\w_\alpha, y) = q^{\ceil{\frac{\angb{y_\rho}{\alpha}+1}{n_\alpha}}-1} \cdot \vep^{\angb{y_\rho}{\alpha}D(y, \alpha^\vee)} \cdot \g(\angb{y_\rho}{\alpha}Q(\alpha^\vee))^{-1} ,\\
\mathbf{t}(\w_\beta, \w_\alpha[y]) = q^{\ceil{\frac{\angb{y_\rho}{\alpha +\beta}+1}{n_\beta}}-1} \cdot \vep^{\angb{y_\rho}{\alpha+\beta}D(\w_\alpha[y], \beta^\vee)} \cdot \g(\angb{y_\rho}{\alpha+\beta}Q(\beta^\vee))^{-1}, \\
\mathbf{t}(\w_\alpha, \w_\beta \w_\alpha[y]) = q^{\ceil{\frac{\angb{y_\rho}{\beta}+1}{n_\alpha}}-1} \cdot \vep^{\angb{y_\rho}{\beta}D(\w_\beta \w_\alpha[y], \alpha^\vee)} \cdot \g(\angb{y_\rho}{\beta}Q(\alpha^\vee))^{-1} .
\end{cases} $$
Meanwhile,
$$\begin{cases}
\mathbf{t}(\w_\beta, y) = q^{\ceil{\frac{\angb{y_\rho}{\beta}+1}{n_\beta}}-1} \cdot \vep^{\angb{y_\rho}{\beta}D(y, \beta^\vee)} \cdot \g(\angb{y_\rho}{\beta}Q(\beta^\vee))^{-1} , \\
\mathbf{t}(\w_\alpha, \w_\beta[y]) = q^{\ceil{\frac{\angb{y_\rho}{\alpha +\beta}+1}{n_\alpha}}-1} \cdot \vep^{\angb{y_\rho}{\alpha+\beta}D(\w_\beta[y], \alpha^\vee)} \cdot \g(\angb{y_\rho}{\alpha+\beta}Q(\alpha^\vee))^{-1} , \\
\mathbf{t}(\w_\beta, \w_\alpha \w_\beta \w_\alpha[y]) = q^{\ceil{\frac{\angb{y_\rho}{\alpha}+1}{n_\beta}}-1} \cdot \vep^{\angb{y_\rho}{\alpha}D(\w_\alpha \w_\beta[y], \beta^\vee)} \cdot \g(\angb{y_\rho}{\alpha}Q(\beta^\vee))^{-1} .
\end{cases} $$
Since $Q(\alpha^\vee)=Q(\beta^\vee)$ and thus $n_\alpha=n_\beta$, to show the equality (\ref{M3}), it suffices to check that the powers of $\vep$ on the two sides of (\ref{M3}) are equal. However, a straightforward computation shows that this is indeed the case, and we may omit the details.

%%%%%%%%%%%%%%%%%%%%%%
\vskip 5pt
\cu{Case $m_{\alpha \beta}=4$}. Let $\alpha, \beta \in \Delta$ be two adjacent roots such that $m_{\alpha \beta}=4$. We assume that $\alpha$ is the longer one. Therefore, $\angb{\alpha^\vee}{\beta}=-1, \angb{\beta^\vee}{\alpha}=-2$, and also $Q(\beta^\vee)=2Q(\alpha^\vee)$. As in the preceding case, we want to show
\begin{equation}\label{M4}
\begin{aligned}
& \mathbf{t}(\w_\beta, \w_\alpha \w_\beta w_\alpha[y]) \cdot \mathbf{t}(\w_\alpha, \w_\beta \w_\alpha[y])\cdot \mathbf{t}(\w_\beta, \w_\alpha[y])\cdot \mathbf{t}(\w_\alpha, y)\\
=&\mathbf{t}(\w_\alpha, \w_\beta \w_\alpha \w_\beta[y]) \cdot \mathbf{t}(\w_\beta, \w_\alpha \w_\beta[y])\cdot \mathbf{t}(\w_\alpha, \w_\beta[y])\cdot \mathbf{t}(\w_\beta, y).
\end{aligned}
\end{equation}
A simple computation yields
$$\begin{cases}
\mathbf{t}(\w_\alpha, y) = q^{\ceil{\frac{\angb{y_\rho}{\alpha}+1}{n_\alpha}}-1} \cdot \vep^{\angb{y_\rho}{\alpha}D(y, \alpha^\vee)} \cdot \g(\angb{y_\rho}{\alpha}Q(\alpha^\vee))^{-1} , \\
\mathbf{t}(\w_\beta, \w_\alpha[y]) = q^{\ceil{\frac{\angb{y_\rho}{\alpha +\beta}+1}{n_\beta}}-1} \cdot \vep^{\angb{y_\rho}{\alpha+\beta}D(\w_\alpha[y], \beta^\vee)} \cdot \g(\angb{y_\rho}{\alpha+\beta}Q(\beta^\vee))^{-1} ,\\
\mathbf{t}(\w_\alpha, \w_\beta \w_\alpha[y]) = q^{\ceil{\frac{\angb{y_\rho}{\alpha +2\beta}+1}{n_\alpha}}-1} \cdot \vep^{\angb{y_\rho}{\alpha+2\beta}D(\w_\beta \w_\alpha[y], \alpha^\vee)} \cdot \g(\angb{y_\rho}{\alpha+2\beta}Q(\alpha^\vee))^{-1} , \\
\mathbf{t}(\w_\beta, \w_\alpha \w_\beta \w_\alpha[y])=q^{\ceil{\frac{\angb{y_\rho}{\beta}+1}{n_\beta}}-1} \cdot \vep^{\angb{y_\rho}{\beta}D(\w_\alpha \w_\beta \w_\alpha[y], \beta^\vee)} \cdot \g(\angb{y_\rho}{\beta}Q(\beta^\vee))^{-1}.
\end{cases} $$
On the other hand, for the right hand side of (\ref{M4}), one has
$$\begin{cases}
\mathbf{t}(\w_\beta, y) = q^{\ceil{\frac{\angb{y_\rho}{\beta}+1}{n_\beta}}-1} \cdot \vep^{\angb{y_\rho}{\beta}D(y, \beta^\vee)} \cdot \g(\angb{y_\rho}{\beta}Q(\beta^\vee))^{-1} , \\
\mathbf{t}(\w_\alpha, \w_\beta[y]) = q^{\ceil{\frac{\angb{y_\rho}{\alpha +2\beta}+1}{n_\alpha}}-1} \cdot \vep^{\angb{y_\rho}{\alpha+2\beta}D(\w_\beta[y], \alpha^\vee)} \cdot \g(\angb{y_\rho}{\alpha+2\beta}Q(\alpha^\vee))^{-1} , \\
\mathbf{t}(\w_\beta, \w_\alpha \w_\beta \w_\alpha[y]) = q^{\ceil{\frac{\angb{y_\rho}{\alpha +\beta}+1}{n_\beta}}-1} \cdot \vep^{\angb{y_\rho}{\alpha+\beta}D(\w_\alpha \w_\beta[y], \beta^\vee)} \cdot \g(\angb{y_\rho}{\alpha+\beta}Q(\beta^\vee))^{-1} , \\
\mathbf{t}(\w_\alpha, \w_\beta \w_\alpha \w_\beta[y])=q^{\ceil{\frac{\angb{y_\rho}{\alpha}+1}{n_\alpha}}-1} \vep^{\angb{y_\rho}{\alpha}D(\w_\beta \w_\alpha \w_\beta[y], \alpha^\vee)} \cdot \g(\angb{y_\rho}{\alpha}Q(\alpha^\vee))^{-1}.
\end{cases} $$
To show equality (\ref{M4}), again it suffices to show that the powers of $\vep$ of the two sides have the same parities, which is achieved from a straightforward check.
\vskip 5pt

Analogous argument for $m_{\alpha \beta}=6$ works, and we give a summary.
\begin{prop} \label{P:adj-roots}
Let $\mca{O}_y$ be a $Y_{Q,n}^{\sct}$-free orbit. For all adjacent $\alpha, \beta \in \Delta$, one has
$$\prod_{i=1}^{ m_{\alpha \beta} } \mathbf{t}(\w_\alpha \w_\beta, (\w_\alpha \w_\beta)^i[y])=1,$$
where $\mathbf{t}(\w_\alpha \w_\beta, y):=\mathbf{t}(\w_\alpha,  \w_\beta[y]) \cdot\mathbf{t}(\w_\beta, y)$.
\end{prop}

\begin{dfn} \label{D:T}
Let $\mca{O}_y \in \OF_{Q,n, \sct}$ be a $Y_{Q,n}^{\sct}$-free orbit. For any $\w=\w_k \w_{k-1}...\w_2 \w_1\in W$ written as a minimum expansion, write $\mathbf{T}(\w, y):=\prod_{i=1}^k \mathbf{t}(\w_i, \w_{i-1}... \w_1[y]),$ which, by Lemma \ref{L:order2} and Proposition \ref{P:adj-roots}, is independent of the choice of minimum expansion of $\w$.
\end{dfn}

Let $\mca{O}_y \in \OF_{Q,n}$ be a $Y_{Q,n}$-free orbit (and therefore $Y_{Q,n}^{\sct}$-free). We define a nonzero $\cc$ with support $\mca{O}_y$ as follows. First, let $\cc(\s_y)=1$, and for any $\alpha\in \Delta$, let
$$\cc(\s_{\w_\alpha[y]}):= \mathbf{t}(\w_\alpha, y) \cdot \cc(\s_y).$$
Inductively, one can define $\cc(\s_{\w[y]}):=\mathbf{T}(\w, y) \cdot \cc(\s_y)$ for any $\w\in W$. It is well-defined and independent of the minimum decomposition of $\w$. Second, extend $\cc$ by
$$\cc(\s_{\w[y]} \cdot \wt{z}) =  \cc(\s_{\w[y]}) \cdot \wchi(\wt{z}), \ \wt{z} \in \wt{A}.$$
and
$$\cc(\wt{t})=0 \text{ if } \wt{t} \notin \bigcup_{\w\in W} \s_{\w[y]}  \cdot \wt{A}.$$
By using the property that  $\mathbf{T}(\w, y)$ and $\cc(\s_{\w[y]})$ are independent of the minimum decomposition of $\w$, we see that equality (\ref{C:crucial}) is satisfied. It follows that $\wp_{Q,n}(\mca{O}_y)$ belongs to the right hand side of the equality (\ref{dWh}). Therefore,
\begin{equation} \label{LB}
\dim \Wh(\Theta(\wt{G}, \wchi)) \ge \val{\wp_{Q,n}(\OF_{Q,n})}.
\end{equation}

\subsection{An upper bound for $\dim \Wh(\Theta(\wt{G}, \wchi))$} First we show a result in the general setting regarding the usual Weyl action. Let $\Psi$ be a root system and $\Psi_s$ be a fixed choice of simple roots. Write $\text{L}:=\langle \Psi \rangle$ for the lattice generated by $\Psi$ and $V=\text{L}\otimes \R$. The Weyl group $W$ associated to $\Psi$ acts on $V$ naturally by the usual linear transformation generated by simple reflections. Recall that we write $\w(v), \w\in W, v\in V$ for this action.

\begin{lm}
Let $v\in V$ be any vector such that $\w(v)\equiv v \mod \text{L}$. Then there exists $\w'\in W$ and $\alpha\in \Psi_s$ such that $\w_\alpha(\w'(v)) \equiv \w'(v) \mod \text{L}$.
\end{lm}
\begin{proof}
Let $W_\text{aff}=\text{L}\rtimes W$ be the affine Weyl group, and denote any element of $W_\text{aff}$ by $\w_\text{a}=(y, \w)$. We call $\w$ the Weyl component of $\w_\text{a}$. The congruence $\w(v)\equiv v \mod \text{L}$ is equivalent to $\w_\text{a}(v)=v$ for some $\w_\text{a}$ which projects to $\w\in W$.

If $\w_\text{a}(v)=v$, it then follows that $v\in V$ lies on the boundary of $\wt{C}$, where $C$ is an alcove (i.e. a fundamental domain) of the action of $W_\text{aff}$ on $V$, see \cite{Bou}. Note that $\wt{C}$ is a simplicial complex whose boundary consists of $|\Psi_s| +1$ walls $\set{\text{E}_i}$.
Moreover, we may assume that for $1\le i\le \val{\Psi_s}$, the wall $\text{E}_i$ lies in the hyperplane fixed by $\w_\text{a}$ whose Weyl component is $\w_{\alpha_i}$ for some $\alpha_i\in \Psi_s$. In this case, one also knows that $\text{E}_{\val{\Psi_s}+1}$ is fixed by $(y, \w_\beta)\in W_\text{aff}$ for some $\beta \in \Psi-\Psi_s$.

Since $v\in \bigcup_i \text{E}_i$, there are two cases. First, suppose $v \in \text{E}_{i}$ for some $1\le i\le \val{\Psi_s}$; then clearly $\w_{\alpha_i}(v)\equiv v \mod \text{L}$ for some $\alpha_i\in \Psi_s$. Otherwise, suppose $v\in \text{E}_{\val{\Psi_s}+1}$. Let $\w'\in W$ be such that $\w'(\beta)\in \Psi_s$. It follows that $\w'(\text{E}_{\val{\Psi_s}+1})$ is fixed by some $\w_\text{a}=(y, \w_\alpha)$ with $\alpha\in \Psi_s$. That is, $\w_\alpha(\w'(v)) \equiv \w'(v) \mod \text{L}$. The proof is completed.
\end{proof}

\begin{prop}
Consider $\cc\in \Ftn(i(\wchi))$ such that $\lambda_\cc^{\wchi}$ is a $\psi$-Whittaker functional on $\Theta(\wt{G}, \wchi)$. If $\mca{O}_{y^0}$ is not $Y_{Q,n}^{\sct}$-free, then $\cc$ is zero on $\mca{O}_{y^0}$. It follows $\dim \Wh(\Theta(\wt{G}, \wchi)) \le \val{\wp_{Q,n}(\OF_{Q,n, \sct})}$.
\end{prop}
\begin{proof}
Write $V=Y\otimes \R$. One has $V=(Y^{\sct}\otimes \R) \oplus V_0$ where $V_0\subseteq V$ is fixed by $W$ pointwise with respect to the usual action, i.e., the action $\w(v)$ of $W$. In general $y^0_\rho\in V$; however, without loss of generality, we may assume $y^0_\rho\in Y^{\sct}\otimes \R$ now. There is a canonical $W$-equivariant isomorphism $Y_{Q,n}^{\sct}\otimes \R \simeq Y^{\sct}\otimes \R$ with respect to that usual action. Moreover, $\set{\alpha_{Q,n}^\vee}_{\alpha\in \Phi}$ forms a root system.

If $\mca{O}_{y^0}$ is not $Y_{Q,n}^{\sct}$-free, then there exists $\w\in W$ such that $\w[y^0]\equiv y^0 \mod Y_{Q,n}^{\sct}$, i.e. $\w(y^0_\rho)\equiv y^0_\rho \mod Y_{Q,n}^{\sct}$. By the preceding Lemma, there exists $y\in \mca{O}_{y^0}$ and $\alpha\in \Delta$ such that $\w_\alpha(y_\rho)\equiv y_\rho \mod Y_{Q,n}^{\sct}$. Now it suffices to show that $\cc$ vanishes on $y$.

By Corollary \ref{C:iff-2}, $\cc(\s_{\w_\alpha[y]}) =\mathbf{t}(\w_\alpha, y) \cdot \cc(\s_y)$.
Since $\w_\alpha(y_\rho)\equiv y_\rho \mod Y_{Q,n}^{\sct}$, it follows $n_\alpha| \angb{y_\rho}{\alpha}$.
Write $\angb{y_\rho}{\alpha}=k\cdot n_\alpha$. Since $\s_{\w_\alpha[y]}=\s_y \cdot \s_{-\angb{y_\rho}{\alpha})\alpha^\vee}  \cdot \vep^{\angb{y_\rho}{\alpha}\cdot D(\alpha^\vee, y)}$, one has
\begin{align*}
& \cc(\s_{\w_\alpha[y]}) \\
=& \wchi(\s_{-k n_\alpha \alpha^\vee}) \cdot \cc(\s_y)  \cdot \vep^{\angb{y_\rho}{\alpha}\cdot D(\alpha^\vee, y)} \\
=& q^{k} \cdot \vep^{kn_\alpha \cdot D(\alpha^\vee, y)} \cdot \cc(\s_y)  .
\end{align*}
On the other hand,
\begin{align*}
& \mathbf{t}(\w_\alpha, y) \cdot \cc(\s_{y})\\
= & q^{k_{y,\alpha} -1} \cdot \GGMA(y, \alpha^\vee) \cdot \g(\angb{y_\rho}{\alpha}Q(\alpha^\vee))^{-1} \cdot \cc(\s_y) \\
=& q^k \cdot (-1, \varpi)_n^{kn_\alpha \cdot D(y, \alpha^\vee)} \cdot (-q^{-1}) \cdot \cc(\s_y) .
\end{align*}
It follows $\cc(\s_y)=-q^{-1}\cdot \vep^{kn_\alpha B(y, \alpha^\vee)} \cdot \cc(\s_y)=(-q^{-1})\cdot \cc(\s_y)$.
Therefore $\cc(\s_y)=0$. The proof is completed. 
\end{proof}

\begin{thm} \label{T:LUB}
Let $\wt{G}$ be an unramified Brylinski-Deligne covering group incarnated by $(D, \eta)$. Let $\wchi$ be an unramified exceptional character and $\Theta(\wt{G}, \wchi)$ the theta representation associated with $\wchi$. Then
$$\val{\wp_{Q,n}(\OF_{Q,n})} \le \text{dim} \Wh(\Theta(\wt{G}, \wchi)) \le \val{\wp_{Q,n}(\OF_{Q,n, \sct})}.$$
\end{thm}

The center $Z(\wt{G}^\vee)$ of the dual group $\wt{G}^\vee$ of $\wt{G}$ is identified with $\Hom(Y_{Q,n}/Y_{Q,n}^{\sct}, \C^\times)$. Therefore  $Y_{Q,n}^{\sct}=Y_{Q,n}$ if and only if $Z(\wt{G}^\vee)=\set{1}$. Immediately it follows:
\begin{cor} \label{C:MC}
If the dual group $\wt{G}^\vee$ of $\wt{G}$ is of adjoint type, i.e. $Z(\wt{G}^\vee)=1$, then $\text{dim} \Wh(\Theta(\wt{G}, \wchi))=\val{\wp_{Q,n}(\OF_{Q,n})}$.
\end{cor}

For groups of type $\TE_8, \TF_4$ and $\TG_2$, the complex dual group of their covering group has trivial center and thus Corollary \ref{C:MC} applies.

More generally, if $\OF_{Q,n}=\OF_{Q,n,\sct}$, then the dimension of $\Wh(\Theta(\wt{G}, \wchi))$ can be uniquely determined. We will illustrate below that Theorem \ref{T:LUB} reovers the result of Kazhdan-Patterson in this case.

\begin{eg} \label{E:KP}
Let $\set{e_1, e_2, ..., e_r}$ be a basis for the cocharacter lattice $Y$ of $\GL_r$. The simple coroots $\Delta^\vee$ of $\GL_r$ are $\Delta^\vee=\set{\alpha_i^\vee:=e_i - e_{i+1}}_{1\le i\le r-1}$. The isomorphism class of $(D,\eta)$ in the incarnation category corresponds to a Weyl-invariant quadratic form $Q$, or equivalently, to the bilinear form $B_Q$. Let $B_Q(e_i, e_j)$ be the Weyl-invariant bilinear form determined by
$$B_Q(e_i, e_i)=2\p, \quad B_Q(e_i, e_j)=\q \text{ if } i\ne j.$$
For any root $\alpha$, one has $Q(\alpha^\vee)=2\p-\q$. We assume $2\p-\q=-1$ and therefore $n_\alpha=n$. The covering groups $\wt{\GL}_r^{(n)}$ arising from such $B_Q$ are exactly those studied by Kazhdan-Patterson. The parameter $\p$ corresponds to the twisting parameter $c$ in \cite{KP}. 

%For simplicity, we consider the cover $\wt{\GL}_r^{(n)}$ obtained as the pull-back of $\wt{\SL}_{r+1}^{(n)}$ (with $Q(\alpha^\vee)=1$ for any coroot of $\SL_{r+1}$) via ${\GL}_r \hookrightarrow {\SL}_{r+1}$ has $\p=\q=-1$. 

From $B_Q$, the lattice $Y_{Q,n}$ is given by
$$\set{\sum_i x_i e_i \in \bigoplus_{i=1}^r \Z e_i:  x_1\equiv x_2\equiv ... \equiv x_r \text{ mod } n, \text{ and } n| (\q r-1)x_i}.$$
The lattice $Y_{Q,n}^{\sct}$ is generated by $\set{\alpha^\vee_{Q,n}}_{\alpha\in \Phi}$. It is easy to check $Y_{Q,n}^{\sct}=Y_{Q,n}\cap Y^{\sct}$, and this has the following implications.
%\begin{lm}
%For the Kazhdan-Patterson convering $\wt{\GL}_r^{(n)}$, one has
%$$Y_{Q,n} \cap Y^{\sct}= Y_{Q,n}^{\sct}.$$
%\end{lm}
%\begin{proof}
%It suffices to show $Y_{Q,n} \cap Y^{\sct} \subseteq Y_{Q,n}^{\sct}$. Any element in $Y^{sc}$ is of the form $\sum_{i=1}^{r-1} x_i \alpha_i^\vee$, i.e. 
%$$(x_1, x_2-x_1, x_3-x_2, ..., x_{r-1} -x_{r-2}, -x_{r-1}) \in \bigoplus_{i=1}^r \Z e_i.$$
%It lies in $Y_{Q,n}$ if and only if
%$$a\equiv x_1\equiv x_2-x_1\equiv x_3-x_2\equiv ... \equiv x_{r-1} -x_{r-2} \equiv -x_{r-1} \text{ mod } n,$$
%and also $n|(\q r-1)a$.
%However, since $x_1 + (x_2-x_1) + ... + (x_{r-1} - x_{r-2}) + (-x_{r-1}) =0$, we see
%$n| ra$. It follows $n|a$, and therefore $n| x_i \text{ for all } i$. That is, $\sum_{i=1}^{r-1} x_i \alpha_i^\vee \in Y_{Q,n}^{sc}$. The proof is completed.
%\end{proof}

Suppose that $\mca{O}_y$ is a not $Y_{Q,n}$-free, i.e. $\w[y]-y \in Y_{Q,n}$ for some $\w\ne 1\in W$. Clearly $\w[y]-y\in Y^{\sct}$ as well. It follows $\w[y]-y \in Y_{Q,n}^{\sct}$, i.e., $\mca{O}_y$ is not $Y_{Q,n}^{\sct}$-free. Therefore, for the Kazhdan-Patterson covering group $\wt{\GL}_{r}^{(n)}$, one has $\OF_{Q,n}=\OF_{Q,n,\sct}$. Consequently, for the covering group $\wt{\GL}_r^{(n)}$ with parameter $(\p, \q)$ such that $2\p-\q=-1$, Theorem \ref{T:LUB} yields
$$\dim \Wh(\Theta(\wt{G}, \wchi))=\val{ \wp_{Q,n}(\OF_{Q,n, \sct}) },$$
which is the content of \cite[Theorem I.3.5]{KP}. Moreover, distinguished theta representations (cf. Definition \ref{D:DTR}) for $\wt{\GL}_r^{(n)}$ are completely determined in \cite[Corollary I.3.6]{KP}.
%\begin{cor}[{[KP84, Theorem 1.3.5]}]
%For the covering $\wt{\GL}_r^{(n)}$, we have $\dim \Wh(\Theta(\wchi))=\val{ \wp_{Q,n}(\OF_{Q,n}) }$.
%\end{cor}
\end{eg}
%\vskip 5pt

In the remaining part of the paper, we will determine the distinguished theta representations for coverings of simply-connected groups of type $\TA_r, \TB_r, \TC_r$ and $\TG_2$. To ease the computations, we will use the standard  coordinates for the coroot system of each type as in \cite[pp. 265-290]{Bou}.

%%%%%%%%%%%%%%%%%%%%%%%%%%%%%%%%%%%%%%%%%%%%%%%%%
\vskip 15pt

\section{The $\TA_r, r\ge 1$ case}

Consider the Dynkin diagram for the simple coroots of $\TA_r$:

$$\qquad 
\begin{picture}(4.7,0.2)(0,0)
\put(1,0){\circle{0.08}}
\put(1.5,0){\circle{0.08}}
\put(2,0){\circle{0.08}}
\put(2.5,0){\circle{0.08}}
\put(3,0){\circle{0.08}}
\put(1.04,0){\line(1,0){0.42}}
\multiput(1.55,0)(0.05,0){9}{\circle{0.02}}
\put(2.04,0){\line(1,0){0.42}}
\put(2.54,0){\line(1,0){0.42}}
\put(1,0.1){\footnotesize $\alpha_{1}^\vee$}
\put(1.5,0.1){\footnotesize $\alpha_{2}^\vee$}
\put(2,0.1){\footnotesize $\alpha_{r-2}^\vee$}
\put(2.5,0.1){\footnotesize $\alpha_{r-1}^\vee$}
\put(3,0.1){\footnotesize $\alpha_r^\vee$}
\end{picture}
$$
\vskip 10pt

The cocharacter lattice is $Y=Y^{\sct}=\bigoplus_{i=1}^{r} \Z \alpha_i^\vee$. As in \cite[page 265]{Bou}, consider the embedding $\ii_\TA: \bigoplus_{i=1}^{r} \Z \alpha_i^\vee \to \bigoplus_{i=1}^{r+1} \Z e_i$, which is given by
$$\ii_\TA: y=(x_1, x_2, ..., x_r) \mapsto \ii_\TA(y)=(x_1, x_2-x_1, x_3-x_2, ..., x_{r} -x_{r-1}, -x_{r}).$$
In particular, we can identify the image of $\ii_\TA$: any $(y_1, y_2, ..., y_r, y_{r+1}) \in \bigoplus_{i=1}^{r+1} \Z e_i$ is equal to $\ii_\TA(y)$ for some $y$ if and only if $\sum_{i=1}^{r+1} y_i=0.$

Meanwhile, $\rho=\sum_{i=1}^r \frac{i(r-i +1)}{2}\alpha_i^\vee$. We use $\ii_\TA: \bigoplus_{i=1}^{r} \Q \alpha_i^\vee \to \bigoplus_{i=1}^{r+1} \Q e_i$ to denote the canonical extension of $\ii_\TA$. Then, $\ii_\TA(\rho)=(r/2, (r-2)/2,..., -(r-2)/2, -r/2) \in \bigoplus_{i=1}^{r+1} \Q e_i$.

It follows that for any $y\in Y$, 
\begin{align*}
&\ii_\TA(y_\rho) \\
=&(x_1 -\frac{r}{2}, ..., x_i -x_{i-1} + (i-1) -\frac{r}{2}, ..., -x_r + r -\frac{r}{2}), \quad 1\le i\le r \\
=&(x_1, x_2-x_1 +1, ..., x_i -x_{i-1} + (i-1), ..., -x_r + r) + (-r/2, -r/2, ..., -r/2).
\end{align*}
From now, we write $\ii_\TA^*(y_\rho):=(x_1^*, x_2^*, ..., x_r^*, x_{r+1}^*)$ for $(x_1, x_2-x_1 +1, ..., x_i -x_{i-1} + (i-1), ..., -x_r + r) \in \bigoplus_i \Z e_i$. Thus,
$$\ii_\TA(y_\rho)=\ii_\TA^*(y_\rho) + (-r/2, -r/2, ..., -r/2).$$
Meanwhile, any $(x_1^*, x_2^*, ..., x_r^*, x_{r+1}^*) \in \bigoplus_i \Z e_i$ is equal to $\ii_\TA^*(y_\rho)$ for some $y$ if and only if $\sum_{i=1}^{r+1} x_i^* = r(r+1)/2.$

\vskip 5pt

Consider the quadratic form $Q$ on $Y=\angb{\alpha_i^\vee}{1\le i\le r}$ with $Q(\alpha^\vee_i)=1$ for all $i$.
Then
\begin{equation*}
B_Q(\alpha_i^\vee, \alpha_j^\vee) =
\begin{cases}
2, & \text{ if } i=j;\\
-1, & \text{ if } j=i+1;\\
0, & \text{ if $\alpha_i^\vee, \alpha_j^\vee$ are not adjacent.}
\end{cases}
\end{equation*}
This gives rise to the degree $n$ covering  group $\wt{\SL}_{r+1}^{(n)}$. Any element $\sum_{i=1}^r x_i \alpha_i^\vee \in Y$ lies in $Y_{Q,n}$ if and only if
\begin{equation*}
\begin{cases}
2x_1 -x_2\\
-x_1 + 2x_2 - x_3\\
-x_2 + 2x_3 - x_4 \\
... \\
-x_{r-2} + 2x_{r-1} - x_r \\
-x_{r-1} + 2x_r
\end{cases}
\in n\Z.
\end{equation*}
By using $\ii_\TA$, we see
$$
Y_{Q,n}=
\left\{
\begin{array}{cc}
(y_1, y_2, ..., y_r) \in \bigoplus_{i=1}^{r+1} \Z e_i: \\
\bullet \ \ \sum_{i=1}^{r+1} y_i=0, \\
\bullet \ \  y_1\equiv ... \equiv y_r \equiv y_{r+1} \text{ mod } n
\end{array}\right\}
\text{ and }
Y_{Q,n}^{\sct}=
\left\{
\begin{array}{cc}
(y_1, y_2, ..., y_r) \in \bigoplus_{i=1}^{r+1} \Z e_i: \\
\bullet \quad \sum_{i=1}^{r+1} y_i=0, \\
\bullet \quad  n|y_i \text{ for all } i.
\end{array}\right\}
$$

The Weyl group $W=S_{r+1}$ acts as permutations on $\bigoplus_{i=1}^{r+1}\Z e_i$. In particular, $\w_{\alpha_i}$ for $\alpha_i\in \Delta$ acts by exchanging $e_i$ and $e_{i+1}$.
 
\subsection{Case I: $\wt{\SL}_{r+1}^{(n)}, n\le r$}
Suppose $n\le r$, then for any $y\in Y$ with $\ii_\TA^*(y_\rho)=(x_1^*, x_2^*, ..., x_{r+1}^*)$, there exists $x_i^*, x_j^*, i\ne j$ such that  $n| (x_i^*-x_j^*)$. Then clearly $\w(y_\rho)-y_\rho \in Y_{Q,n}^{\sct}$ for some $\w\in W$. That is, $\mca{O}_{y} \notin \OF_{Q, n, \sct}$ and one has in this case 
$$\OF_{Q, n, \sct} = \emptyset.$$
Therefore, $\dim \Wh(\Theta(\wt{\SL}_{r+1}^{(n)}, \wchi))=0$ for $n\le r$.

%%%%%%%%
\subsection{Case II: $\wt{\SL}_{r+1}^{(n)}, n=r+1$} \label{S:A-delic}
In this case, the dual group for $\wt{\SL}_n^{(n)}$ is $\SL_n$, see \cite{We2}. Consider $\mca{O}_y \in \OF_{Q,n, \sct}$ such that  $\ii^*_\TA(y_\rho)=(0, 1, 2, ..., r-1, r) \in \bigoplus_{i=1}^{r+1} \Z e_i$. It is easy to check $\wp_{Q,n}^{\sct}(\OF_{Q,n,\sct})=\set{\wp_{Q,n}^{\sct}(\mca{O}_y)}$, and this implies $\val{\wp_{Q,n}(\OF_{Q, n, \sct})}=1$. However, we have $\mca{O}_y \notin \OF_{Q,n}$. For example, let $\w_\natural$ be such that $\ii^*_\TA(\w_\natural(y_\rho))=(1, 2, ..., r, 0)$, then $\ii_\TA(\w_\natural(y_\rho))- \ii_\TA(y_\rho)=(1, 1, ..., 1, -r) \in Y_{Q,n}$. That is, $\w_\natural[y]-y\in Y_{Q,n}$. Therefore,
$$\val{\wp_{Q,n} (\OF_{Q, n})}=0.$$

It follows that $0\le \dim \Wh(\Theta(\wt{\SL}_{n}^{(n)}, \wchi))\le 1$. In this case, it is delicate to determine $\dim \Wh(\Theta(\wt{\SL}_{n}^{(n)}, \wchi))$, and there are additional constraints on the exceptional character $\wchi$ such that $\Theta(\wt{\SL}_{n}^{(n)}, \wchi)$ is distinguished. The analysis below is devoted to this.

%%%%%%%%%%%%%%%%%%%%%%%%%%%%%%%%
\subsubsection{The reduction step}
It is clear that $\ii_{\TA}^*(y_\rho)=(0, 1, 2, ..., r-1, r)$ if and only if $y=0$. Moreover, $\ii_{\TA}^*(\w_\natural(y_\rho))=(1, 2, 3, ..., r, 0)$ for $\w_\natural=\w_{\alpha_r} \w_{\alpha_{r-1}} ... \w_{\alpha_2} \w_{\alpha_1}$. As above,
$$\w_\natural [0]-0=\sum_{i=1}^r i \cdot \alpha_i^\vee \in Y_{Q,n}.$$
Write $y_{Q,n}:=\sum_{i=1}^r i \cdot \alpha_i^\vee$. In fact, the set $\set{n\alpha_i^\vee: 2\le i\le r} \cup \set{y_{Q,n}}$ forms a basis for $Y_{Q,n}$, whereas $\set{n\alpha_i^\vee: 2\le i\le r} \cup \set{n\cdot y_{Q,n}}$ is a basis for $Y_{Q,n}^{\sct}$. It follows that any exceptional character $\wchi$ is determined by its value at $\s_{y_{Q,n}}$.

We choose the bisector $D$ on $Y^{\sct}$ such that $D(\alpha_i^\vee, \alpha_j^\vee)$ is given by
$$D(\alpha_i^\vee, \alpha_j^\vee)=
\begin{cases}
Q(\alpha^\vee_i) & \text { if } i=j; \\
0 & \text{ if } i< j;\\
B_Q(\alpha_i^\vee, \alpha_j^\vee) & \text{ if } i>j.
\end{cases}
$$

Recall from Corollary \ref{C:iff-2} that $\cc \in \Ftn(i(\wchi))$ gives rise to a $\psi$-Whittaker functional of $\Theta(\wt{\SL}_n^{(n)}, \wchi)$ if and only if for all $y$ and $\alpha \in \Delta$,
$$\cc(\s_{\w_\alpha[y]}) =q^{k_{y,\alpha} -1} \cdot \GGMA(y, \alpha^\vee) \cdot \g(B(\alpha^\vee, y_\rho))^{-1} \cdot \cc(\s_y).$$
For $1\le i\le r$, write $y_{\langle i \rangle}=\w_{\alpha_i} \w_{\alpha_{i-1}} ... \w_{\alpha_1}[0]$ and we set $y_{\langle 0\rangle }=0$. Recall that $\mathbf{t}(\w_\alpha, y)$ is the coefficient in the above formula. In this case, it reads $\mathbf{t}(\w_\alpha, y)=q^{k_{y,\alpha} -1} \cdot \GGMA(y, \alpha^\vee) \cdot \g(\angb{y_\rho}{\alpha})^{-1}$ since $Q(\alpha^\vee)=1$ (and therefore $n_\alpha=n$) for all $\alpha\in \Delta$. In order to have $\dim \Wh(\Theta(\wt{\SL}_{n}^{(n)}, \wchi))=1$, we must have the equality
\begin{equation} \label{E:A}
\wchi(\s_{y_{Q,n}}) = \mathbf{T}(\w_\natural, 0) \text{ where } \mathbf{T}(\w_\natural, 0)=\prod_{i=1}^r \mathbf{t}(\w_{\alpha_{i}}, y_{\langle i-1\rangle }).
\end{equation}

We would like to show that the equality (\ref{E:A}) is also sufficient. Consider any $\w'\in W, y\in \mca{O}_0$, one has $\cc(\s_{\w'[y]})=\mathbf{T}(\w', y) \cdot \cc(\s_y)$. Now assume that $\w'[y]-y\in Y_{Q,n}$, we have
$$\cc(\s_{\w'[y]-y + y})=\wchi(\s_{\w'[y]-y})\cdot \cc(\s_y) \cdot \vep^{D(\w'[y]-y, y)}.$$
To show $\dim \Wh(\Theta(\wt{\SL}_{n}^{(n)}, \wchi))=1$, it suffices to show $\cc(\s_y)$ to be nonzero for all $y\in \mca{O}_0$ such that $\w'[y]-y\in Y_{Q,n}$. That is, it requires
\begin{equation} \label{E:gen-conA}
\wchi(\s_{\w'[y]-y})=\vep^{D(\w'[y]-y, y)} \cdot \mathbf{T}(\w',y).
\end{equation}

Write $\w'[y]-y=\sum_{i=2}^{r} k_i \cdot \alpha_{i, Q,n}^\vee + k_1\cdot y_{Q,n}$. Note that $\mca{O}_0$ is $Y_{Q,n}^{\sct}$-free, thus $k_1\ne 0$. We may reduce the negative case to the positive case by a simple computation, and therefore we can assume that $k_1\ge 1$. Further more, we may apply induction on $k_1$, and thus it suffices to: i) prove the inductive step, ii) check the equality (\ref{E:gen-conA}) when $\w'[y]-y=\sum_{i=2}^{r} k_i \alpha_{i, Q,n}^\vee + y_{Q,n}$. The assertion i) can be checked easily, and thus we will only outline the proof of ii). 

For ii), if $\w'[y]-y=\sum_{i=2}^{r} k_i \alpha_{i, Q,n}^\vee + y_{Q,n}$, then it is not hard to see that $\w'[y]-y=\w(y_{Q,n})$, i.e., $\w^{-1}\w'[y]-\w^{-1}[y]=y_{Q,n}$ for some $\w\in W$. We may change $\w$ if necessary such that $\w^{-1}[y]=0$. With this assumption, $\w^{-1}\w' \w=\w_\natural$, i.e. $\w'=\w \w_\natural \w^{-1}$. Therefore, we are reduced to show that for any $\w\in W$:
\begin{equation} \label{E:AA}
\wchi(\s_{\w \w_\natural[0]-\w[0]})=\vep^{D(\w \w_\natural [0]-\w[0], \w[0])} \cdot \mathbf{T}(\w \w_\natural \w^{-1},\w[0]).
\end{equation}
To show (\ref{E:AA}), we would like to apply induction on the length of $\w$. When $\w=1$, it is just the equality (\ref{E:A}).
For the induction step, assuming the equality (\ref{E:AA}), we would like to prove that for $\alpha\in \Delta$ the following equality holds:
\begin{equation} \label{E:A3}
\wchi(\s_{\w_\alpha \w \w_\natural[0]-\w[0]})=\vep^{D(\w_\alpha \w \w_\natural [0]-\w_\alpha \w[0], \w_\alpha \w[0])} \cdot \mathbf{T}(\w_\alpha \w \w_\natural \w^{-1} \w_\alpha^{-1},\w_\alpha \w[0]).
\end{equation}
For this purpose, write $x:=\w \w_\natural[0] -\w[0] \in Y_{Q,n}$. We have $n_\alpha | \angb{x}{\alpha}$. Write $\angb{x}{\alpha}=k\cdot n_\alpha$. 

The left hand side of (\ref{E:A3}) is
\begin{align*}
& \wchi\big( \s_{x-\angb{x}{\alpha}\alpha^\vee} \big) \\
=& \wchi\big( \s_{x} \big) \cdot \wchi\big( \s_{-k n_\alpha \alpha^\vee} \big) \cdot \vep^{D(x, -kn_\alpha \alpha^\vee)} \\
=& \wchi\big( \s_{x} \big) \cdot \wchi\big(\wt{h}_\alpha(\varpi^{n_\alpha})\big)^{-k} 
%\text{ by the fairness of $D$ }
 \\
=& q^k \cdot \wchi(\s_x).
\end{align*}
The right hand side of (\ref{E:A3}) is
% (by using fairness of $D$)
\begin{align*}
& \vep^{D(\w_\alpha(x), \w_\alpha \w[0])} \cdot \mathbf{t}(\w_\alpha, \w \w_\natural[0]) \cdot \mathbf{T}(\w \w_\natural  \w^{-1}, \w[0]) \cdot \mathbf{t}(\w_\alpha, \w_\alpha \w[0]) \\
=& \vep^{D(x, \w_\alpha \w[0]) - \w[0])} \cdot \wchi(\s_x) \cdot 
\mathbf{t}(\w_\alpha, \w \w_\natural[0]) \cdot \mathbf{t}(\w_\alpha, \w_\alpha \w[0]) \text{ by (\ref{E:AA})} \\
=& \vep^{D(x, \w_\alpha \w[0]) - \w[0])} \cdot \wchi(\s_x) \cdot \mathbf{t}(\w_\alpha, \w \w_\natural[0]) \cdot \mathbf{t}(\w_\alpha, \w[0])^{-1}\\
=& \vep^{D(x, \w_\alpha \w[0]) - \w[0])} \cdot q^{\ceil{\frac{\angb{\w[0]}{\alpha}}{n_\alpha}}-1+k} \cdot \g(\angb{\w(0_\rho)}{\alpha}Q(\alpha^\vee))^{-1} \cdot \vep^{\angb{\w(0_\rho)}{\alpha}\cdot D(\w \w_\natural[0], \alpha^\vee)} \\
& \qquad \qquad \cdot \wchi(\s_x) \cdot q^{-\ceil{\frac{\angb{\w[0]}{\alpha}}{n_\alpha}}+1} \cdot \g(\angb{\w(0_\rho)}{\alpha}Q(\alpha^\vee)) \cdot \vep^{\angb{\w(0_\rho)}{\alpha}\cdot D(\w[0], \alpha^\vee)} \\
=& \wchi(\s_x) \cdot q^k \cdot \vep^{\angb{\w(0_\rho)}{\alpha}D(x, \alpha^\vee)} \cdot \vep^{\angb{\w(0_\rho)}{\alpha}D(x, \alpha^\vee)} \\
=&  \wchi(\s_x) \cdot q^k,
\end{align*}
which is clearly equal to the left hand side. To summarize, we have

\begin{prop} \label{P:cond-A}
Let $\wchi \in \Hom_\iota (Z(\wt{T}), \C^\times)$ be an exceptional character of $\wt{\SL}_{n}^{(n)}$. Then 
$$\dim \Wh(\Theta(\wt{\SL}_{n}^{(n)}, \wchi))=1$$ if and only if $\wchi$ is the unique exceptional character satisfying (\ref{E:A}).
\end{prop}

We would like to explicate the condition given by (\ref{E:A}).
\begin{lm} \label{L:explicitA}
One has 
$$\mathbf{T}(\w_\natural, 0)=
\begin{cases}
q^{-r/2} & \text{ if } n \text{ is odd};\\
\vep^{n(n-2)/8} \cdot q^{-n/2} \cdot \g(-n/2)^{-1} & \text{ if } n \text{ is even}.
\end{cases}
$$
\end{lm}
\begin{proof}
We compute each $\mathbf{t}(\w_{\alpha_i}, y_{\langle i-1 \rangle})$ for $1\le i\le r$. First, one can check easily $y_{\langle i\rangle}=\sum_{j=1}^i i\cdot \alpha_i^\vee=\alpha_1^\vee + 2\alpha_2^\vee + ... + i\cdot \alpha_i^\vee$. Thus, $\angb{y_{\langle i-1\rangle}}{\alpha_i}=-(i-1)$ and therefore
$$k_{y_{\langle i-1 \rangle}, \alpha_i}=0 \text{ for all } 1\le i\le r.$$
Second, $\GGMA( y_{\langle i-1\rangle}, \alpha_i^\vee)=\vep^{-i\cdot D(y_{\langle i-1 \rangle}, \alpha_i^\vee)}$. Since $D(\alpha_j^\vee, \alpha_i^\vee)=0$ for all $j< i$, we see $\GGMA(y_{\langle i-1\rangle}, \alpha_i^\vee)=1$.
Therefore, $\mathbf{t}(\w_{\alpha_i}, y_{\langle i-1\rangle})=q^{-1} \cdot \g(-i)^{-1}$. Now, if $1\le i, j\le n$ and $i+j=n$, one has
\begin{align*}
& \g(-i)^{-1} \cdot \g(-j)^{-1} \\
= &\g(-i)^{-1} \cdot \big( \wt{\g(j)}\cdot \vep^j \big)^{-1}\\
= & |\g(j)|^{-2} \cdot \vep^i\\
=& q\cdot \vep^i.
\end{align*}
The result then follows from simply multiplying together each term.
\end{proof}
%%%%
\vskip 5pt
\subsubsection{Interlude: Weil-index}
Let $\ggma_\psi$ be the Weil-index given in \S \ref{S:UDC}.

\begin{lm} \label{L:G-W}
Suppose $n=2m$ is an even number. Then the following equality holds:
$$\g(m)=\frac{q^{-1/2}}{\ggma_\psi(\varpi)}.$$
\end{lm}
\begin{proof}
By definition, $\g(m)$ is equal to
\begin{align*}
& \int_{O_F^\times} (u, \varpi)_2 \cdot \psi^{-1}(\varpi^{-1} u) \mu(u) \\
=&  \int_{O_F^\times} \ggma_\psi(\varpi u) \ggma_\psi(\varpi)^{-1} \ggma_\psi(u)^{-1}  \cdot \psi^{-1}(\varpi^{-1} u) \mu(u) \\
=&  \ggma_\psi(\varpi)^{-1}\cdot \int_{O_F^\times} \ggma_\psi(\varpi u) \cdot \psi^{-1}(\varpi^{-1} u) \mu(u) .
\end{align*}
%First, $\ggma_\psi(\varpi^{-1}) \ggma_\psi(\varpi) \cdot (\varpi, \varpi)_2=1$ gives
%$$\ggma_\psi(\varpi^{-1}) \ggma_\psi(\varpi)^{-1}= \ggma_\psi(\varpi^{-1})^2 \cdot (\varpi, \varpi)_2=1$$

However, by Equality (3.7) of \cite[Lemma 3.2]{Szp1},
$$\ggma_\psi(\varpi u)=q^{-1/2}\Big(1+ q \int_{O_F^\times} \psi(\varpi^{-1} v^2 u)\mu(v) \Big).$$

Thus,
\begin{align*}
&\g(m) \\
=& q^{-1/2} \cdot \ggma_\psi(\varpi)^{-1}  \cdot \int_{O_F^\times}\Big(1+ q \int_{O_F^\times} \psi(\varpi^{-1} v^2 u)\mu(v)\Big) \psi^{-1}(\varpi^{-1} u) \mu(u) \\
=&  q^{-1/2} \cdot \ggma_\psi(\varpi)^{-1}  \cdot \Big(-1/q + q\cdot \int_{O_F^\times} \int_{O_F^\times} \psi \big(\varpi^{-1} u (v^2-1) \big) \mu(u) \mu(v) \Big) 
\end{align*}

Let $D=\{v\in O_F^\times: |1-v^2|=1\}$ and $H=\{v\in O_F^\times: |1-v^2|\le q^{-1}\}$. We get
\begin{align*}
&\int_{O_F^\times} \Big( \int_{O_F^\times} \psi \big(\varpi^{-1} u (v^2-1) \big) \mu(u) \Big) \mu(v) \\
=&\int_{v \in H} \int_{O_F^\times} \psi \big(\varpi^{-1} u (v^2-1) \big) \mu(u) \mu(v) +\int_{v\in D} \int_{O_F^\times} \psi \big(\varpi^{-1} u (v^2-1) \big) \mu(u) \mu(v) \\
=&\mu(H) \cdot (1-q^{-1}) + \mu(D) \cdot (-q^{-1}) \text{ by (8.19) of \cite[Lemma 8.6]{Szp1} }\\
=& 2q^{-1} \cdot (1-q^{-1}) + (1-3q^{-1}) \cdot (-q^{-1}) \text{ by \cite[Lemma 8.9]{Szp1} } \\
=& q^{-1} + q^{-2}.
\end{align*}
The result follows easily by simplification.
\end{proof}

\subsubsection{An explicit criterion}
Consider the \emph{unitary} distinguished character $\wchi^0_{\psi'}$ constructed in \cite{GG}, which we recalled and gave in formula (\ref{UDC}). Then the character $\wchi_{\psi'}=\wchi^0_{\psi'} \cdot \delta_B(\cdot)^{\frac{1}{2n}}$ is an exceptional character. In the simply-connected case, $J=Y_{Q,n}^{\sct}$. For the definition of $\wchi^0_{\psi'}$, we pick a basis $\set{y_i}$ for $Y_{Q,n}$ such that $\set{k_iy_i}$ is a basis for $J=Y_{Q,n}^{\sct}$. Then by definition,
$$\wchi^0_{\psi'} (\s_{y_i})=\ggma_{\psi'}(\varpi)^{2(k_i-1)Q(y_i)/n}$$
and, for $y=\sum_i n_i y_i \in Y_{Q,n}$, one has
$$\wchi^0_{\psi'} (\s_y)=\prod_i \wchi^0_{\psi'}(\varpi^{n_i})^{2(k_i-1)Q(y_i)/n} \cdot \vep^{\sum_{i<j} n_i n_j D(y_i, y_j)}.$$
For the covering group $\wt{\SL}_{n}^{(n)}$, we take $y_i=n\alpha_i^\vee, 2\le i\le r$ and $y_1=y_{Q,n}$, with $k_i=1$ for $2\le i\le r$ and $k_1=n$.

An easy computation shows $Q(y_{Q,n})=r(r+1)/2$, and thus
\begin{equation} \label{Dis-char}
\begin{aligned}
& \wchi_{\psi'}(\s_{y_{Q,n}}) \\
=& \wchi^0_{\psi'}(\s_{y_{Q,n}}) \cdot \delta_B(\s_{y_{Q,n}})^{\frac{1}{2n}} \\
=& \ggma_{\psi'}(\varpi)^{(n-1)^2} \cdot q^{-(n-1)/2}.
\end{aligned}
\end{equation}

\begin{prop}
For the exceptional character $\wchi_{\psi'}=\wchi_{\psi'}^0 \cdot \delta_B(\cdot)^{\frac{1}{2n}}$ given above, one has $\dim \Wh(\Theta(\wt{\SL}_{n}^{(n)}, \wchi_{\psi'}))=1$ if and only if the following holds:
$$\begin{cases}
\text{any } \psi', & \text{ if $n$ is odd};\\
\ggma_{\psi'}(\varpi)=\ggma_{\psi}(\varpi), & \text{ if $n\equiv 0, 2 \mod 8$};\\
%\ggma_{\psi'}(\varpi)=\ggma_{\psi}(\varpi), & \text{ if $n\equiv 2 \mod 8$};\\
\ggma_{\psi'}(\varpi)=(-1, \varpi)_4\cdot \ggma_{\psi}(\varpi), & \text{ if $n \equiv 4 \mod 8$};\\
\ggma_{\psi'}(\varpi)=\ggma_{\psi}(\varpi)^{-1}, & \text{ if $n\equiv 6 \mod 8$}.
\end{cases}$$
\end{prop}
\begin{proof}
By the value of $\wchi_{\psi'}(\s_{y_{Q,n}})$ in (\ref{Dis-char}), it follows from Lemma \ref{L:explicitA} that the equality (\ref{E:A}) is equivalent to
\begin{equation} \label{exp-psi}
\ggma_{\psi'}(\varpi)^{(n-1)^2} \cdot q^{-(n-1)/2}=
\begin{cases}
q^{-r/2} & \text{ if } n \text{ is odd};\\
(-1, \varpi)_n^{n(n-2)/8} \cdot q^{-n/2} \cdot \g(-n/2)^{-1} & \text{ if } n \text{ is even}.
\end{cases}
\end{equation}
For $n$ odd, the equality holds for any $\psi'$. Now we assume $n$ even.

For $n=4k +2$, by Lemma \ref{L:G-W}, the required equality in (\ref{exp-psi}) becomes
$$\ggma_{\psi'}(\varpi)=\ggma_{\psi}(\varpi)^{2k+1}.$$
In particular, if $k$ is even, it is equivalent to $\ggma_{\psi'}(\varpi)=\ggma_{\psi}(\varpi)$. If $k$ is odd, it is equivalent to $\ggma_{\psi'}(\varpi)=\ggma_{\psi}(\varpi)^{-1}$.

For $n=4k$, applying Lemma \ref{L:G-W} again, the equality in (\ref{exp-psi}) reads
$$\ggma_{\psi'}(\varpi)=(-1, \varpi)_n^k \cdot \ggma_{\psi}(\varpi)=(-1, \varpi)_4 \cdot \ggma_{\psi}(\varpi).$$
A special case is when $k$ is even. In this case $(-1, \varpi)_4=1$ and therefore it is equivalent to $\ggma_{\psi'}(\varpi)=\ggma_{\psi}(\varpi)$.
\end{proof}

\begin{cor} \label{C:AC01}
Consider $\psi'=\psi_a$ for some $a\in F^\times$. Assume that $\psi_a$ has conductor $O_F$, i.e. $a\in O_F^\times$. Then $\dim \Wh(\Theta(\wt{\SL}_n^{(n)}, \wchi_{\psi_a}))=1$ if and only if the following holds:
$$\begin{cases}
a\in O_F^\times, & \text{ if $n$ is odd};\\
a\in (O_F^\times)^2, & \text{ if $n\equiv 0, 2 \mod 8$};\\
% \ggma_{\psi'}(\varpi)=\ggma_{\psi}(\varpi), & \text{ if $n\equiv 2 \mod 8$};\\
a^2\in - (O_F^\times)^4 , & \text{ if $n \equiv 4 \mod 8$};\\
a\in - (O_F^\times)^2, & \text{ if $n\equiv 6 \mod 8$}.
\end{cases}$$
\end{cor}
%%%%%%%%%%%%%%%%%
\vskip 5pt

\begin{rmk}
As the referee points out, the fact that for any exceptional representation  $\Theta(\wt{\SL}_{n}^{(n)}, \wchi)$ there exists $\psi$ such that it is $\psi$-generic and that $\dim \Wh (\Theta(\wt{\SL}_{n}^{(n)}, \wchi))\le 1$ for all $\psi$ also follows from the work of \cite{KP} on $\wt{\GL}_n^{(n)}$ combined with the relation between $\wt{\SL}_n^{(n)}$ and $\wt{\GL}_n^{(n)}$ given by Adams in \cite{Ada}. However, our Corollary \ref{C:AC01} above gives precise information for the matching between $\psi$ and the distinguished theta representation in terms of the distinguished character.
\end{rmk}

\begin{eg} \label{E:SL2}
The first nontrivial example is the metaplectic covering $\wt{\SL}_2^{(2)}$. In this case, we have $Y_{Q,n}=Y=\Z \cdot \alpha^\vee$ and $Y_{Q,n}^{\sct}=\Z \cdot (2\alpha^\vee)$. As mentioned at the beginning of \S \ref{S:A-delic}, one has that the lower and upper bounds in Theorem \ref{T:LUB} are 0 and 1 respectively and thus
$$0\le \dim \Wh(\Theta(\wt{\SL}_2^{(2)}, \wchi)) \le 1$$
for any exceptional $\wchi$. For the character $\psi_a$, the representation $\Theta(\wt{\SL}_2^{(2)}, \wchi_{\psi_a})$ is the even Weil representation in the following exact sequence:
%$$\begin{tikzcd}
%\text{St}(\wchi_{\psi_a}) \ar[r, hook] & I(\wchi_{\psi_a}) \ar[r, two heads] & \Theta(\wt{\SL}_2^{(2)}, \wchi_{\psi_a}),
%\end{tikzcd} $$
$$\xymatrix{
\text{St}(\wchi_{\psi_a}) \ar@{^(->}[r] & I(\wchi_{\psi_a}) \ar@{>>}[r]  & \Theta(\wt{\SL}_2^{(2)}, \wchi_{\psi_a}),
} $$
where $\text{St}(\wchi_{\psi_a})$ is the metaplectic analogue of the Steinberg representation. Corollary \ref{C:AC01} above recovers the well-known fact, which follows from the work of Gelbart and Piatetski-Shapiro \cite{GePS},  that for $\wt{\SL}_2^{(2)}$ the even Weil representation $\Theta(\wt{\SL}_2^{(2)}, \wchi_{\psi_a})$ (for unramified data) is $\psi$-generic if and only if $a\in (O_F^\times)^2$. We note that this also follows directly from the computation of the local coefficient for $\wt{\SL}_2^{(2)}$ in \cite{Szp3}.
\end{eg}

\begin{eg}
We also discuss explicitly the example $\wt{\SL}_3^{(3)}$. Consider $\wt{\SL}_3^{(3)}$ with cocharacter lattice $Y=\langle \alpha_1^\vee, \alpha_2^\vee \rangle$. Consider $Q$ such that $Q(\alpha^\vee_i)=1$. Then
$$Y_{Q,n} =\langle 2\alpha_1^\vee + \alpha_2^\vee, 3\alpha_1^\vee \rangle = \langle 2\alpha_2^\vee + \alpha_1^\vee, 3\alpha_2^\vee \rangle.$$
Note $Y=\langle 2\alpha_1^\vee + \alpha_2^\vee, \alpha_1^\vee \rangle = \langle 2\alpha_2^\vee + \alpha_1^\vee, \alpha_2^\vee \rangle$. We know $\rho=\alpha_1^\vee + \alpha_2^\vee$. For $y=0$ one has
$$y_\rho=0_\rho= -(\alpha_1^\vee + \alpha_2^\vee).$$
Consider $\w_\natural=\w_{\alpha_1} \w_{\alpha_2}$, then $\w_{\alpha_2}[y]=\alpha_2^\vee$ and moreover $\w_{\alpha_1} \w_{\alpha_2}[y]= 2\alpha^\vee_1 +\alpha_2^\vee$. One has
\begin{align*}
& \cc(\s_{\w_1\w_2[y]}) \\
=& q^{k_{\w_2[y],\alpha_1} -1} \cdot \GGMA(\w_2[y], \alpha_1^\vee) \cdot \g(Q(\alpha_1^\vee)(\angb{\w_2[y]}{\alpha_1}-1))^{-1} \\
& \qquad \cdot q^{k_{y,\alpha_2} -1} \cdot \GGMA(y, \alpha_2^\vee) \cdot \g(Q(\alpha_2^\vee)(\angb{y}{\alpha_2}-1))^{-1} \cdot \cc(\s_y)\\
=&q^{\ceil{\frac{\angb{\alpha_2^\vee}{\alpha_1}}{3}} + \ceil{\frac{\angb{y}{\alpha_2}}{3}} -2} \cdot \GGMA(\alpha_2^\vee, \alpha_1^\vee) \cdot \GGMA(0, \alpha_2^\vee) \cdot \g(-2)^{-1} \g(-1)^{-1} \cdot \cc(1_{\wt{\SL}_3^{(3)}}) \\
=& q^{-2} \cdot q \cdot \cc(1_{\wt{\SL}_3^{(3)}})= q^{-1},
\end{align*}
where $\cc$ is normalized to take value 1 at the $1\in \wt{\SL}_3^{(3)}$. This implies that necessarily $ \cc(\s_{\w_1\w_2[y]}) =q^{-1}$, and thus
$$\wchi(\s_{\w_1\w_2[y]})=q^{-1}.$$
Note, this is not a consequence of $\wchi$ being exceptional, although it is compatible. Clearly, an exceptional character $\wchi$ is such that
$$\begin{cases}
\wchi(\s_{\w_1\w_2[y]})^3 =q^{-3}, \\
\wchi(\s_{3\alpha_1^\vee})=q^{-1}.
\end{cases} $$
In particular, if $\wchi(\s_{\w_1\w_2[y]})=\zeta \cdot q^{-1}$ for some third root of unity $\zeta \ne 1$, then $\dim \Wh(\Theta(\wt{\SL}_3^{(3)}, \wchi))=0$ for such $\wchi$.
\end{eg}

\subsection{Case III: $\wt{\SL}_{r+1}^{(n)}, n=r+2$} For $n=r+2$, we show $Y_{Q,n}=Y_{Q,n}^{\sct}$ and therefore Corollary \ref{C:MC} applies. Pick any $(y_1, y_2, ..., y_{r+1})\in Y_{Q,n}$, we have
$$a\equiv y_1 \equiv y_2 \equiv ... \equiv y_{r+1} \mod n,$$
where $a\in \set{0, 1, 2, ..., r+1}$. Write $y_i=k_i n + a$. Since $\sum_{i=1}^{r+1} y_i=0$, one has
$$n\cdot \big(\sum_{i=1}^{r+1} k_i\big) + (r+1)\cdot a=0.$$
In particular, $n| (r+1)a$. However, $\text{gcd}(n, r+1)=1$, therefore $n|a$ and $a=0$. That is, $Y_{Q,n}=Y_{Q,n}^{\sct}$ and therefore $\dim \Wh(\Theta(\wt{\SL}_{r+1}^{(r+2)}, \wchi))=\val{\wp_{Q,n}(\OF_{Q,n})}$. Note that, the equality $Y_{Q,n}=Y_{Q,n}^{\sct}$ reflects the fact that the dual group for $\wt{\SL}_n^{(n+1)}$ is $\text{PGL}_n$ (cf. \cite[\S 2.7.2]{We2}).

We claim that the dimension is equal to 1 in this case. Let $\mca{O}_y \in \OF_{Q,n,\sct}$ be a $Y_{Q,n}^{\sct}$-free orbit with $\ii_\TA^*(y_\rho)=(0, 1, ..., r-1, r) \in \bigoplus_{i=1}^{r+1} \Z e_i$. We know that $\mca{O}_y$ is $Y_{Q,n}$-free (or equally, $Y_{Q,n}^{\sct}$-free). Moreover, one can check easily that $\wp_{Q,n}(\OF_{Q,n})=\set{\wp_{Q,n}(\mca{O}_y)}$. Therefore $\text{dim} \Wh(\Theta(\wt{\SL}_{r+1}^{(r+2)}, \wchi))=1$ for the unique exceptional character $\wchi$ in this case.

\subsection{Case IV: $\wt{\SL}_{r+1}^{(n)}, n\ge r+3$}
\begin{lm} \label{L:A01}
Consider $y\in Y$ such that $\ii_\TA^*(y_\rho)=(x_1^*, x_2^*, ..., x_r^*, x_{r+1}^*)$ with $x_i^*=i-1$. If $n\ge r+3$, the orbit  $\mca{O}_y$ is $Y_{Q,n}$-free.
\end{lm}
\begin{proof}
Suppose not, there exists $\w\ne 1$ such that $\w[y]-y\in Y_{Q,n}$. Identify $\w$ with a permutation, then we have
$$(x_1^*, x_2^*..., x_{r+1}^*) - (x_{\w(1)}^*, x_{\w(2)}^*..., x_{\w(r+1)}^*) \in Y_{Q,n}.$$
More precisely, $i-\w(i) \equiv j-\w(j) \text{ mod } n$ for all $i, j$. Clearly, $n\nmid (i-\w(i))$ for all $i$, otherwise one can deduce $\w(i)=i$ for all $i$ and therefore $\w=1$. That is, $(i-\w(i))$ is either negative or positive. We reorder the terms $(i-\w(i))$'s as
$$-r\le (i_1-\w(i_1)) \le (i_2-\w(i_2)) \le ... < 0 < ...   \le (i_r -\w(i_r)) \le  (i_{r+1} -\w(i_{r+1})) \le r.$$
Write $(i_{1}-\w(i_1))=-s, s\in \N$ and $(i_{r+1}-\w(i_{r+1}))=t, t\in \N$. It is easy to see that any negative $i-\w(i)$ must be equal to $-s$, and any positive $i-\w(i)$ must be equal to $t$.

We claim that $2< t+s \le r+1$ and therefore $n\nmid (t+s)$, i.e. $\w[y] -y \notin Y_{Q,n}$ for all $\w\ne 1$. Note $0-\w(0)=-s$ and $r-\w(r)=t$. Suppose $t+s > r+1$, then there exists $i_0$ such that $r+1-t< i_0 < 1+s$. However, there exists no $i'$ such that $\w(i')=i_0$. This is a contradiction, and the claim follows.

Therefore $\mca{O}_y$ is $Y_{Q,n}$-free for the given $y$. The proof is completed.
\end{proof}

It follows that $\text{dim} \Wh(\Theta(\wt{\SL}_{r+1}^{(n)}, \wchi)) \ge 1$ for $n\ge r+3$. In principle, one could proceed as in \S \ref{S:A-delic} to analyze every element in $\wp_{Q,n}(\OF_{Q,n,\sct})$ and determine completely $\text{dim} \Wh(\Theta(\wt{\SL}_{r+1}^{(n)}, \wchi))$ in this case. However, the level of complexity of the computation depends inevitably on (the center of) the dual group of $\wt{\SL}_r^{(n)}$ and could be quite involved for general $n\ge r+3$.

We summarize for the $n\le r+2$ cases below.
\begin{thm} \label{T:A}
Consider the Brylinski-Deligne covering $\wt{\SL}_{r+1}^{(n)}, n\le r+2$ with $Q(\alpha^\vee)=1$ for all coroots $\alpha^\vee$. Let $\wchi$ be an exceptional character of $\wt{\SL}_{r+1}^{(n)}$. Then $\dim \Wh(\Theta(\wt{\SL}_{r+1}^{(n)}, \wchi))=1$ if and only if 
\begin{enumerate}
\item[$\bullet$] $n=r+2$ and $\wchi$ is the only exceptional character, or
\item[$\bullet$] $n=r+1$ and $\wchi$ is the unique exceptional character satisfying (\ref{E:A}).
\end{enumerate}
\end{thm}

%%%%%%%%%%%%%%%%%%%%%%%%%%%%%%%%%%%%%%%%%%
\vskip 15pt
\section{The $\TC_r, r\ge 2$ case}
Consider the Dynkin diagram for the simple coroots for $\TC_r$:

$$ \qquad 
\begin{picture}(4.7,0.2)(0,0)
\put(1,0){\circle{0.08}}
\put(1.5,0){\circle{0.08}}
\put(2,0){\circle{0.08}}
\put(2.5,0){\circle{0.08}}
\put(3,0){\circle{0.08}}
\put(1.04,0){\line(1,0){0.42}}
\multiput(1.55,0)(0.05,0){9}{\circle*{0.02}}
\put(2.04,0){\line(1,0){0.42}}
\put(2.54,0.015){\line(1,0){0.42}}
\put(2.54,-0.015){\line(1,0){0.42}}
\put(2.74,-0.04){$>$}
\put(1,0.1){\footnotesize $\alpha_1^\vee$}
\put(1.5,0.1){\footnotesize $\alpha_2^\vee$}
\put(2,0.1){\footnotesize $\alpha_{r-2}^\vee$}
\put(2.5,0.1){\footnotesize $\alpha_{r-1}^\vee$}
\put(3,0.1){\footnotesize $\alpha_r^\vee$}
\end{picture}
$$
\vskip 10pt

Let $Y=Y^{\sct}=\langle \alpha_1^\vee, \alpha^\vee_2, ..., \alpha_{r-1}^\vee, \alpha_r^\vee \rangle$ be the cocharacter lattice of $\Sp_{2r}$, where $\alpha_r^\vee$ is the short coroot. Let $Q$ be the Weyl-invariant quadratic on $Y$ such such $Q(\alpha_r^\vee)=1$. Then the bilinear form $B_Q$ is given by
$$
B_Q(\alpha_i^\vee, \alpha_j^\vee) =
\begin{cases}
2 & \text{if } i=j=r; \\
4& \text{if } 1\le i=j \le r-1; \\
-2 & \text{if } j=i+1;\\
0 & \text{if $\alpha_i^\vee, \alpha_j^\vee$ are not adjacent}.
\end{cases}
$$
A simple computation gives:
$$Y_{Q,n}=\set{\sum_{i=1}^n x_i \alpha_i^\vee: n| (2x_i) }.$$
We write $n_2:=n/\text{gcd}(2, n)$. Then 
$$Y_{Q,n}=\langle n_2 \alpha_1^\vee, n_2\alpha_2^\vee, ..., n_2 \alpha_{r-1}^\vee, n_2\alpha_r^\vee \rangle$$
 and $$Y_{Q,n}^{\sct}=\langle n_2 \alpha_1^\vee, n_2\alpha_2^\vee, ..., n_2 \alpha_{r-1}^\vee, n\alpha_r^\vee \rangle.$$

The map $\ii_\TC: \bigoplus_{i=1}^r \Z \alpha_i^\vee \to \bigoplus_{i=1}^r \Z e_i$ is given by
$$\ii_\TC: (x_1, x_2, x_3 ..., x_r) \mapsto (x_1, x_2-x_1, x_3-x_2, ..., x_{r-1} -x_{r-2}, x_r -x_{r-1}).$$
Here $\ii_\TC$ is an isomorphism. The Weyl group is $W= S_r \rtimes (\Z/2\Z)^r$, where $S_r$ is the permutation group on $\bigoplus_i \Z e_i$ and each $(\Z/2\Z)_i$ acts by $e_i \mapsto \pm e_i $. In particular, $\w_{\alpha_i}, 1\le i\le r-1,$ acts on $(y_1, y_2, ..., y_r) \in \bigoplus_i \Z e_i$ by exchanging $y_i$ and $y_{i+1}$, while $\w_{\alpha_r}$ acts by $(-1)$ on $\Z e_r$.

Moreover, $y\in Y$ lies in $Y_{Q,n}$ if and only all entries of $\ii_\TC(y)$ are divisible by $n_2$. It is easy to obtain
$$
Y_{Q,n}=
\left\{
\begin{array}{cc}
(y_1, y_2, ..., y_r) \in \bigoplus_{i=1}^{r} \Z e_i: \\
\bullet \ \ n_2 | y_i \text{ for all } i.
%\bullet \ \  y_1\equiv ... \equiv y_r \equiv y_{r_1} \text{ mod } n
\end{array}\right\}
\text{ and }
Y_{Q,n}^{\sct}=
\left\{
\begin{array}{cc}
(y_1, y_2, ..., y_r) \in \bigoplus_{i=1}^{r} \Z e_i: \\
\bullet \ \ n_2 | y_i \text{ for all } i, \\
\bullet \ \ n| \sum_i y_i.
\end{array}\right\}
$$
We further note
$$ 2\rho=\sum_{i=1}^r (2r-2i +1)e_i =\sum_{i=1}^r i(2r-i) \alpha_i^\vee.$$
Assume $x_0=0$, then
\begin{align*}
\ii_\TC(y_\rho) = \big( x_i -x_{i-1} - (r-i + 1/2) \big)_{1\le i\le r}. 
%\big( x_1 - (r-1), x_2-x_1 - (r-2), ..., x_i-x_{i-1} - (r-i), ..., x_r - x_{r-1} - (r-r) \big) \\
%& \qquad - (1/2, 1/2, ..., 1/2)
\end{align*}
Write $x_i^*:=x_i - x_{i-1} - (r-i)$, and also $\ii^*_\TC(y_\rho):=(x_1^*, x_2^*, ..., x_{r-1}^*, x_r^*)$. Then
$$\ii_\TC(y_\rho) = \ii^*_\TC(y_\rho) - (1/2, 1/2, ..., 1/2, 1/2).$$
We will discuss the two cases depending on the parity of $n$ separately.

\subsection{For $n$ odd} In this case, $n_2=n$ and
$$nY=Y_{Q,n}^{\sct}=Y_{Q,n}=\set{(y_1, ..., y_r) \in \bigoplus_{i=1}^r \Z e_i: n| y_i \text{ for all } i}.$$
The complex dual group for $\wt{\Sp}_{2r}^{(n)}$ for $n$ odd is $\text{SO}_{2r+1}$.
%\begin{prop}
%Any Brylinski-Deligne covering $\wt{Sp}_{2r}^{(n)}$ is $Y_{Q,n}^{sc}$-benign. Consequently, 
%$$\text{dim} \Wh(J(\wchi)) = \val{ \wp_{Q,n}(\OF_{Q,n, sc})}.$$
%\end{prop}
%\begin{proof}
%We first show $\wt{\Sp}_{2r}^{(n)}$ is $Y_{Q,n}^{sc}$-benign. We assume $w[y] \equiv y \text{ mod } Y_{Q,n}^{sc}$ for some $w\in W$. That is, $\ii_\TC(w(y_\rho))- \ii_\TC(y_\rho) \in Y_{Q,n}^{sc}$.
%
%If $\ii_\TC(y_\rho)=(x_1^* -1/2, x_2^* -1/2, ..., x_i^*-1/2, ..., x_r^* -1/2)$ is such that $x_i^* \equiv x_j^* \text{ mod } n$ for some $i\ne j$, then we see there exists $w', w_\alpha \in S_r, \alpha\in \Delta$ such that $w_\alpha (w'[y]) \equiv w'[y] \text{ mod } Y_{Q,n}^{sc}$.
%
%Suppose $n\nmid (x_i^* -x_j^*)$ for all $i\ne j$. Then since $\ii_\TC(w(y_\rho)) - \ii_\TC(y_\rho) \in Y_{Q,n}^{sc}$, there must exists $i, j$ such that $n| (x_j^* -1/2) + (x_i^* -1/2)$, i.e. $n| (x_j^*+x_i^* -1)$. 
%
%If $i\ne j$, there exists $w'\in W$ such that
%$$\ii_\TC(w'(y_\rho))=\big( x_i^*-1/2, -(x_j^*-1/2), ...\big).$$
%Then $\ii_\TC(w_{\alpha_1}w'(y_\rho))=\big(-(x_j^*-1/2), x_i^*-1/2, ...\big)$. It follows the entries of $\ii_\TC(w_{\alpha_1}w'(y_\rho)) - \ii_\TC(w'(y_\rho))$ are divisible by $n$. That is, $w_{\alpha_1}w'[y] \equiv w'[y] \text{ mod } Y_{Q,n}^{sc}$. If $i=j$, then we can find $w'$ such that $\ii_\TC(w'(y_\rho))=\big(...,  x_i^*-1/2\big)$. Then $w_{\alpha_r} w'[y] \equiv w'[y] \text{ mod } Y_{Q,n}^{sc}$. 
%
%Therefore, in any case, $\wt{\Sp}_{2r}^{(n)}$ is $Y_{Q,n}^{sc}$-benign.
%\end{proof}
%%%%%%%%%%%%

\begin{prop} \label{P:Sp-odd}
Let $n$ be an odd number, one has 
$$ 
\begin{cases}
\val{ \wp_{Q,n}(\OF_{Q,n, \sct})} \ge 2 & \text{ if } n\ge 2r+3; \\
\val{ \wp_{Q,n}(\OF_{Q,n, \sct})} =1 & \text{ if } n=2r+1;\\
\val{ \wp_{Q,n}(\OF_{Q,n, \sct})} =0 & \text{ if } n\le 2r-1.
\end{cases}$$
Therefore, for $n$ odd, we have $\text{dim} \Wh(\Theta(\wt{\Sp}_{2r}^{(n)}, \wchi)) =1$ if and only if $n=2r+1$ for the only exceptional character of $\wt{\Sp}_{2r}^{(2r+1)}$.
\end{prop}
\begin{proof}
We have written 
$$\ii_\TC(y_\rho) = \ii^*_\TC(y_\rho) - (1/2, 1/2, ..., 1/2, 1/2).$$
Since $x_1, ..., x_r$ are arbitrary, the associated $x_i^*$'s are also arbitrary. 

First, when $n\ge 2r+3$, consider the orbits $\mca{O}_y$ and $\mca{O}_{y'}$ where
$$\ii_\TC^*(y_\rho)=(1, 2, ..., r-1, r) \text{ and } \ii_\TC^*(y'_\rho)=(1, 2, ..., r-1, r+1).$$
If $r=2$, consider $\mca{O}_y$ and $\mca{O}_{y'}$ with $\ii_\TC^*(y_\rho)=(1, 2)$ and $\ii_\TC^*(y'_\rho)=(1, 3)$. Both $\mca{O}_y, \mca{O}_{y'}$ are $Y_{Q,n}$-free orbits. For example, for $\mca{O}_y$, this follows from the fact that the entries of $\ii_\TC(\w(y_\rho))- \ii_\TC(y_\rho)$ are either $j-i$ or $j+i-1$, for $0\le i, j \le r-1$. One can check also that $\wp_{Q,n}(\mca{O}_y)\ne \wp_{Q,n}(\mca{O}_{y'})$, and therefore $\val{ \wp_{Q,n}(\OF_{Q,n})} \ge 2$.

Second, assume that $n=2r+1$. Consider $\mca{O}_y$ such that $\ii_\TC^*(y_\rho)=(1, 2, ..., r-1, r)$. For $r=2$, consider $\ii_\TC^*(y_\rho)=(1, 2)$. It can be checked easily that $\wp_{Q,n}(\OF_{Q,n})=\set{\wp_{Q,n}(\mca{O}_y)}$. Thus, $\text{dim} \Wh(\Theta(\wt{\Sp}_{2r}^{(2r+1)}, \wchi)) =1$.

Third, assume that $n\le 2r-1$, we want to show that $\OF_{Q,n, \sct}=\emptyset$. If $\ii_\TC^*(y_\rho)=(x_1^*, x_2^*, ..., x_i^*, ..., x_r^*)$ is such that $x_i^* \equiv x_j^* \text{ mod } n$ for some $i\ne j$, then clearly $\mca{O}_y \notin \OF_{Q,n,\sct}$.
Now if $n\nmid (x_i^* -x_j^*)$ for all $i\ne j$; since $n\le 2r-1$, it is not hard to see that there always exist $i, j$ such that $n| (x_j^* -1/2) + (x_i^* -1/2)$, i.e. $n| (x_j^*+x_i^* -1)$. In this case, one also has $\mca{O}_y \notin \OF_{Q,n,\sct}$. In any case, $\OF_{Q,n, \sct}=\emptyset$ for $n\le 2r-1$. The proof is completed.
\end{proof}

%%%%%%
\subsection{For $n$ even} Write $n=2m$, in this case,
$$Y_{Q,n}=\langle m \alpha_1^\vee, m\alpha_2^\vee, ..., m \alpha_{r-1}^\vee, m\alpha_r^\vee \rangle, \quad Y_{Q,n}^{\sct}=\langle m \alpha_1^\vee, m\alpha_2^\vee, ..., m \alpha_{r-1}^\vee, n\alpha_r^\vee \rangle. $$

Equivalently, one has:
$$
Y_{Q,n}=
\left\{
\begin{array}{cc}
(y_1, y_2, ..., y_r) \in \bigoplus_{i=1}^{r} \Z e_i: \\
\bullet \ \ m | y_i \text{ for all } i.
%\bullet \ \  y_1\equiv ... \equiv y_r \equiv y_{r_1} \text{ mod } n
\end{array}\right\}
\text{ and }
Y_{Q,n}^{\sct}=
\left\{
\begin{array}{cc}
(y_1, y_2, ..., y_r) \in \bigoplus_{i=1}^{r} \Z e_i: \\
\bullet \ \ m | y_i \text{ for all } i, \\
\bullet \ \ n| \sum_i y_i.
\end{array}\right\}
$$
The dual group for $\wt{\Sp}_{2r}^{(n)}$ with $n$ even is $\Sp_{2r}$.

\subsubsection{For $m\ge 2r+2$} 
In this case, consider the orbits $\mca{O}_y, \mca{O}_{y'}$ given in the proof of Proposition \ref{P:Sp-odd}. They are $Y_{Q,n}$-free; moreover, $\mca{O}_y$ and $\mca{O}_{y'}$ are distinct in the image of $\wp_{Q,n}$. Thus, we have $\val{\wp_{Q,n} (\OF_{Q,n})}\ge 2$. 

\subsubsection{For $m\le 2r-2$} We can check easily in this case $\OF_{Q,n,\sct}=\emptyset$.

%%%%%%%%%%%%%%%%%%%%%%%%%
\subsubsection{For $m=2r-1$} In this case, consider $y$ with $\ii^*_\TC(y_\rho)=(1, 2, ..., r-1, r)$, i.e.
$$\ii_\TC(y_\rho)=(1-1/2, 2-1/2, ..., (r-1)-1/2, r-1/2).$$
Consider $\w_{\alpha_r} \in W$, then $\ii_\TC(\w_{\alpha_r}(y_\rho))=(1-1/2, 2-1/2, ..., (r-1)-1/2, -(r-1/2))$.

Note that $\mca{O}_y$ is $Y_{Q,n}^{\sct}$-free, and $\wp_{Q,n}^{\sct}(\OF_{Q,n,\sct})=\set{\wp_{Q,n}^{\sct}(\mca{O}_y)}=\set{\wp_{Q,n}^{\sct}(\mca{O}_0)}$. However, it is not $Y_{Q,n}$-free, since $\ii_\TC(y_\rho -\w_{\alpha_r}(y_\rho))=(0, 0, ..., m) \in Y_{Q,n}$. Recall that any $\cc\in \Ftn(i(\wchi))$ which gives rise to $\lambda^{\wchi}_{\cc}\in \Wh(\Theta(\wt{G}, \wchi))$ satisfies $\cc (\s_{ \w_{\alpha_r}[y]}) = \mathbf{t}(\w_{\alpha_r}, y) \cdot \cc(\s_{y})$ where
$$\mathbf{t}(\w_{\alpha_r}, y)=q^{k_{y,\alpha_r} -1} \cdot \GGMA(y, \alpha_r^\vee) \cdot \g(Q(\alpha_r^\vee)\cdot \angb{y_\rho}{\alpha_r})^{-1}.$$
Meanwhile, in our case $\w_{\alpha_r}[y]-y=(-m)\alpha_r^\vee \in Y_{Q,n}$. It follows that
$$\cc(\s_{\w_{\alpha_r}[y]}) =\vep^{D(\w_{\alpha_r}(y_\rho)-y_\rho, y)} \cdot \wchi(\s_{\w_{\alpha_r}(y_\rho)-y_\rho}) \cdot \cc(\s_y).$$
For $\cc$ to be nonzero on $\mca{O}_y$, i.e., $\wp_{Q,n}(\mca{O}_y)$ contributes to the right hand side of (\ref{dWh}), one has
$$\wchi(\s_{-m\alpha_r^\vee})=q^{k_{y,\alpha_r} -1} \cdot \g(Q(\alpha_r^\vee)\cdot \angb{y_\rho}{\alpha_r})^{-1}.$$
Moreover, we can argue as in \S \ref{S:A-delic} that this condition is also sufficient. One has $\angb{y}{\alpha_r}=2r$ and thus $k_{y, \alpha_r}=1$. The equality is thus simplified to
\begin{equation} \label{E:C}
\wchi(\s_{-m\alpha_r^\vee})=\g(m)^{-1}.
\end{equation}

Consider the exceptional character $\wchi_{\psi'}=\wchi_{\psi'}^0 \cdot \delta_B^{1/2n}$, which relies on the distinguished unitary character  $\wchi_{\psi'}^0$ depending on a non-trivial character $\psi': F\to \C^\times$ (see \S \ref{S:UDC}). Since $\wchi^0_{\psi'}(\s_{m\alpha_r^\vee})=\ggma_{\psi'}(\varpi)^{mQ(\alpha_r^\vee)}$, by Lemma \ref{L:G-W}, the equality (\ref{E:C}) becomes $\ggma_{\psi}(\varpi)=(-1, \varpi)_n^{m^2} \cdot \ggma_{\psi'}(\varpi)^{-m}$, which can be further reduced to
$$\ggma_{\psi'}(\varpi)=(-1, \varpi)_n^{r+1} \cdot \ggma_{\psi}(\varpi)=(-1, \varpi)_2^{r+1} \cdot \ggma_{\psi}(\varpi).$$

In particular, if $\psi'=\psi_a$ with $a\in O_F^\times$, then the equality is equivalent to $(a(-1)^{r+1}, \varpi)_2=-1$, i.e. $a\in (-1)^{r+1}\cdot (O_F^\times)^2$.

%%%%%%%%%%%%%%%%%%%%%%%
\subsubsection{For $m=2r$} We claim that in this case $\OF_{Q,n} =\OF_{Q,n,\sct}$. Clearly it suffices to show that $\OF_{Q,n} \supseteq \OF_{Q,n,\sct}$. Equivalently, if $\mca{O}_y$ is not $Y_{Q,n}$-free, we would like to show that it is not $Y_{Q,n}^{\sct}$-free. Write $\ii_\TC^*(y_\rho)=(x_1^*, x_2^*, ..., x_r^*)$. By assumption, $\ii_\TC(y-\w[y])=\ii_\TC^*(y_\rho-\w(y_\rho))\in Y_{Q,n}$ for some $\w\in W$. Entries of $\ii_\TC(y-\w[y])$ can not be of the form $2x_i^*-1$ since $m$ is even; therefore they are of the form $0, x_i^*-x_j^*$ or $x_i^*+x_j^*-1$ for $i\ne j$. In this case, it is easy to see that $\ii_\TC(y-\w'[y])\in Y_{Q,n}^{\sct}$ for some $\w'\in W$, i.e., $\mca{O}_y$ is not $Y_{Q,n}^{\sct}$-free.

Consequently, $\dim \Wh(\Theta(\wt{\Sp}_{2r}^{(4r)}, \wchi))=\val{\wp_{Q,n}(\OF_{Q,n})}$. On the other hand, consider $\mca{O}_y$ with $\ii_\TC^*(y_\rho)=(1, ..., r-1, r)$. It is not hard to see that $\wp_{Q,n}(\OF_{Q,n})=\set{ \wp_{Q,n}(\mca{O}_y)}$. Therefore, we always have $\dim \Wh(\Theta(\wt{\Sp}_{2r}^{(4r)}, \wchi))=1$ for any of the two exceptional characters of $\wt{\Sp}_{2r}^{(4r)}$.
%%%%%%%%%%%%%%%%%%%%%%%%

\subsubsection{For $m= 2r+1$} \label{S:C-ineq}
In this case, consider $\mca{O}_y$ with $\ii^*_\TC(y_\rho)=(1, 2, ..., r-1, r)$. It can be checked that $\wp_{Q, n}(\OF_{Q,n})=\set{\wp_{Q,n}(\mca{O}_y)}$ with $\mca{O}_y \in \OF_{Q,n}$, i.e. $\val{\wp_{Q, n}(\OF_{Q,n})}=1$. On the other hand, $\wp_{Q, n}(\OF_{Q,n, \sct})=\set{\wp_{Q,n}(\mca{O}_y)} \cup \set{ \wp_{Q,n}(\mca{O}_{z_i}): 1\le i\le r}$ with $z_i$ described as follows. Recall that we write $z_{i,\rho}:=z_i -\rho$. For $1\le i\le r-1$, $z_i$ is such that $\ii_\TC^*(z_{i,\rho})=(0, 2, 3, ..., \widehat{i+1}, ..., r, r+1)$, which denotes the $r$-tuple obtained from the $r+1$-tuple $(0, 2, 3, ..., r-1, r, r+1)$ by removing the entry $i+1$. Meanwhile, $z_r$ is such that $\ii_\TC^*(z_{r,\rho})=(2, 3, ..., r-1, r, r+1)$. 

Note that $\mca{O}_{z_i} \in \OF_{Q,n,\sct}\backslash \OF_{Q,n}$, since 
$$\ii_\TC(\w_{\alpha_r}[z_i]-z_i)=\ii_\TC(\w_{\alpha_r}(z_{i,\rho})-z_{i,\rho})=-(0, 0, ..., 0, m)=\ii_\TC(-m\alpha_r^\vee) \in Y_{Q,n}.$$
The elements $\wp_{Q,n}(\mca{O}_y)$ and $\wp_{Q,n}(\mca{O}_{z_i})$'s are all distinct. It follows that $\val{\wp_{Q, n}(\OF_{Q,n, \sct})}= r+1$. Therefore,
$$1 \le \dim \Wh(\Theta(\wt{\Sp}_{2r}^{(4r+2)}, \wchi)) \le r+1.$$
However, as there are only two exceptional characters $\wchi$, the dimension $\Wh(\Theta(\wt{\Sp}_{2r}^{(4r+2)}, \wchi))$ can take at most two values. In fact, we will determine completely the value and its dependence on $\wchi$.

\begin{prop}  \label{P:C:4r+2}
Let $\wchi$ be an exceptional character of $\wt{\Sp}_{2r}^{(4r+2)}$. Then
$$
\dim \Wh(\Theta(\wt{\Sp}_{2r}^{(4r+2)}, \wchi)) =
\begin{cases}
1 & \text{ if }\  \wchi(\s_{-m\alpha_r^\vee}) = -q^{1/2} \cdot \ggma_\psi(\varpi),\\
r+1 & \text{ if }\  \wchi(\s_{-m\alpha_r^\vee}) = q^{1/2} \cdot \ggma_\psi(\varpi).
\end{cases}
$$
\end{prop}
\begin{proof}
First, we show that if $\wchi$ is an exceptional character, then $\wchi(\s_{-m\alpha_r^\vee})=\pm q^{1/2} \cdot \ggma_\psi(\varpi)$.  Consider
\begin{align*}
&\wchi(\s_{-m\alpha_r^\vee})^2\\
=&\wchi(\s_{-n\alpha_r^\vee})\cdot \vep^{m^2Q(\alpha_r^\vee)} \\
=& \wchi(\s_{n\alpha_r^\vee})^{-1}\cdot \vep \\
=& q\cdot (-1, \varpi)_2,
\end{align*}
which has square roots exactly $\pm q^{1/2} \cdot \ggma_\psi(\varpi)$. That is, an exceptional character $\wchi$ of $\wt{\Sp}_{2r}^{(4r+2)}$ is uniquely determined by the sign.

Second, arguing as in \S \ref{S:A-delic} , we see that $\wp_{Q,n}(\mca{O}_{z_i}), 1\le i\le r$ contributes to the right hand side of the Equality (\ref{dWh}) if and only if (as in Equality (\ref{E:gen-conA}))
\begin{equation} \label{Cri:4r+2}
\wchi(\s_{\w_{\alpha_r}[z_i]-z_i})=\vep^{D(\w_{\alpha_r}[z_i]-z_i, z_i)} \cdot  \mathbf{t}(\w_{\alpha_r}, z_i).
\end{equation}
That is, $\dim \Wh(\Theta(\wt{\Sp}_{2r}^{(4r+2)}, \wchi))=1+\val{ \set{z_i: \text{ the equality (\ref{Cri:4r+2}) holds for } z_i}}$.
Note that, $\w_{\alpha_r}[z_i]-z_i=-m\alpha_r^\vee$ for all $i$. On the other hand, we claim that the right hand side of (\ref{Cri:4r+2}) is independent of $i$. A simple computation gives $\angb{z_{i,\rho}}{\alpha_r}=m$ and therefore
\begin{align*}
& \vep^{D(\w_{\alpha_r}[z_i]-z_i, z_i)} \cdot  \mathbf{t}(\w_{\alpha_r}, z_i) \\
=& \vep^{D(\alpha_r^\vee, z_i)} \cdot  q^{\ceil{\frac{\angb{z_{i,\rho}}{\alpha_r}+1}{n_{\alpha_r}}}-1} \cdot \vep^{\angb{z_{i,\rho}}{\alpha_r} \cdot D(z_i, \alpha_r^\vee)} \cdot \g(\angb{z_{i,\rho}}{\alpha_r} \cdot Q(\alpha_r^\vee))^{-1} \\
=& \vep^{B_Q(z_i, \alpha_r^\vee)} \cdot q^{\ceil{\frac{m+1}{n}}-1} \cdot \g(m)^{-1} \\
=& \g(m)^{-1} \text{ by the evenness of $B_Q$}.
\end{align*}
Thus, it follows that $\dim \Wh(\Theta(\wt{\Sp}_{2r}^{(4r+2)}, \wchi))=1$ or $r+1$. Moreover, it is equal to $1$ if and only if $\wchi(\s_{-m\alpha_r^\vee}) \ne \g(m)^{-1}$. By Lemma \ref{L:G-W}, $\g(m)^{-1}=q^{1/2}\cdot \ggma_\psi(\varpi)$. Therefore, $\dim \Wh(\Theta(\wt{\Sp}_{2r}^{(4r+2)}, \wchi))=1$ (resp. $r+1$) if and only if $\wchi(\s_{-m\alpha_r^\vee})$ is equal to $-q^{1/2}\cdot \ggma_\psi(\varpi)$ (resp. $q^{1/2}\cdot \ggma_\psi(\varpi)$). The proof is completed.
\end{proof}

We summarize results in this section as follows.
\begin{thm} \label{T:C}
Consider the Brylinski-Deligne covering group $\wt{\Sp}_{2r}^{(n)}$ with $r\ge 2, n\ge 1$. Let $\wchi$ be an unramified exceptional character, then $\dim \Wh(\Theta(\wt{\Sp}_{2r}^{(n)}, \wchi))=1$ if and only if the following holds:
\begin{enumerate}
\item[$\bullet$] $n=4r-2$ and $\wchi$ is the unique exceptional character satisfying (\ref{E:C}), or
\item[$\bullet$] $n=4r$ and $\wchi$ is any exceptional character of $\wt{\Sp}_{2r}^{(4r)}$, or
\item[$\bullet$] $n=4r+2$ and $\wchi$ is the unique exceptional character given in Proposition \ref{P:C:4r+2}, or
\item[$\bullet$] $n=2r+1$ and $\wchi$ is the only exceptional character of $\wt{\Sp}_{2r}^{(2r+1)}$.
\end{enumerate}

Moreover, consider the exceptional character $\wchi_{\psi_a}:=\wchi_{\psi_a}^0\cdot \delta_B^{1/2n}$ associated with $\psi_a$. Assume that $\psi_a$ has conductor $O_F$, i.e. $a\in O_F^\times$. Then, one has $\dim \Wh(\Theta(\wt{\Sp}_{2r}^{(4r-2)}, \wchi_{\psi_a}))=1$ if and only if $a\in (-1)^{r+1} \cdot (O_F^\times)^2$; also, $\dim \Wh(\Theta(\wt{\Sp}_{2r}^{(4r+2)}, \wchi_{\psi_a}))=1$ if and only if $a\in (-1)^{r} \cdot (O_F^\times)^2$.
\end{thm}

\vskip 15pt 
%%%%%%%%%%%%%%%%%%%%%%%%%%%%%%%%%%%%%%%%%%
\section{The $\TB_r, r\ge 2$ case}
Consider the Dynkin diagram for the simple coroots for the group $\Spin_{2r+1}$ of type $\TB_r$:

$$ \qquad 
\begin{picture}(4.7,0.2)(0,0)
\put(1,0){\circle{0.08}}
\put(1.5,0){\circle{0.08}}
\put(2,0){\circle{0.08}}
\put(2.5,0){\circle{0.08}}
\put(3,0){\circle{0.08}}
\put(1.04,0){\line(1,0){0.42}}
\multiput(1.55,0)(0.05,0){9}{\circle*{0.02}}
\put(2.04,0){\line(1,0){0.42}}
\put(2.54,0.015){\line(1,0){0.42}}
\put(2.54,-0.015){\line(1,0){0.42}}
\put(2.74,-0.04){$<$}
\put(1,0.1){\footnotesize $\alpha_1^\vee$}
\put(1.5,0.1){\footnotesize $\alpha_2^\vee$}
\put(2,0.1){\footnotesize $\alpha_{r-2}^\vee$}
\put(2.5,0.1){\footnotesize $\alpha_{r-1}^\vee$}
\put(3,0.1){\footnotesize $\alpha_r^\vee$}
\end{picture}
$$
\vskip 10pt

Let $Y=\langle \alpha_1^\vee, \alpha^\vee_2, ..., \alpha_{r-1}^\vee, \alpha_r^\vee \rangle$ be the cocharacter lattice of $\text{Spin}_{2r+1}$, where $\alpha_r^\vee$ is the long coroot. Let $Q$ be the Weyl-invariant quadratic form on $Y$ such that $Q(\alpha_r^\vee)=2$, i.e. $Q(\alpha_i^\vee)=1$ for $1\le i\le r-1$. Then the bilinear form $B_Q$ is given by
$$
B_Q(\alpha_i^\vee, \alpha_j^\vee) =
\begin{cases}
4 & \text{if } i=j=r; \\
2& \text{if } 1\le i=j \le r-1; \\
-1 & \text{if } 1\le i \le r-2 \text{ and } j=i+1;\\
-2 & \text{if } i=r-1, j=r;\\
0 & \text{if $\alpha_i^\vee, \alpha_j^\vee$ are not adjacent}.
\end{cases}
$$

The map $\ii_\TB: \bigoplus_{i=1}^{r} \Z \alpha_i^\vee \to \bigoplus_{i=1}^r \Z e_i$ is given by
$$\ii_\TB: (x_1, x_2, x_3 ..., x_r) \mapsto (x_1, x_2-x_1, x_3-x_2, ..., x_{r-1} -x_{r-2}, 2x_r -x_{r-1}).$$
In particular, any $(y_1, ..., y_r) \in \bigoplus_{i=1}^r \Z e_i$ is equal to $\ii_\TB(y)$ for some $y$ if and only if $2| (\sum_i y_i)$.

The Weyl group is $W= S_r \rtimes (\Z/2\Z)^r$, where $S_r$ is the permutation group on $\bigoplus_i \Z e_i$ and $(\Z/2\Z)_i: e_i \mapsto \pm e_i $. In particular, $\w_{\alpha_i}, 1\le i\le r-1,$ acts on $(y_1, y_2, ..., y_r) \in \bigoplus_i \Z e_i$ by exchanging $y_i$ and $y_{i+1}$. Also, $\w_{\alpha_r}$ acts by $(-1)$ on $\Z e_r$. 

A simple computation gives:
\begin{equation*}
Y_{Q,n}=
\left\{
\begin{array}{cc}
(y_1, y_2, ..., y_r) \in \bigoplus_{i=1}^r \Z e_i: \\
\bullet \ 2| (\sum_{i=1}^r y_i), \\
\bullet \ y_1 \equiv ... \equiv y_r \text{ mod } n\\
\bullet \ n| 2y_i \text{ for all } i.
\end{array}\right\} ,
\quad
Y_{Q,n}^{\sct}=
\left\{
\begin{array}{cc}
(y_1, y_2, ..., y_r) \in \bigoplus_{i=1}^r \Z e_i: \\
\bullet \ 2| (\sum_{i=1}^r y_i), \\
\bullet \ n|y_i \text{ for all } i.
\end{array}\right\}
\end{equation*}
In particular, if $n$ is odd, then $Y_{Q,n}=Y_{Q,n}^{\sct}$. 

We note that $2\rho=\sum_{i=1}^r 2(r-i +1)e_i $, and therefore $\rho=\sum_{i=1}^r (r-i +1)e_i $.
If $y=(x_1, x_2, ..., x_r) \in \bigoplus_i \Z \alpha_i^\vee$, then
\begin{align*}
\ii_\TB(y_\rho) =&  \big(x_1 - (r-1+1), x_2-x_1 -(r-2+1), ..., x_i -x_{i-1} -(r-i+1), ..., \\
& \qquad \qquad ..., x_{r-1} - x_{r-2} -(r-(r-1)+1), 2x_r - x_{r-1} - (r-r+1)\big) \\
:= & (x_1^*, x_2^*, ..., x_{r-1}^*, x_r^*).
\end{align*}
Any $(x_1^*, ..., x_r^*) \in \bigoplus_i \Z e_i$ such that $2|\big(\sum_i x_i^* + r(r+1)/2\big)$ is equal to $\ii_\TB(y_\rho)$ for some $y$.

%%%%%%%%%%%%%%%%%%%%%%%%%%%%
\subsection{For $n$ odd} 
In this case,
$$nY=Y_{Q,n}^{\sct}=Y_{Q,n}.$$
Therefore, $\text{dim} \Wh(\Theta(\wt{\Spin}_{2r+1}^{(n)}, \wchi)) = \val{ \wp_{Q,n}(\OF_{Q,n, \sct})}$, where $\wchi$ is the only exceptional character of $\wt{\Spin}_{2r+1}^{(n)}$. For $n$ odd, the dual group for $\wt{\Spin}_{2r+1}^{(n)}$ is $\text{PGSp}_{2r}$.

\begin{prop} \label{P:Spin-odd}
Let $n$ be an odd number, one has 
$$ 
\begin{cases}
\val{ \wp_{Q,n}(\OF_{Q,n, \sct})} \ge 2 & \text{ if } n\ge 2r+3; \\
\val{ \wp_{Q,n}(\OF_{Q,n, \sct})} =0 & \text{ if } n\le 2r-1;\\
\val{ \wp_{Q,n}(\OF_{Q,n, \sct})} =1 & \text{ if } n= 2r+1.
\end{cases}$$
Therefore, when $n$ is odd, we have $\text{dim} \Wh(\Theta(\wt{\Spin}_{2r+1}^{(n)}, \wchi))=1$ if and only if $n=2r+1$.
\end{prop}
\begin{proof}
First, assume that $n\ge 2r+3$. We write
$$\ii_\TB(y_\rho)=(x_1^*, x_2^*, ..., x_i^*, ..., x_r^*) \text{ with } 2| \big(\sum_{i=1}^r x_i^* + r(r+1)/2\big).$$
For $r\ge 3$, let $y\in Y$ (resp. $y'$) be such that $\ii_\TB(y_\rho)=(1, 2, 3, ..., r-2, r-1, r)$  (resp. $\ii_\TB(y'_\rho)=(1, 2, ..., r-2, r, r+1)$). For $r=2$, we take $(x_1^*, x_2^*)=(1, 2)$ or $(2, 3)$, and let $y$ and $y'$ be the associated element in $Y$ respectively. In any case, the two orbits $\mca{O}_y$ and $\mca{O}_{y'}$ are $Y_{Q,n}$-free. Moreover, $\wp_{Q,n}(\mca{O}_y) \ne \wp_{Q,n}(\mca{O}_{y'})$. Thus, for $n\ge 2r+3$, one has $\val{ \wp_{Q,n}(\OF_{Q,n, \sct})} \ge 2$.

Second, assume that $n\le 2r-1$, we want to show that $\OF_{Q,n, \sct}=\emptyset$. If $\ii_\TB(y_\rho)=(x_1^*, x_2^*, ..., x_i^*, ..., x_r^*)$ is such that $x_i^* \equiv x_j^* \text{ mod } n$ for some $i\ne j$, then clearly $\mca{O}_y \notin \OF_{Q,n,\sct}$. Suppose $n\nmid (x_i^* -x_j^*)$ for all $i\ne j$, since $n\le 2r-1$, it is not hard to see that there always exist $i, j$ such that $n| (x_j^*+x_i^*)$. That is, $\mca{O}_y \notin \OF_{Q,n,\sct}$ for any $\mca{O}_y$.

Third, if $n=2r+1$, consider the orbit $\mca{O}_y$ with 
$$\ii_\TB(y_\rho)=(x_1^*, x_2^*, ..., x_{r-1}^*, x_r^*)=(1, 2, 3, ..., r-2, r-1, r).$$
(For $r=2$, consider $\ii_\TB(y_\rho)=(1, 2)$.) One has $\wp_{Q,n}(\OF_{Q,n,\sct}) = \set{\wp_{Q,n}(\mca{O}_y)}$, and therefore $\val{ \wp_{Q,n}(\OF_{Q,n,\sct})} =1$ for $n= 2r+1$.
\end{proof}

%%%%%%
\subsection{For $n$ even} Write $n=2m$. In this case,
$$Y=\set{(y_1, y_2, ..., y_r) \in \bigoplus_i \Z e_i: 2| (\sum_{i=1}^r y_i)}.$$
Moreover,
\begin{equation*}
Y_{Q,n}=
\left\{
\begin{array}{cc}
(y_1, y_2, ..., y_r) \in \bigoplus_i \Z e_i: \\
\bullet \ 2| (\sum_{i=1}^r y_i), \\
\bullet \ y_i=k_in + m \text{ for all } i\\
\quad  \text{ or } y_i=k_i n \text{ for all } i.
\end{array}\right\} ,
\quad
Y_{Q,n}^{\sct}=
\left\{
\begin{array}{cc}
(y_1, y_2, ..., y_r) \in \bigoplus_i \Z e_i: \\
%\bullet \ 2| (\sum_{i=1}^r y_i), \\
\bullet \ n| y_i \text{ for all } i.
\end{array}\right\}
\end{equation*}

We see easily that for $y_i=k_i n + m$, one has $(y_1, y_2, ..., y_r) \in Y_{Q,n}$ if and only if $2| (rm)$. In fact, for $n$ even, the dual group for $\wt{\Spin}_{2r+1}^{(n)}$ is equal to $\text{SO}_{2r+1}$ if $m$ and $r$ are both odd; otherwise, the dual group is $\Spin_{2r+1}$, see \cite{We2}. We discuss case by case according to the parities of $r$ and $m$.

%First, we deal with the case $r=2$. We show that it agrees with the $\TB_2$ case discussed in $\S$.
%\subsubsection{For $r=2$} In this case,
%%%%%%%%%%
%Now, we assume $r\le 3$ and discuss the fours cases below.

\subsubsection{For $m$ odd and $r$ odd} In particular, one has $r\ge 3$. In this case, $Y_{Q,n}=Y_{Q,n}^{\sct}$, and $\wp_{Q,n}(\OF_{Q,n})=\wp_{Q,n}(\OF_{Q,n,\sct})$. Consider the following situations:

\begin{enumerate}
\item[$\bullet$] If $n> 2(r+1)$ (i.e. $m> r+1$ and therefore $m\ge r+2$), consider $y$ such that $\ii_\TB(y_\rho)=(x_1^*, x_2^*, ..., x_r^*)$ is equal to
$$(1, 2, ..., r-2, r-1, r) \text{ or } (1, 2, ..., r-2, r, r+1).$$
We can check the two orbits $\mca{O}_y$ for these two choices of $y$ are $Y_{Q,n}$-free, and moreover their images with respect to the map $\wp_{Q,n}$ are distinct in $\wp_{Q,n}(\OF_{Q,n})$. Thus, $\val{\wp_{Q,n}(\OF_{Q,n})} \ge 2$ in this case.
\item[$\bullet$] If $n< 2r$ (i.e. $m< r$ and therefore $m\le r-2$), in this case, one can check that $\wp_{Q,n}(\OF_{Q,n,\sct})=\emptyset$.
\item[$\bullet$] If $n=2r$ (note $n\ne 2(r+1)$), i.e. $m=r$. In this case, one can also check that $\wp_{Q,n}(\OF_{Q,n,\sct})=\emptyset$.
\end{enumerate}

Therefore, $\text{dim} \Wh(\Theta(\wt{\Spin}_{2r+1}^{(n)}, \wchi))\ne 1$ for both $r$ and $m$ odd.
%%%%%%%%%%

\subsubsection{For $m$ odd and $r\ge 2$ even} \label{S:B-o-e}

In this case, $Y_{Q,n}\ne Y_{Q,n}^{\sct}$. One has the following situations:
\begin{enumerate}
\item[$\bullet$] Assume $n> 2(r+1)$ (i.e. $m> r+1$ and thus $m\ge r+3$).

\cu{Case I} If $r\ge 3$, consider $y$ and $y'$ such that
$$\ii_\TB(y_\rho)=(1, 2, ..., r-2, r-1, r) \text{ and } \ii_\TB(y'_\rho)=(1, 2, ..., r-2, r, r+1).$$
We can check the orbits $\mca{O}_y, \mca{O}_{y'}$ are $Y_{Q,n}$-free and $\wp_{Q,n}(\mca{O}_y)\ne \wp_{Q,n}(\mca{O}_{y'})$. Thus, $\val{\wp_{Q,n}(\OF_{Q,n})} \ge 2$.

\cu{Case II} If $r= 2$ and $m\ge r+5$, consider $\mca{O}_y$ and $\mca{O}_{y'}$ with $\ii_\TB(y_\rho)=(1, 2)$ and $\ii_\TB(y'_\rho)=(2, 3)$. Then as in the preceding case, they are $Y_{Q,n}$-free and $\wp_{Q,n}(\mca{O}_y)\ne \wp_{Q,n}(\mca{O}_{y'})$. Thus, $\val{\wp_{Q,n}(\OF_{Q,n})} \ge 2$.

\cu{Case III} If $r=2$ and $m=5$, consider $\mca{O}_y$  with $\ii_\TB(y_\rho)=(1, 2)$. It is easy to check $\wp_{Q,n}(\OF_{Q,n})=\set{\wp_{Q,n}(\mca{O}_y)}$. On the other hand, let $z, z'$ be such that $\ii_\TB(z_\rho)=(1, 4)$ and $\ii_\TB(z'_\rho)=(2, 3)$. Then $\wp_{Q,n}(\OF_{Q,n,\sct})=\set{\wp_{Q,n}(\mca{O}_y), \wp_{Q,n}(\mca{O}_z), \wp_{Q,n}(\mca{O}_{z'})}$, a set of size 3. Note, $\mca{O}_z, \mca{O}_{z'} \in \OF_{Q,n,\sct} \backslash \OF_{Q,n}$. That is, $\val{\wp_{Q,n}(\OF_{Q,n})}=1$ and $\val{\wp_{Q,n}(\OF_{Q,n, \sct})} =3$ in this case.

Let $\w, \w'\in W$ be such that $\ii_\TB(\w[z]-z)=\ii_\TB(\w'[z']-z')=-(5, 5)\in Y_{Q,n}$. Write $y_{Q,n}=\ii_\TB(\w[z]-z) \in Y_{Q,n}$. Then, as in \S \ref{S:C-ineq}, $\dim \Wh(\Theta(\wt{\Spin}_{5}^{(10)}, \wchi))=1$ if and only if
\begin{equation} \label{B:2-5}
\wchi(\s_{y_{Q,n}})\ne \vep^{D(y_{Q,n}, z)}\cdot \mathbf{T}(\w, z) \text{ and } \wchi(\s_{y_{Q,n}})\ne \vep^{D(y_{Q,n}, z')}\cdot \mathbf{T}(\w', z').
\end{equation}
However, as in Proposition \ref{P:C:4r+2}, it can be checked easily that $\vep^{D(y_{Q,n}, z)}\cdot \mathbf{T}(\w, z)=\vep^{D(y_{Q,n}, z')}\cdot \mathbf{T}(\w', z')$, and the condition (\ref{B:2-5}) is equivalent to
\begin{equation} \label{B:ineq}
\wchi(\s_{-5\alpha_r^\vee})= -q^{1/2} \cdot \ggma_\psi(\varpi).
\end{equation}
This agrees with the result from Proposition (\ref{P:C:4r+2}) for the $\wt{\TC}_2^{(10)}$ case.
\item[$\bullet$] If $n< 2r$ (i.e. $m\le r$ and therefore $m\le r-1$), in this case, one can check $\wp_{Q,n}(\OF_{Q,n,\sct})=\emptyset$.
\item[$\bullet$] If $n=2(r+1)$ (note $n\ne 2r$), i.e. $r=m-1$. In this case, one can check $\wp_{Q,n}^{\sct}(\OF_{Q,n,\sct})=\set{\wp_{Q,n}^{\sct}(\mca{O}_0)}$ (and thus $\wp_{Q,n}(\OF_{Q,n,\sct})=\set{\wp_{Q,n}(\mca{O}_0)}$) is a singleton with
$$\ii_\TB(0_\rho)=(-r, -(r-1), ..., -2, -1).$$
That is, $\mca{O}_0$ is $Y_{Q,n}^{\sct}$-free. However, it is not $Y_{Q,n}$-free, since there exists $\w\in W$ such that $\ii_\TB(\w(0_\rho))=(1, 2, ..., r-1, r)$. It follows that
$$\ii_\TB(\w(0_\rho)-0_\rho)=(m, m, ..., m, m) \in Y_{Q,n}.$$
\end{enumerate}

Write $y_{Q,n}=\w(0_\rho)-0_\rho=\w[0]-0$. It follows from an analogous argument for Proposition \ref{P:cond-A} that $\text{dim} \Wh(\Theta(\wt{\Spin}_{2r+1}^{(2r+2)}, \wchi))= 1$ if and only if $\wchi$ is the unique exceptional character satisfying
\begin{equation} \label{E:B1}
\wchi(\s_{y_{Q,n}}) = \mathbf{T}(\w, 0).
\end{equation}
One can explicate the equality by computing the right hand side as in Lemma \ref{L:explicitA}. We omit the details here.
%%%%%
\subsubsection{For $m$ even and $r\ge 3$ odd} In this case, $Y_{Q,n}\ne Y_{Q,n}^{\sct}$. We have:
\begin{enumerate}
\item[$\bullet$] If $n> 2(r+1)$ (i.e. $m> r+1$ and therefore $m\ge r+3$), consider $y$ and $y'$ such that
$$\ii_\TB(y_\rho)=(1, 2, ..., r-2, r-1, r) \text{ and } \ii_\TB(y'_\rho)=(1, 2, ..., r-2, r, r+1).$$
We can check the orbits $\mca{O}_y, \mca{O}_{y'}$ are $Y_{Q,n}$-free and $\wp_{Q,n}(\mca{O}_y)\ne \wp_{Q,n}(\mca{O}_{y'})$. Thus, $\val{\wp_{Q,n}(\OF_{Q,n})} \ge 2$.
\item[$\bullet$] If $n< 2r$ (i.e. $m<r$ and therefore $m\le r-1$), in this case, one can check $\wp_{Q,n}(\OF_{Q,n,\sct})=\emptyset$.
\item[$\bullet$] If $n=2(r+1)$ (note $n\ne 2r$), i.e. $r=m-1$. In this case, $\wp_{Q,n}(\OF_{Q,n,\sct})=\set{\wp_{Q,n}(\mca{O}_y)}$ is a singleton with
$$\ii_\TB(0_\rho)=(-r, -(r-1), ..., -2, -1).$$
The situation is exactly as the third case of \S \ref{S:B-o-e}. That is, $\mca{O}_0$ is $Y_{Q,n}^{\sct}$-free but not $Y_{Q,n}$-free. Consider $\w\in W$ such that $\ii_\TB(\w(0_\rho))=(1, 2, ..., r-1, r)$ and
$$\ii_\TB(\w(0_\rho)-0_\rho)=(m, m, ..., m, m) \in Y_{Q,n}.$$
\end{enumerate}

Write $y_{Q,n}=\w(0_\rho)-0_\rho=\w[0]-0$. Then $\text{dim} \Wh(\Theta(\wt{\Spin}_{2r+1}^{(2r+2)}, \wchi))= 1$ if and only if $\wchi$ is the unique exceptional character satisfying
\begin{equation} \label{E:B2}
\wchi(\s_{y_{Q,n}}) = \mathbf{T}(\w, 0).
\end{equation}
%%%%%%%
%%%%%
\subsubsection{For $m$ even and $r\ge 2$ even} In this case, $Y_{Q,n}\ne Y_{Q,n}^{\sct}$. One has the following situations:
\begin{enumerate}
\item[$\bullet$] If $n> 2(r+1)$ (i.e. $m> r+1$ and therefore $m\ge r+2$).

\cu{Case I}, $r\ge 4$. Consider $y$ and $y'$ such that
$$\ii_\TB(y_\rho)=(1, 2, ..., r-2, r-1, r) \text{ and } \ii_\TB(y'_\rho)=(1, 2, ..., r-2, r, r+1).$$
We can check easily that the orbits $\mca{O}_y$ and $\mca{O}_{y'}$ for these two choices are $Y_{Q,n}$-free. Note that $\val{\wp_{Q,n}(\OF_{Q,n})} \ge 2$, since $\wp_{Q,n}(\mca{O}_y)\ne \wp_{Q,n}(\mca{O}_{y'})$.

\cu{Case II}, $r=2$. Consider $y$ and $y'$ such that $\ii_\TB(y_\rho)=(1, 2)$ and $\ii_\TB(y'_\rho)=(2, 3)$. For $m\ge 4$, $\mca{O}_y$ and $\mca{O}_{y'}$ are both $Y_{Q,n}$-free. Moreover, we can check that $\wp_{Q,n}(\OF_{Q,n,sc}) \subseteq \set{\wp_{Q,n}(\mca{O}_y), \wp_{Q,n}(\mca{O}_{y'})}$. Now if $m\ge 6$, then $\wp_{Q,n}(\mca{O}_y) \ne \wp_{Q,n}(\mca{O}_{y'})$. On the other hand, for $m=4$, one has $\wp_{Q,n}(\mca{O}_y) = \wp_{Q,n}(\mca{O}_{y'})$ and therefore $\dim\Wh(\Theta(\wt{\Spin}_{5}^{(8)}, \wchi))=1$ for any exceptional character $\wchi$ in this case.

To summarize for the case $m\ge r+2$:
$$
\begin{cases}
\dim\Wh(\Theta(\wt{\Spin}_{2r+1}^{(n)}, \wchi))=1 & \text{ if } m=4, r=2;\\
\dim\Wh(\Theta(\wt{\Spin}_{2r+1}^{(n)}, \wchi))\ge 2 & \text{ if } r\ge 4 \text{ and } m\ge r+2, \text{ or } r=2 \text{ and } m\ge 6.
\end{cases}
$$
\item[$\bullet$] If $n< 2r$ (i.e. $m<r$ and therefore $m\le r-2$), in this case, one can check easily $\wp_{Q,n}(\OF_{Q,n,\sct})=\emptyset$.
\item[$\bullet$] If $n=2r$ (note $n\ne 2(r+1)$), i.e. $r=m$, one also has $\wp_{Q,n}(\OF_{Q,n,\sct})=\emptyset$.
\end{enumerate}

\vskip 10pt

From the above discussion, we observe that for $r=2$, the result agrees with that for covering groups of type $\TC_2$, as expected. Therefore, we just summarize our result for covering $\wt{\Spin}_{2r+1}^{(n)}$ with $r\ge 3$ as follows.

\begin{thm} \label{T:B}
Consider Brylinski-Deligne covering $\wt{\Spin}_{2r+1}^{(n)}$ with $r\ge 3$. Let $\wchi$ be an exceptional character, then $\dim\Wh(\Theta(\wt{\Spin}_{2r+1}^{(n)}, \wchi))=1$ if and only if one of the following holds:
\begin{enumerate}
\item[$\bullet$] $n=2(r+1)$, and $\wchi$ is the unique exceptional character satisfying (\ref{E:B1}) or (\ref{E:B2});
\item[$\bullet$] $n=2r+1$, and $\wchi$ is the only exceptional character of $\wt{\Spin}_{2r+1}^{(2r+1)}$.
\end{enumerate} 
\end{thm}

%%%%%%%%%%%%%%%%%%%%%%%%%%%%%%%%
\vskip 15pt
\section{The $\TG_2$ case}
Consider $\TG_2$ with Dykin diagram for its simple coroots:

$$
\begin{picture}(5.7,0.2)(0,0)
\put(2.5,0){\circle{0.08}}
\put(3,0){\circle{0.08}}
\put(2.53,0.018){\line(1,0){0.44}}
\put(2.54,0){\line(1,0){0.42}}
\put(2.53,-0.018){\line(1,0){0.44}}
\put(2.74,-0.040){$<$}
\put(2.5,0.1){\footnotesize $\alpha_1^\vee$}
\put(3,0.1){\footnotesize $\alpha_2^\vee$}
\end{picture}
$$
\vskip 10pt

Let $Y=\langle \alpha_1^\vee, \alpha^\vee_2 \rangle$ be the cocharacter lattice of $\TG_2$, where $\alpha_1^\vee$ is the short coroot. Let $Q$ be the Weyl-invariant quadratic on $Y$ such such $Q(\alpha_1^\vee)=1$ (thus $Q(\alpha_2^\vee)=3$). Then the bilinear form $B_Q$ is given by
$$
B_Q(\alpha_i^\vee, \alpha_j^\vee) =
\begin{cases}
2 & \text{if } i=j=1; \\
-3& \text{if } i=1, j=2; \\
6 & \text{if } i=j=2.
\end{cases}
$$
A simple computation gives:
$$Y_{Q,n}=Y_{Q,n}^{\sct}=\Z (n_{\alpha_1}\alpha_1^\vee) \oplus \Z (n_{\alpha_2} \alpha_2^\vee),$$
where $n_{\alpha_2}=n/\text{gcd}(n, 3)$ and $n_{\alpha_1}=n$.

\vskip 5pt

The map $\ii_\TG: \bigoplus_{i=1}^2 \Z \alpha_i^\vee \to \bigoplus_{i=1}^3 \Z e_i$ is given by
$$\ii_\TG: (x_1, x_2) \mapsto (x_1-2x_2, x_2-x_1, x_2).$$
Any $(y_i)_i\in \bigoplus_{i=1}^3 \Z e_i$ lies in the image of $\ii_\TG$ if and only if $y_1 + y_2 + y_3=0$. 

The Weyl group $W=\langle \w_{\alpha_1}, \w_{\alpha_2} \rangle$ generated by $\w_{\alpha_1}$ and $\w_{\alpha_2}$ is the Dihedral group of order $12$. In particular, $\w_{\alpha_1}(y_1, y_2, y_3)=(y_2, y_1, y_3) \in \bigoplus_{i=1}^3 \Z e_i$, and $\w_{\alpha_2}(y_1, y_2, y_3)=(-y_1, -y_3, -y_2)$.

By using $\ii_\TG$, we could identify
$$
Y_{Q,n}=Y_{Q,n}^{\sct}=
\left\{
\begin{array}{cc}
(y_1, y_2, y_3) \in \bigoplus_{i=1}^{3} \Z e_i: \\
\bullet \ y_1 +y_2 + y_3 =0,\\
\bullet \ \  y_1\equiv y_2 \equiv y_3 \text{ mod } n.
\end{array}\right\}
$$
We have $ \rho=5\alpha_1^\vee + 3\alpha_2^\vee$ with $\ii_\TG(\rho) =(-1, -2, 3) \in \bigoplus_{i=1}^3 \Z e_i$.
It follows that for any $y=(x_1, x_2) \in \bigoplus_{i=1}^2 \Z \alpha_i^\vee$,
\begin{align*}
\ii_\TG(y_\rho) =(x_1-2x_2 -1, x_2-x_1 -2, x_2 + 3) \in \bigoplus_{i=1}^3 \Z e_i.
\end{align*}
We may write $\ii_\TG(y_\rho)= (x_1^*, x_2^*, x_3^*)$. In particular, $(x_1^*, x_2^*, x_3^*) \in \bigoplus_{i=1}^3 \Z e_i$ lies in the image of $\ii_\TG$ if and only if $x_1^* + x_2^* + x_3^*=0$.

Since $Y_{Q,n}=Y_{Q,n}^{\sct}$, it follows that $\dim \Wh(\Theta(\wt{\TG}_2^{(n)}, \wchi))=\val{\wp_{Q,n}(\OF_{Q,n}) }$, where $\wchi$ is the only exceptional character of $\wt{\TG}_2^{(n)}$ as $Z(\wt{\TG}_2^\vee)=1$. 

To determine the $n$ such that $\dim \Wh(\Theta(\wt{\TG}_2^{(n)}, \wchi))=1$, we only give an outline of the argument, the details of which consists of basic combinatorial computations:

\begin{enumerate}
\item[$\bullet$] For $n=7, 8$ or $n\ge 10$, the orbit $\mca{O}_y$ with $\ii_\TG(y_\rho)=(-2, -1, 3)$ is $Y_{Q,n}$-free.
\item[$\bullet$] For $n=8, 10, 11$ or $n\ge 13$, the orbit $\mca{O}_{y'}$ with $\ii_\TG(y'_\rho)=(-3, -1, 4)$ is $Y_{Q,n}$-free. Moreover, for $n=8, 10, 11$ or $n\ge 13$, one has $\wp_{Q,n}(\mca{O}_y) \ne \wp_{Q,n}(\mca{O}_{y'})$ for $\ii_\TG(y_\rho)=(-2, -1, 3)$ and $\ii_\TG(y'_\rho)=(-3, -1, 4)$.
\item[$\bullet$] If $\OF_{Q,n,\sct}\ne \emptyset$, then necessarily $|Y/Y_{Q,n}^{\sct}| \ge |W|$, i.e. $n\cdot n_{\alpha_2} \ge 12$. Thus $n\ge 4$.
\item[$\bullet$] One can also check by hand that $\OF_{Q,n,\sct}=\emptyset$ for  $n=4, 5, 6, 9$.
\item[$\bullet$] For $n=7, 12$, $\wp_{Q,n}(\OF_{Q,n})=\set{\wp_{Q,n}(\mca{O}_y)}$ with $\ii_\TG(y_\rho)=(-2, -1, 3)$. That is, we have $\dim \Wh(\Theta(\wt{\TG}_2^{(n)}, \wchi))=1$ for $n=7$ or $12$.
\end{enumerate}
\vskip 5pt

To summarize, 
\begin{thm} \label{T:G2}
Consider the Brylinski-Deligne covering $\wt{\TG}_2^{(n)}$. Let $\wchi$ be the only exceptional character on $\wt{\TG}_2^{(n)}$, then $\dim \Wh(\Theta(\wt{\TG}_2^{(n)}, \wchi))=1$ if and only if $n=7$ or $12$.
\end{thm}

%%%%%%%%%%%%%%%
%%%%%%%%%%%%%%%%%%%%%%%%%%%%%%%%%%%%%%%%%

\end{document}